\documentclass[letterpaper,10pt]{IEEEtran}  % Comment this line out
\usepackage{amsmath}
\usepackage{amsthm,amssymb,amsfonts}
\usepackage{xypic}
\usepackage{authblk}
\usepackage{color}
\usepackage{graphicx}
\usepackage{subfigure}
\usepackage{hyperref}
\usepackage{fancyhdr}
\hypersetup{colorlinks=true,linkcolor=black,citecolor=black}
\usepackage[ruled,vlined]{algorithm2e}

\newcommand{\oneplusmus}{\sqrt{\mu_s}}
\renewcommand{\root}{\operatorname{root}}
\newcommand{\norm}[1]{\left\lVert #1 \right\rVert}
\newcommand{\map}[3]{#1:#2 \rightarrow #3}
\newcommand{\setdef}[2]{\{#1 \;|\; #2\}}

\newcommand{\real}{{\mathbb{R}}}
\newcommand{\realpos}{\mathbb{R}_{>0}}
\newcommand{\naturals}{\mathbb{N}}
\newcommand{\Xhb}{X_{\operatorname{hb}}}
\newcommand{\Xhba}{X^a_{\operatorname{hb}}}
\newcommand{\step}{\operatorname{step}}
\newcommand{\dstep}{\operatorname{step}^{\operatorname{d}}}
\newcommand{\pstep}{\operatorname{step}^{\operatorname{p}}}
\newcommand{\hstep}{\frak{step}}

\newcommand{\hsd}{\mathfrak{b}^{\operatorname{d}}}
\newcommand{\hsp}{\mathfrak{b}^{\operatorname{p}}}
\newcommand{\ET}{\operatorname{ET}}
\newcommand{\ST}{\operatorname{ST}}
\newcommand{\pb}{{\operatorname{p}}}
\newcommand{\db}{{\operatorname{d}}}
\newcommand{\miet}{\operatorname{MIET}}

\newcommand{\oprocendsymbol}{\hbox{$\bullet$}}
\newcommand{\oprocend}{\relax\ifmmode\else\unskip\hfill\fi\oprocendsymbol}
\newcommand{\longthmtitle}[1]{\mbox{}{{\rm (#1).}}}
\newcommand{\feedback}{\mathfrak{k}}

\newtheorem{theorem}{Theorem}[section]
\newtheorem{proposition}[theorem]{Proposition}
\newtheorem{lemma}[theorem]{Lemma}

\newtheorem{remark}[theorem]{Remark}

\newcommand{\fA}{\mathfrak{A}}

\newcommand{\fD}{\mathfrak{D}}

\newcommand{\fC}{\mathfrak{C}}
\newcommand{\fB}{\mathfrak{B}}

\renewcommand{\baselinestretch}{.995}

\allowdisplaybreaks

\title{\LARGE \bf Resource-Aware Discretization of Accelerated
  Optimization Flows}

\author{Miguel Vaquero \quad Pol Mestres \quad Jorge
  Cort\'es % <-this % stops a space
  \thanks{This work was supported by NSF Award
    ECCS-1917177.}% <-this % stops a space
  \thanks{MV and JC are with the Department of Mechanical and
    Aerospace Engineering, University of California, San Diego,
    {\tt\small \{mivaquerovallina,cortes\}@ucsd.edu}. PM is with the
    Centro de Formacion Interdisciplinaria Superior, Universidad
    Polit\'ecnica de Catalu\~{n}a, {\tt\small polmestres4@gmail.com}}%
}

\begin{document}

\maketitle

\begin{abstract}
  This paper tackles the problem of discretizing accelerated
  optimization flows while retaining their convergence properties.
  Inspired by the success of resource-aware control in developing
  efficient closed-loop feedback implementations on digital systems,
  we view the last sampled state of the system as the resource to be
  aware of.  The resulting variable-stepsize discrete-time algorithms
  retain by design the desired decrease of the Lyapunov certificate of
  their continuous-time counterparts.  Our algorithm design employs
  various concepts and techniques from resource-aware control that, in
  the present context, have interesting parallelisms with the
  discrete-time implementation of optimization algorithms. These
  include derivative- and performance-based triggers to monitor the
  evolution of the Lyapunov function as a way of determining the
  algorithm stepsize, exploiting sampled information to enhance
  algorithm performance, and employing high-order holds using more
  accurate integrators of the original dynamics.  Throughout the
  paper, we illustrate our approach on a newly introduced
  continuous-time dynamics termed heavy-ball dynamics with displaced
  gradient, but the ideas proposed here have broad applicability to
  other globally asymptotically stable flows endowed with a Lyapunov
  certificate.
\end{abstract}

\section{Introduction}\label{sec:intro}

A recent body of research seeks to understand the acceleration
phenomena of first-order discrete optimization methods by means of
models that evolve in continuous time.  Roughly speaking, the idea is
to study the behavior of ordinary differential equations (ODEs) which
arise as continuous limits of discrete-time accelerated
algorithms. The basic premise is that the availability of the powerful
tools of the continuous realm, such as differential calculus, Lie
derivatives, and Lyapunov stability theory, can be then brought to
bear to analyze and explain the accelerated behavior of these flows,
providing insight into their discrete counterparts.  Fully closing the
circle to provide a complete description of the acceleration
phenomenon requires solving the outstanding open question of how to
discretize the continuous flows while retaining their accelerated
convergence properties.  However, the discretization of accelerated
flows has proven to be a challenging task, where retaining
acceleration seems to depend largely on the particular ODE and the
discretization method employed.  This paper develops a resource-aware
approach to the discretization of accelerated optimization flows.

% To fully close the circle requires to show how the heavy-ball and
% Nesterov's methods emerge as discretizations of the continuous-time
% dynamics. However, the discretization of accelerated flows has
% proven to be a challenging task, where standard discretization
% schemes may not retain the accelerated properties of the continuous
% flows~\cite{BS-SSD-MIJ-WJS:19}. In fact, retaining acceleration
% seems to depend largely on the particular ODE and the discretization
% method employed. Some explicit discretization schemes, like forward
% Euler, can even become numerically unstable after a few iterations,
% cf.~\cite{AW-ACW-MIJ:16}.

\subsubsection*{Literature Review}

% The continuous setting has also been used to build models that help in the
% understanding of discrete algorithms. One of the most salient examples is
% the vast literature devoted to explain the mystery of accelerated,
% first-order methods in optimization (Polyak's heavy-ball, Nesterov's
% accelerated gradient...). Roughly speaking, accelerated, first-order optimization
% algorithms are discrete dynamical systems able to cope with
% unconstrained optimization problems only using information of the
% value of the function and its gradient at certain points. In their
% important properties, these methods stand out by their nice scaling
% properties (complexity almost independent of the dimension) and
% optimal convergence-rate, which made them very popular among the
% machine learning community. This success manifests the importance of
% understanding the most relevant principles underlying
% acceleration. The comprehension of accelerated methods will permit
% both to advance the current state-of-the-art optimization methods in
% the mentioned fields, and to export similar ideas to related areas,
% like control theory.

The acceleration phenomenon goes back to the seminal
paper~\cite{BTP:64} introducing the so-called heavy-ball method, which
employed momentum terms to speed up the convergence of the classical
gradient descent method. % The momentum term, if tuned properly,
% prevents the appearance of oscillations commonly encountered in
% gradient descent, enforcing to move along a direction that takes
% into account states.
The heavy-ball method achieves optimal convergence rate in a
neighborhood of the minimizer for arbitrary convex functions and
global optimal convergence rate for quadratic objective functions.
Later on, the work~\cite{YEN:83} proposed the Nesterov's accelerated
gradient method and, employing the technique of estimating sequences,
showed that it converges globally with optimal convergence rate for
convex and strongly-convex smooth functions.  The algebraic nature of
the technique of estimating sequences does not fully explain the
mechanisms behind the acceleration phenomenon, and this has motivated
many approaches in the literature to provide fundamental understanding
and insights. These include coupling dynamics~\cite{ZAZ-LO:17},
dissipativity theory~\cite{BH-LL:17}, integral quadratic
constraints~\cite{LL-BR-AP:16,BVS-RAF-KML:18}, and geometric
arguments~\cite{SB-YTL-MS:15}.

Of specific relevance to this paper is a recent line of research
initiated by~\cite{WS-SB-EJC:16} that seeks to understand the
acceleration phenomenon in first-order optimization methods by means
of models that evolve in continuous time.~\cite{WS-SB-EJC:16}
introduced a second-order ODE as the continuous limit of Nesterov's
accelerated gradient method and characterized its accelerated
convergence properties using Lyapunov stability analysis.  The ODE
approach to acceleration now includes the use of Hamiltonian dynamical
systems~\cite{MB-MJ-AW:18,CJM-DP-YWT-BO-AD:18}, inertial systems with
Hessian-driven damping~\cite{HA-ZC-JF-HR:19}, and high-resolution
ODEs~\cite{BS-SSD-MIJ-WJS:18-arxiv,BS-JG-SK:20}.  This body of
research is also reminiscient of the classical dynamical systems
approach to algorithms in optimization, see~\cite{RWB:91,UH-JBM:94}.
The question of how to discretize the continuous flows while
maintaining their accelerated convergence rates has also attracted
significant attention, motivated by the ultimate goal of fully
understanding the acceleration phenomenon and taking advantage of it
to design better optimization algorithms.  Interestingly,
discretizations of these ODEs do not necessarily lead to
acceleration~\cite{BS-SSD-MIJ-WJS:19-arxiv}.  In fact, explicit
discretization schemes, like forward Euler, can even become
numerically unstable after a few iterations~\cite{AW-ACW-MIJ:16}.
Most of the discretization approaches found in the literature are
based on the study of well-known integrators, including symplectic
integrators~\cite{MB-MJ-AW:18,AW-LM-AW:19}, Runge-Kutta
integrators~\cite{JZ-AM-SS-AJ:18} or modifications of Nesterov's
\emph{three sequences}~\cite{AW-ACW-MIJ:16,AW-LM-AW:19,ACW-BR-MIJ:18}.
Our previous work~\cite{MV-JC:19-nips} instead developed a
variable-stepsize discretization using zero-order holds and
state-triggers based on the derivative of the Lyapunov function of the
original continuous flow. Here, we provide a comprehensive approach
based on powerful tools from resource-aware control, including
performance-based triggering and state holds that more effectively use
sampled information.  Other recent approaches to the acceleration
phenomena and the synthesis of optimization algorithms using
control-theoretic notions and techniques include~\cite{ASK-PME-TK:18},
which employs hybrid systems to design a continuous-time dynamics with
a feedback regulator of the viscosity of the heavy-ball ODE to
guarantee arbitrarily fast exponential convergence,
and~\cite{DH-RGS:19}, which introduced an algorithm which alternates
between two (one fast when far from the minimizer but unstable, and
another slower but stable around the minimizer) continuous heavy-ball
dynamics.

\subsubsection*{Statement of Contributions}

This paper develops a resource-aware control framework to the
discretization of accelerated optimization flows that fully exploits
their dynamical properties.  Our approach relies on the key
observation that resource-aware control provides a principled way to
go from continuous-time control design to real-time implementation
with stability and performance guarantees by opportunistically
prescribing when certain resource should be employed.  In our
treatment, the resource to be aware of is the last sampled state of
the system, and hence what we seek to maximize is the stepsize of the
resulting discrete-time algorithm.
% The resulting strategies are variable-stepsize algorithms that
% preserve by design the convergence properties of the original
% dynamics.
Our first contribution is the introduction of a second-order
differential equation which we term heavy-ball dynamics with displaced
gradient. This dynamics generalizes the continuous-time heavy-ball
dynamics analyzed in the literature by evaluating the gradient of the
objective function taking into account the second-order nature of the
flow.  We establish that the proposed dynamics retains the same
convergence properties as the original one while providing additional
flexibility in the form of a design parameter.

Our second contribution uses trigger design concepts from
resource-aware control to synthesize criteria that determine the
variable stepsize of the discrete-time implementation of the
heavy-ball dynamics with displaced gradient.  We refer to these
criteria as event- or self-triggered, depending on whether the
stepsize is implicitly or explicitly defined.  We employ derivative-
and performance-based triggering to ensure the algorithm retains the
desired decrease of the Lyapunov function of the continuous flow. In
doing so, we face the challenge that the evaluation of this function
requires knowledge of the unknown optimizer of the objective
function. To circumvect this hurdle, we derive bounds on the evolution
of the Lyapunov function that can be evaluated without knowledge of
the optimizer.  We characterize the convergence properties of the
resulting discrete-time algorithms, establishing the existence of a
minimum inter-event time and performance guarantees with regards to
the objective function. 

Our last two contributions provide ways of exploiting the sampled
information to enhance the algorithm performance. Our third
contribution provides an adaptive implementation of the algorithms
that adaptively adjusts the value of the gradient displacement
parameter depending on the region of the space to which the state
belongs.  Our fourth and last contribution builds on the fact that the
continuous-time heavy-ball dynamics can be decomposed as the sum of a
second-order linear dynamics with a nonlinear forcing term
corresponding to the gradient of the objective function. Building on
this observation, we provide a more accurate hold for the
resource-aware implementation by using the samples only on the
nonlinear term, and integrating exactly the resulting linear system
with constant forcing.  We establish the existence of a minimum
inter-event time and characterize the performance with regards to the
objective function of the resulting high-order-hold algorithm.
Finally, we illustrate the proposed optimization algorithms in
simulation, comparing them against the heavy-ball and Nesterov's
accelerated gradient methods and showing superior performance to other
discretization methods proposed in the literature.

\section{Preliminaries}\label{sec:preliminaries}
This section presents basic notation and preliminaries.

\subsection{Notation}\label{assumptions}
We denote by $\real$ and $\realpos$ the sets of real and positive
real numbers, resp.  All vectors are column vectors.  We denote their
scalar product by $\langle \cdot,\cdot\rangle$.  We use $\norm{\cdot{
  }}$ to denote the $2$-norm in Euclidean space. Given $\mu \in
\realpos$, 
% a function $f:\real^n\rightarrow \real$ is convex if $f(kx
% + (1-k)y)\leq kf(x) + (1 - k)f(y)$ for $x$, $y \in \real^n$ and $k\in
% [0,1]$. 
a continuously differentiable function $f$ is $\mu$-strongly convex if
$f(y) - f(x) \geq\langle \nabla f(x), y - x\rangle +
\frac{\mu}{2}\norm{x - y}^2$ for $x$, $y \in \real^n$.  Given $L \in
\realpos$ and a function $f:X \rightarrow Y$ between two normed spaces
$(X,\norm{\cdot{}}_X)$ and ($Y,\norm{\cdot{}}_Y$), $f$ is
$L$-Lipschitz if $\norm{f(x) - f(x')}_{Y}\leq L\norm{x - x'}_{X}$ for
$x$, $x' \in X$. 
% We endow the space of $\real^{n\times m}$ matrices with the
% induced matrix norm, namely $\norm{A}=\max_{\norm{x}=1}\norm{Ax}$.
The functions we consider here are continuously differentiable,
$\mu$-strongly convex and have $L$-Lipschitz continuous gradient.  We
refer to the set of functions with all these properties
by~$\mathcal{S}_{\mu,L}^1(\real^n)$.  A function
$f:\real^n\rightarrow\real$ is positive definite relative to $x_*$ if
$f(x_*)=0$ and $f(x)>0$ for $x\in\real^n\setminus\{x_*\}$.

\subsection{Resource-Aware Control}\label{sec:resource-aware}

Our work builds on ideas from resource-aware control to develop
discretizations of continuous-time accelerated flows. Here, we provide
a brief exposition of its basic elements and refer
to~\cite{WPMHH-KHJ-PT:12,CN-EG-JC:19-auto} for further details.

Given a controlled dynamical system $\dot{p} = X(p,u)$, with $p \in
\real^n$ and $u \in \real^m$, assume we are given a stabilizing
continuous state-feedback $\map{\feedback}{\real^n}{\real^m}$ so that the
closed-loop system $\dot{p} = X(p,\feedback(p))$ has $p_*$ as a globally
asymptotically stable equilibrium point. Assume also that a Lyapunov
function $\map{V}{\real^n}{\real}$ is available as a certificate of
the globally stabilizing nature of the controller. Here, we assume
this takes the form
\begin{equation}\label{eq:lyapunov_decay}
  \dot{V}  =  \langle \nabla V(p) X(p,\feedback(p)) \rangle  \leq - \frac{\sqrt{\mu}}{4} V (p), 
  % \Leftrightarrow \dot{V} + \alpha V \leq 0.
\end{equation}
for all $p \in \real^n$.  Although exponential decay of $V$ along the
system trajectories is not necessary, we restrict our attention to
this case as it arises naturally in our treatment.

Suppose we are given the task of implementing the controller signal
over a digital platform, meaning that the actuator cannot be
continuously updated as prescribed by the specification
$u=\feedback(p)$. In such case, one is forced to discretize the
control action along the execution of the dynamics, while making sure
that stability is still preserved.  A simple-to-implement approach is
to update the control action \emph{periodically}, i.e., fix $h >0$,
sample the state as $\{p(kh)\}_{k=0}^\infty $ and implement
\begin{align*}
  \dot{p}(t) = X(p(t),\feedback(p(kh))) , \quad t \in [kh,(k +1)h].
\end{align*}
This approach requires $h$ to be small enough to ensure that $V$
remains a Lyapunov function and, consequently, the system remains
stable.  By contrast, in \emph{resource-aware control}, one employs
the information generated by the system along its trajectory to update
the control action in an opportunistic fashion.  Specifically, we seek
to determine in a state-dependent fashion a sequence of times
$\{t_k\}_{k=0}^\infty$, not necessarily uniformly spaced, such that
$p_*$ remains a globally asymptotically stable equilibrium for the
system
\begin{align}\label{eq:triggered-design}
  \dot{p}(t) &= X(p(t),\feedback(p(t_k))) , \quad t \in
  [t_k,t_{k+1}] .
\end{align}
The main idea to accomplish this is to let the state sampling be
guided by the principle of maintaining the same type of exponential
decay~\eqref{eq:lyapunov_decay} along the new dynamics.  To do this,
one defines triggers to ensure that this decay is never violated by
prescribing a new state sampling. Formally, one sets $t_0 = 0$ and $
t_{k + 1} = t_k + \step (p(t_k))$, where the stepsize is defined by
\begin{align}\label{eq:step-generic}
  \step (\hat{p}) &= \min \setdef{t > 0}{b(\hat {p},t) = 0}.
\end{align}
We refer to the criteria as \emph{event-triggering} or
\emph{self-triggering} depending on whether the evaluation of the
function $b$ requires monitoring of the state $p$ along the trajectory
of~\eqref{eq:triggered-design} (ET) or just knowledge of its initial
condition $\hat{p}$ (ST). The more stringent requirements to implement
event-triggering lead to larger stepsizes versus the more conservative
ones characteristic of self-triggering.  In order for the state
sampling to be implementable in practice, the inter-event times
$\{t_{k + 1} - t_k\}_{k=0}^\infty$ must be uniformly lower bounded by
a positive minimum inter-event time, abbreviated MIET. In particular,
the existence of a MIET rules out the existence of Zeno behavior,
i.e., the possibility of an infinite number of triggers in a finite
amount of time.

Depending on how the evolution of the function $V$ is examined, we
describe two types of triggering conditions\footnote{In both cases,
  for a given $z \in \real^n$, we let $p(t;\hat{p})$ denote the
  solution of $\dot{p}(t) = X(p(t),\feedback(\hat{p}))$ with initial condition
  $p(0) = \hat{p}$.}:
\begin{LaTeXdescription}
\item[Derivative-based trigger:] In this case, $b^\db$ is defined as
  an upper bound of the expression $\frac{d}{dt} V(p(t;\hat{p})) +
  \frac{\sqrt{\mu}}{4} V(p(t;\hat{p}))$.  This definition ensures
  that~\eqref{eq:lyapunov_decay} is maintained
  along~\eqref{eq:triggered-design};

\item[Performance-based trigger:] In this case, $b^\pb$ is defined as an
  upper bound of the expression $V(p(t;\hat{p})) - e^{-\frac{\sqrt{\mu}}{4}
    t}V(\hat{p})$. Note that this definition ensures that the integral
  version of~\eqref{eq:lyapunov_decay} is maintained
  along~\eqref{eq:triggered-design}.
\end{LaTeXdescription}
In general, the performance-based trigger gives rise to stepsizes that
are at least as large as the ones determined by the derivative-based
approach, cf.~\cite{PO-JC:18-cdc}. This is because the latter
prescribes an update as soon as the exponential decay is about to be
violated, and therefore, does not take into account the fact that the
Lyapunov function might have been decreasing at a faster rate since
the last update.  Instead, the performance-based approach reasons over
the \emph{accumulated decay} of the Lyapunov function since the last
update, potentially yielding longer inter-sampling times.

% In the reference~\cite{MV-JC:19-nips} the authors implemented a
% discretization of continuous accelerated optimization flows based on
% the derivative-based triggering conditions. In the next section we
% will present a more powerful approximation based on the
% performance-based design.

A final point worth mentioning is that, in the event-triggered control
literature, the notion of \emph{resource} to be aware of can be many
different things, beyond the actuator described above, including the
sensor, sensor-controller communication, communication with other
agents, etc. This richness opens the way to explore more elaborate
uses of the sampled information beyond the zero-order hold
in~\eqref{eq:triggered-design}, something that we also leverage later
in our presentation.

\section{Problem Statement}\label{se:problem-statement}

Our motivation here is to show that principled approaches to
discretization can retain the accelerated convergence properties of
continuous-time dynamics, fill the gap between the continuous and
discrete viewpoints on optimization algorithms, and lead to the
construction of new ones. Throughout the paper, we focus on the
continuous-time version of the celebrated heavy-ball
method~\cite{BTP:64}.  Let $f$ be a function in
$\mathcal{S}_{\mu,L}^1(\real^n)$ and let $x_*$ be its unique
minimizer.  The heavy-ball method is known to have an optimal
convergence rate in a neighborhood of the minimizer. For its
continuous-time counterpart, consider the following $s$-dependent
family of second-order differential equations, with $s \in \realpos$,
proposed in~\cite{BS-SSD-MIJ-WJS:18-arxiv},
\begin{subequations}\label{eq:continuous-hb-dynamics}
  \begin{align}
    \begin{bmatrix}
      \dot{x}  \\
      \dot{v}
    \end{bmatrix}
    & =
    \begin{bmatrix}
      v
      \\
      - 2\sqrt{\mu}v - (1+\sqrt{\mu s})\nabla f(x))
    \end{bmatrix},
    \\
    x(0) & =x_0, \quad  v(0)=-\frac{2\sqrt{s}\nabla
      f(x_0)}{1+\sqrt{\mu s}} .
    \label{eq:initial-state}
    % \\
    % \begin{bmatrix}
    %   x(0)
    %   \\
    %   v (0)
    % \end{bmatrix}
    % & = 
    % \begin{bmatrix}
    %   0
    %   \\
    %   -\frac{2\sqrt{s}\nabla f(x_0)}{1+\sqrt{\mu s}}
    % \end{bmatrix} .
  \end{align}
\end{subequations}
We refer to this dynamics as~$\Xhb$. The following result
characterizes the convergence properties
of~\eqref{eq:continuous-hb-dynamics} to $p_*=[x_*, 0]^T$.

\begin{theorem}[\cite{BS-SSD-MIJ-WJS:18-arxiv}]\label{th:hb}
  Let $\map{V}{\real^n \times \real^n}{\real}$ be
  \begin{align}\label{eq:continuous-hb-lyapunov}
    V(x,v) & = (1+\sqrt{\mu s})(f(x) - f(x_*)) +
    \displaystyle\frac{1}{4}\norm{v}^2 \nonumber
    \\
    & \quad + \displaystyle\frac{1}{4}\norm{v +2\sqrt{\mu}(x-x_*)}^2,
  \end{align}
  which is positive definite relative to $[x_*, 0]^T$. Then $\dot{V}
  \leq-\frac{\sqrt{\mu}}{4}V$ along the
  dynamics~\eqref{eq:continuous-hb-dynamics} and, as a consequence,
  $p_*=[x_*, 0]^T$ is globally asymptotically stable.  Moreover, for
  $s\leq 1/L$, the exponential decrease of $V$ implies
  \begin{equation}\label{eq:decay_fun}
    f(x(t))-f(x_*)\leq
    \frac{7\norm{x(0) -
        x_*}^2}{2s}e^{-\frac{\sqrt{\mu}}{4}t} .
  \end{equation}
\end{theorem}

This result, along with analogous
results~\cite{BS-SSD-MIJ-WJS:18-arxiv} for the Nesterov's accelerated
gradient descent, serves as an inspiration to build Lyapunov
functions that help to explain the accelerated convergence rate of the
discrete-time methods.

% Nonetheless, a final step is necessary: to discretize the continuous
% dynamics in proper way. This necessity comes mainly from
% \begin{enumerate}
% \item The continuous models do not exactly match the discrete
%   algorithms. Showing how the mentioned algorithms emerge as
%   discretization of continuous dynamics would completely fill the gap
%   and provide a full description of the acceleration phenomenon.

% \item The development of continuous optimization flows is usually much
%   easier that proceeding directly in the discrete setting. Nonetheless
%   straightforward discretizations are unable to maintain the
%   accelerated features. Discretizations conserving convergence rate
%   will allow the construction of new optimization algorithms starting
%   from continuous dynamics.
% \end{enumerate}

Inspired by the success of resource-aware control in developing
efficient closed-loop feedback implementations on digital systems,
here we present a discretization approach to accelerated optimization
flows using resource-aware control.  At the basis of the approach
taken here is the observation that the convergence
rate~\eqref{eq:decay_fun} of the continuous flow is a direct
consequence of the Lyapunov nature of the
function~\eqref{eq:continuous-hb-lyapunov}. In fact, the integration
of $\dot{V} \leq-\frac{\sqrt{\mu}}{4}V$ along the system trajectories
yields
\begin{equation*}
  V(x(t),v(t)) \leq
  e^{-\frac{\sqrt{\mu}}{4}t} V(x(0),v(0)).
\end{equation*} 
Since $ f(x(t)) - f(x_*) \leq V(x(t),v(t))$, we deduce
\begin{align*}
  f(x(t)) - f(x_*) \leq e^{-\frac{\sqrt{\mu}}{4}t} V(x(0),v(0)) =
  \mathcal{O}(e^{-\frac{\sqrt{\mu}}{4}t}) .
\end{align*}
The characterization of the convergence rate via the decay of the
Lyapunov function is indeed common among accelerated optimization
flows.  This observation motivates the resource-aware approach to
discretization pursued here, where the resource that we aim to use
efficiently is the sampling of the state itself. By doing so, the
ultimate goal is to give rise to large stepsizes that take maximum
advantage of the decay of the Lyapunov function (and consequently of
the accelerated nature) of the continuous-time dynamics in the
resulting discrete-time implementation.

% Therefore, one should construct discrete dynamics which aim at
% conserving the same Lyapunov decay as the continuous ones. This
% rapidly brings to the attention the existence of techniques in control
% theory designed for similar endeavors, the event-triggered approach
% which serves as a motivation for our constructions. More precisely,
% the performance-based approach precisely imposes the Lyapunov decay in
% \eqref{motivation_decay}. Strengthening the link with the discrete
% setting, taking a step of length $\Delta_k$, we get
% \begin{align*}
%   V(x(t_{k+1}),v(t_{k+1})))& \leq e^{-\frac{\sqrt{\mu}}{4}\Delta_k}
%   V(x(t_k),v(t_k))
%   \\
%   & \leq e^{-\frac{\sqrt{\mu}}{4} \sum_{i=0}^k \Delta_i} V(x(0),v(0)).
% \end{align*}
% If we assume now a lower bound of the stepsizes, say $\Delta$, then
% using the last expression we obtain the discrete analogue of
% \eqref{motivation_decay2} (using the notation $x(t_k) = x_k$)
% \[
% f(x_k) - f(x_*) =
% \mathcal{O}(e^{-\frac{\sqrt{\mu}}{4}k \Delta}).
% \]
% This argument shows how Lyapunov decay and stepsize can be combined to
% give a lower bound of the convergence rate.

\section{Resource-Aware Discretization of Accelerated Optimization
  Flows}\label{sec:performance-based}

In this section we propose a discretization of accelerated
optimization flows using state-dependent triggering and analyze the
properties of the resulting discrete-time algorithm.  For convenience,
we use the shorthand notation $p = [x,v]^T$.  In following with the
exposition in Section~\ref{sec:resource-aware}, we start by
considering the zero-order hold implementation $\dot p = \Xhb
(\hat{p})$, $p(0) = \hat{p}$ of the heavy-ball
dynamics~\eqref{eq:continuous-hb-dynamics},
\begin{subequations}\label{eq:forward-euler}
  \begin{align}
    \dot{x} & = \hat{v},
    \\
    \dot{v} & = -2\sqrt{\mu}\hat{v} - (1 + \sqrt{\mu s}) \nabla
    f(\hat{x}).
  \end{align}  
\end{subequations}
Note that the solution trajectory takes the form $p(t) = \hat{p} + t
\Xhb (\hat p)$, which in discrete-time terminology corresponds to a
forward-Euler discretization
of~\eqref{eq:continuous-hb-dynamics}. Component-wise, we have
\begin{align*}
  x(t) & = \hat{x} + t\hat{v},
  \\
  v(t) & =\hat{v} -t \big( 2\sqrt{\mu}\hat{v} + (1 + \sqrt{\mu s})
  \nabla f(\hat{x}) \big).
\end{align*}

As we pointed out in Section~\ref{sec:resource-aware}, the use of
sampled information opens the way to more elaborated constructions than
the zero-order hold in~\eqref{eq:forward-euler}. As an example, given
the second-order nature of the heavy-ball dynamics, it would seem
reasonable to leverage the (position, velocity) nature of the pair
$(\hat{x},\hat{v})$ (meaning that, at position $\hat{x}$, the system
is moving with velocity $\hat{v}$) by employing the modified zero-order
hold
\begin{subequations}\label{eq:forward-euler-a}
  \begin{align}
    \dot{x} & = \hat{v},
    \\
    \dot{v} & = -2\sqrt{\mu}\hat{v} - (1 + \sqrt{\mu s}) \nabla
    f(\hat{x} + a \hat{v}),
  \end{align}  
\end{subequations}
where $a \ge 0$. Note that the trajectory
of~\eqref{eq:forward-euler-a} corresponds to the forward-Euler
discretization of the continuous-time dynamics
\begin{align}\label{eq:continuous-hb-dynamics_a}
  \begin{bmatrix}
    \dot{x}  \\
    \dot{v}
  \end{bmatrix}
  & =
  \begin{bmatrix}
    v
    \\
    - 2\sqrt{\mu}v - (1+\sqrt{\mu s})\nabla f(x + av))
  \end{bmatrix},
\end{align}
We refer to this as the {\it heavy-ball dynamics with displaced
  gradient} and denote it by~$\Xhba$ (note
that~\eqref{eq:forward-euler-a}
and~\eqref{eq:continuous-hb-dynamics_a} with $a=0$
recover~\eqref{eq:forward-euler}
and~\eqref{eq:continuous-hb-dynamics}, respectively). In order to
pursue the resource-aware approach laid out in
Section~\ref{sec:resource-aware} with the modified zero-order hold
in~\eqref{eq:forward-euler-a}, we need to characterize the asymptotic
convergence properties of the heavy-ball dynamics with displaced
gradient, which we tackle next.

\begin{remark}\longthmtitle{Connection between the use of sampled
    information and high-resolution-ODEs} {\rm A number of
    works~\cite{ML-AMD:19,MM-MIJ:19,IS-JM-GD-GH:13} have explored
    formulations of Nesterov's accelerated that employ
    displaced-gradient-like terms similar to the one used above. Here,
    we make this connection explicit. Given Nesterov's algorithm
    % \begin{subequations}\label{nesterov}
    \begin{align*}
      y_{k+1} & = x_k -s \nabla f(x_k)
      \\
      x_{k+1} &= y_{k+1} + \displaystyle\frac{1-\sqrt{\mu s}}{1 +
        \sqrt{\mu s}}(y_{k+1} - y_k)
    \end{align*}
    % \end{subequations}
    % is re-written in \cite{BS-SSD-MIJ-WJS:18-arxiv} as
    % \begin{align}
    %   x_{k+1} &= x_k + \displaystyle\frac{1-\sqrt{\mu s}}{1 +
    %   \sqrt{\mu
    %   s}}(x_k - x_{k-1}) \nonumber
    %   \\
    %   & -s\nabla f(x_k) - \displaystyle\frac{1-\sqrt{\mu s}}{1 +
    %   \sqrt{\mu s}}s(\nabla f(x_k) - \nabla f(x_{k
    %   -1})). \label{x-variable}
    % \end{align}
    the work~\cite{BS-SSD-MIJ-WJS:18-arxiv} obtains the following
    limiting high-resolution ODE
    % that after considering the algorithm \eqref{x-variable} the discrete
    % counterpart of a continuous dynamics, using Taylor's expansion and
    % disregarding high-order terms ($\mathcal{O}(s)$) one achieves the
    % expression of the high-resolution ODE
    \begin{equation}\label{hr-ODE}
      \ddot{x} + 2\sqrt{\mu}\dot{x} +
      \sqrt{s}\nabla^2 f(x)\dot{x} + (1+\sqrt{\mu s})\nabla f(x) = 0.
    \end{equation}
    Interestingly, considering instead the evolution of the $y$-variable
    % One may proceed in an alternative way and consider isolating the
    % variables $y$ in \eqref{nesterov}, instead of $x$ to obtain
  % \begin{align}
  %   y_{k+1} &= y_k + \displaystyle\frac{1-\sqrt{\mu s}}{1 + \sqrt{\mu
  %       s}}(y_k - y_{k-1}) \nonumber
  %   \\
  %   & -s\nabla f(y_k + \displaystyle\frac{1-\sqrt{\mu s}}{1 +
  %     \sqrt{\mu s}} (y_k - y_{k - 1}))\label{y-variable}.
  % \end{align}
  % Starting from \eqref{y-variable} 
  and applying similar arguments to the ones
  in~\cite{BS-SSD-MIJ-WJS:18-arxiv}, one instead obtains
  \begin{equation}\label{hr-ODE_new}
    \ddot{y} + 2\sqrt{\mu}\dot{y} + (1+\sqrt{\mu s})\nabla f \big(y +
    \frac{\sqrt{s}}{1 +\sqrt{\mu s}}\dot{y}\big) = 0 , 
  \end{equation}
  which corresponds to the continuous heavy-ball dynamics
  in~\eqref{eq:continuous-hb-dynamics} evaluated with a displaced
  gradient, i.e.,~\eqref{eq:continuous-hb-dynamics_a}. Even further,
  if we Taylor expand the last term in~\eqref{hr-ODE_new} as
  \[
  \nabla f(y + \displaystyle\frac{\sqrt{s}}{1 +\sqrt{\mu s}}\dot{y}) =
  \nabla f(y) + \nabla^2 f(y) \displaystyle\frac{\sqrt{s}}{1
    +\sqrt{\mu s}}\dot{y} + \mathcal{O}(s)
  \]
  and disregard the $\mathcal{O}(s)$ term, we recover~\eqref{hr-ODE}.
  % Strengthening the connection with high-resolution ODEs, observe
  % that in \eqref{hr-ODE_new} the last term, $(1+\sqrt{\mu s})\nabla
  % f(y + \displaystyle\frac{\sqrt{s}}{1 +\sqrt{\mu s}}\dot{y})$ could
  % be of order higher than $\mathcal{O}(\sqrt{s})$, as it is
  % evaluated by the gradient of $f$. But if we use Taylor's expansion
  % \[
  % \nabla f(y + \displaystyle\frac{\sqrt{s}}{1 +\sqrt{\mu s}}\dot{y}) =
  % \nabla f(y) + \nabla^2 f(y) \displaystyle\frac{\sqrt{s}}{1
  % +\sqrt{\mu s}}\dot{y} + \mathcal{O}(s)
  % \]
  % then \eqref{hr-ODE_new} becomes
  % \[
  % \ddot{y} + 2\sqrt{\mu}\dot{y} + (1+\sqrt{\mu s})( \nabla f(y) +
  % \nabla^2 f(y) \displaystyle\frac{\sqrt{s}}{1 +\sqrt{\mu s}}\dot{y} +
  % \mathcal{O}(s)) = 0 \label{new-hr-ODE}
  % \]
  % and disregarding the terms
  % $\mathcal{O}(s)$ we achieve
  % \[
  % \ddot{y} + 2\sqrt{\mu}\dot{y}+ \sqrt{s}\nabla^2 f(y)\dot{y} +
  % (1+\sqrt{\mu s}) \nabla f(y) = 0
  % \]
  % which exactly matches~\eqref{hr-ODE}. 
  This shows that~\eqref{hr-ODE_new} is just~\eqref{hr-ODE} with extra
  higher-order terms in $s$, and provides evidence of the role of
  gradient displacement in enlarging the modeling capabilities of
  high-resolution ODEs.}  \oprocend
\end{remark}

\subsection{Asymptotic Convergence of Heavy-Ball Dynamics with
  Displaced Gradient}\label{sec:hbdg-asymptotic}

In this section, we study the asymptotic convergence of heavy-ball
dynamics with displaced gradient.  Interestingly, for $a$ sufficiently
small, this dynamics enjoys the same convergence properties as the
dynamics~\eqref{eq:continuous-hb-dynamics}, as the following result
shows.

\begin{theorem}\longthmtitle{Global asymptotic stability  of
    heavy-ball dynamics with
    displaced gradient}\label{th:continuous-sampled}
  Let $\beta_1,\dots,\beta_4>0$ be
  \begin{align*}
    \beta_1 &= \oneplusmus\mu, \quad \beta_2 =
    \displaystyle\frac{\oneplusmus L}{\sqrt{\mu}},
    \\
    \beta_3 &= \frac{13\sqrt{\mu}}{16}, \quad \beta_4 = \frac{4
      \mu^{2}\sqrt{s} + 3 L \sqrt{\mu} \oneplusmus }{8 L^2},
    % \beta_3 &= -(\Sfrac{\sqrt{\mu}}{4}\frac{3}{4} - \sqrt{\mu}),
    % \\
    % \beta_4 &= -\big((1 + \sqrt{\mu s})\frac{\frac{\sqrt{\mu}}{4}-
    % \sqrt{ \mu}}{2 L} \notag
    % \\
    % & \quad + (\frac{\sqrt{\mu}}{4} 2\mu-\frac{\sqrt{\mu}(1 +
    % \sqrt{\mu s})\mu)}{2})\frac{1}{L^2} \big).
  \end{align*}
  where, for brevity, $\oneplusmus = 1 + \sqrt{\mu s}$,
  and define
  \begin{align}\label{eq:a1}
    a^*_1 = \frac{2}{\beta_2^2} \Big( \beta_1 \beta_4 +
    \sqrt{\beta_2^2 \beta_3 \beta_4 + \beta_1^2 \beta_4^2} \Big) .
  \end{align}
  Then, for $0 \leq a \leq a^*_1$, $\dot{V}
  \leq-\frac{\sqrt{\mu}}{4}V$ along the
  dynamics~\eqref{eq:continuous-hb-dynamics_a} and, as a consequence,
  $p_*=[x_*, 0]^T$ is globally asymptotically stable. Moreover, for
  $s\leq 1/L$, the exponential decrease of $V$
  implies~\eqref{eq:decay_fun} holds along the trajectories
  of~$\Xhba$.
\end{theorem}
\begin{proof}%[Proof of Theorem~\ref{th:continuous-sampled}]
  Note that
  \begin{align*}
    & \langle \nabla V(p), X^a_{\operatorname{hb}}(p) \rangle +
    \frac{\sqrt{\mu}}{4} V(p) =
    \\
    & = (1 \!+\! \sqrt{\mu s})\langle \nabla f(x),v \rangle \!-\!
    \sqrt{\mu}\norm{v}^2 \!-\! \oneplusmus\langle \nabla f(x \!+\! a
    v),v \rangle
    \\
    & \quad - \sqrt{\mu}\oneplusmus\langle \nabla f(x + av),x - x_*
    \rangle + \frac{\sqrt{\mu}}{4} V(x,v)
    \\
%OLD STUFF
    % & = (1 + \sqrt{\mu s})\langle \nabla f(x),v \rangle -
    % \sqrt{\mu}\norm{v}^2 + \frac{\sqrt{\mu}}{4}
    % V(x,v)
    % \\
    % & \quad - (1+\sqrt{\mu s})\langle \nabla f(x + a v) -
    % \nabla f(x) + \nabla f(x),v \rangle
    % \\
    % & \quad - \sqrt{\mu}(1+\sqrt{\mu s})\langle \nabla f(x +
    % av) - \nabla f(x) + \nabla f(x),x - x_*
    % \rangle
    % \\
%END OLD STUFF
    & = \underbrace{-\sqrt{\mu}\norm{v}^2 - \sqrt{\mu}\oneplusmus\langle \nabla f(x), x - x_* \rangle +\frac{\sqrt{\mu}}{4}
      V(x,v)}_{\textrm{Term~I}}
    \\
    & \quad \underbrace{-\oneplusmus\langle \nabla f(x + av) -
      \nabla f(x),v \rangle}_{\textrm{Term~II}}
    \\
    & \quad \underbrace{ -\sqrt{\mu}\oneplusmus\langle \nabla
      f(x + av) - \nabla f(x),x - x_*\rangle}_{\textrm{Term~III}},
  \end{align*}
  where in the second equality, we have added and subtracted
  $\sqrt{\mu} \oneplusmus\langle \nabla f(x ),x - x_*
  \rangle$. Observe that ``Term~I'' corresponds to $\langle \nabla
  V(p),X_{\operatorname{hb}}(p)\rangle + \frac{\sqrt{\mu}}{4}V(p)$ and
  is therefore negative by
  Theorem~\ref{th:hb}. From~\cite{MV-JC:19-nips}, this term can be
  bounded as
  \begin{align*}
    \textrm{Term~I} & \leq \frac{-13\sqrt{\mu}}{16} \norm{v}^2
    \\
    & \quad+ \Big(\frac{4 \mu^{2}\sqrt{s}+3 L \sqrt{\mu}\oneplusmus}{8
      L^2} \Big)\norm{\nabla f(x)}^2.
    % OLD STUFF & \quad+ \Big( \big(1 + \sqrt{\mu s}\big)
    % \frac{\frac{\sqrt{\mu}}{4}- \sqrt{ \mu}}{2 L}
    % \\
    % & \quad + (\frac{\sqrt{\mu}}{4} 2\mu-\frac{\sqrt{\mu}\big(1 +
    % \sqrt{\mu s})\mu)}{2}\big) \frac{1}{L^2} \Big)\norm{\nabla
    % f(x)}^2.
    % END OLD STUFF
  \end{align*}
  Let us study the other two terms.  By strong convexity, we have $ -
  \langle \nabla f(x + av) - \nabla f(x), v \rangle \leq - a \mu
  \norm{v}^2$, and therefore
  \begin{align*}
    \textrm{Term~II} & \leq -a\oneplusmus\mu \norm{v}^2 \le 0 .
  \end{align*}
  % \begin{figure}[h]
  %   \centering
  %   \includegraphics[width=0.35\textwidth]{epsfiles/term_intuitive_explanation}
  %   \caption{Intuitive explanation of how ``Term II'' affects the
  %     Lyapunov decay.}\label{fig:intuitive}
  % \end{figure}
  % The interpretation of ``Term III'' is
  % harder.  Intuitively, this term would help when moving towards the
  % minimizer under reasonable assumptions, see
  % Figure~\ref{fig:intuitive}
  Regarding Term~III, one can use the $L$-Lipschitzness of $\nabla f$
  and strong convexity to obtain
  \begin{align*}
    \textrm{Term~III} & \leq \displaystyle\frac{a}{\mu}\sqrt{\mu} \oneplusmus L\norm{v}\norm{\nabla f(x)}.
  \end{align*}
  % The upper bound of ``Term~III'' is obtained using the
  % Cauchy-Schwartz inequality and the bound $\norm{x - x_*} \leq
  % \norm{\nabla f(x)}/\mu$.
  %   
  % \marginMV{The bound $\norm{x - x_*}
  % \leq \norm{\nabla f(x)}/\mu$ is obtained through the definition of
  % strong-convexity
  % \[
  % \mu\norm{x - x_*}^2 \leq \langle \nabla f(x) - \nabla f(x_*), x -
  % x_* \rangle \leq \norm{\nabla f(x)}\norm{x - x_*}
  % \]
  % and dividing by $\norm{x - x_*}$ we obtain
  % \[
  % \mu\norm{x - x_*} \leq \norm{\nabla f(x)}
  % \]}
  % 
  % Therefore,
  % \begin{align*}
  %   \dot{V} + \frac{\sqrt{\mu}}{4} V & \leq
  %   (\frac{\sqrt{\mu}}{4}\frac{3}{4} - \sqrt{\mu})\norm{v}^2
  %   \\ %\noalign{\medskip}
  %   & \quad  + \big((1 + \sqrt{\mu s})\frac{\frac{\sqrt{\mu}}{4}- \sqrt{ \mu}}{2 L}\\
  %   & \quad + (\frac{\sqrt{\mu}}{4} 2\mu-\frac{\sqrt{\mu}(1 + \sqrt{\mu
  %   s})\mu)}{2})\frac{1}{L^2} \big)\norm{\nabla f(x)}^2
  %   \\ %\noalign{\medskip}
  %   & \quad + a \big(-(1 + \sqrt{\mu s})\mu \norm{v}^2
  %   \\ %\noalign{\medskip}
  %   & \quad + \displaystyle\frac{1}{\mu}\sqrt{\mu} (1 + \sqrt{\mu
  %   s})L\norm{v}\norm{\nabla f(x)}) \big).
  % \end{align*}
  Now, using the notation in the statement, we can write
  \begin{align}\label{eq:a_condition}
    & \langle \nabla V(p), X^a_{\operatorname{hb}}(p) \rangle +
    \frac{\sqrt{\mu}}{4} V(p)
    \\
    & \leq a \big(\!-\!\beta_1\norm{v}^2 + \beta_2\norm{v}\norm{\nabla
      f(x)} \big) \!-\! \beta_3\norm{v}^2 \!-\!\beta_4\norm{\nabla
      f(x)}^2 . \notag
  \end{align}
  If $-\beta_1\norm{v}^2 + \beta_2\norm{v}\norm{\nabla f(x)} \leq 0$,
  then the RHS of~\eqref{eq:a_condition} is negative for any $a \geq
  0$. If $-\beta_1\norm{v}^2 + \beta_2\norm{v}\norm{\nabla f(x)} > 0$,
  the RHS of~\eqref{eq:a_condition} is negative if and only if 
  \[
  a \leq \displaystyle\frac{\beta_3\norm{v}^2 + \beta_4\norm{\nabla
      f(x)}^2}{-\beta_1\norm{v}^2 + \beta_2\norm{v}\norm{\nabla
      f(x)}}. 
  \]
  The RHS of this equation corresponds to $g({\norm{\nabla
      f(x)}}/{\norm{\nabla v}})$, with the function $g$ defined
  in~\eqref{eq:g}. From Lemma~\ref{lemma:bound}, as long as
  $-\beta_1\norm{v}^2 + \beta_2\norm{v}\norm{\nabla f(x)} > 0$, this
  function is lower bounded by
  \begin{align*}
    a_1^* = \frac{\beta_3 + \beta_4(z_{\root}^+)^2}{-\beta_1 + \beta_2
      z_{\root}^+} >0 ,
  \end{align*}
  where $z_{\root}^+$ is defined in~\eqref{eq:zroot}. This exactly
  corresponds to~\eqref{eq:a1}, concluding the
  result.
\end{proof}

\begin{remark}\longthmtitle{Adaptive displacement along the
    trajectories of heavy-ball dynamics with displaced
    gradient}\label{rem:a-over-region}
  {\rm From the proof of Theorem~\ref{th:continuous-sampled}, one can
    observe that if $(x,v)$ is such that $\underline{n} \leq
    \norm{\nabla f(x)} < \overline{n}$ and $\underline{m} \leq
    \norm{v} < \overline{m}$, for $\underline{n}, \overline{n},
    \underline{m}, \overline{m} \in \realpos$, then one can upper
    bound the LHS of~\eqref{eq:a_condition}~by
    \[
    a (-\beta_1 \underline{m}^2 + \beta_2 \overline{m} \, \overline{n}) -
    \beta_3 \underline{m}^2 - \beta_4 \underline{n}^2.
    \]
    If $-\beta_1 \underline{m}^2 + \beta_2 \overline{m} \,
    \overline{n} \le 0$, any $a \ge 0$ makes this expression negative.
    If instead $ -\beta_1 \underline{m}^2 + \beta_2 \overline{m} \,
    \overline{n} \geq 0$, then $a$ must satisfy
    \begin{align}\label{eq:adaptive-a}
      a \leq \Big| \frac{\beta_3 \underline{m}^2 + \beta_4
        \underline{n}^2}{-\beta_1 \underline{m}^2 + \beta_2
        \overline{m} \, \overline{n}} \Big| .
    \end{align}
    This argument shows that over the region $R =
    \setdef{(x,v)}{\underline{n} \leq \norm{\nabla f(x)} <
      \overline{n} \text{ and } \underline{m} \leq \norm{v} <
      \overline{m}}$, any $a \ge 0$ satisfying~\eqref{eq:adaptive-a}
    ensures that $\dot V \le -\frac{\sqrt{\mu}}{4} V$, and hence the
    desired exponential decrease of the Lyapunov function. This
    observation opens the way to modify the value of the parameter $a$
    adaptively along the execution of the heavy-ball dynamics with
    displaced gradient, depending on the region of state space visited
    by its trajectories.  } \oprocend
\end{remark}

\subsection{Triggered Design of Variable-Stepsize
  Algorithms}\label{sec:trigger-design}

%Equipped with this knowledge,
% Section~\ref{sec:trigger-design} proceeds to design a discrete-time
% algorithm based on the performance-based trigger approach.

In this section we propose a discretization of the continuous
heavy-ball dynamics based on resource-aware control. To do so, we
employ the approaches to trigger design described in
Section~\ref{sec:resource-aware} on the dynamics~$\Xhba$, whose
forward-Euler discretization corresponds to the modified zero-order
hold~\eqref{eq:forward-euler-a} of the heavy-ball dynamics.

Our starting point is the characterization of the asymptotic
convergence properties of~$\Xhba$ developed in
Section~\ref{sec:hbdg-asymptotic}.  The trigger design necessitates of
bounding the evolution of the Lyapunov function $V$
in~\eqref{eq:continuous-hb-lyapunov} for the continuous-time
heavy-ball dynamics with displaced gradient along its zero-order hold
implementation. However, this task presents the challenge that the
definition of $V$ involves the minimizer $x_*$ of the optimization
problem itself, which is unknown (in fact, finding it is the ultimate
objective of the discrete-time algorithm we seek to design). In order
to synthesize computable triggers, this raises the issue of bounding
the evolution of $V$ as accurately as possible while avoiding any
requirement on the knowledge of~$x_*$.  The following result, whose
proof is presented in Appendix~\ref{app:appendix},
addresses this point.

\begin{proposition}\longthmtitle{Upper bound for derivative-based
    triggering with zero-order hold}\label{prop:upper-bound-derivative}
  Let $a \ge 0$ and
  define
  \begin{align*}
    b^\db_{\ET}(\hat{p}, t;a) &= A_{\ET}(\hat{p}, t;a) +
    B_{\ET}(\hat{p}, t;a) + C_{\ET}(\hat{p};a),
    \\
    b^\db_{\ST}(\hat{p}, t;a) &= B^q_{\ST}(\hat{p};a)t^2 +
    (A_{\ST}(\hat{p};a) + B^l_{\ST}(\hat{p};a))t
    \\
    & \quad + C_{\ST}(\hat{p};a),
  \end{align*}
  where
  \begin{align*}
    & A_{\ET}(\hat{p}, t;a) = 2 \mu t \norm{\hat{v}}^2 +
    \oneplusmus\big(\langle \nabla f(\hat{x}+t \hat{v}) - \nabla
    f(\hat{x}), \hat{v} \rangle
    \\
    &\quad + 2t \sqrt{\mu} \langle \nabla f(\hat{x} + a \hat{v}),
    \hat{v} \rangle + t\oneplusmus \norm{\nabla f(\hat{x} + a
      \hat{v})}^2 \big),
    \\
    %%%%%%%%%%%%%%% 
    & B_{\ET}(\hat{p}, t;a) = \frac{\sqrt{\mu}t^2}{16} \norm{2
      \sqrt{\mu} \hat{v} + \oneplusmus\nabla f(\hat{x} + a \hat{v})}^2
    \\
    & \quad - \frac{t\mu}{4} \norm{\hat{v}}^2
    % \\
    % & \quad 
    + \frac{\sqrt{\mu}\oneplusmus}{4}\big(
    f(\hat{x} + t \hat{v}) - f(\hat{x}) +
    \\
    &\quad  -t \langle \hat{v}, \nabla f(\hat{x}+a
    \hat{v}) \rangle
    % \\
    % &\quad
    + \frac{ t^2 \oneplusmus}{4}\norm{\nabla f(\hat{x} +
      a \hat{v})}^2
    \\
    &\quad - \frac{t \sqrt{\mu}}{L} \norm{\nabla f(\hat{x} + a
      \hat{v})}^2 
    % \\
    % & \quad
    +t\sqrt{\mu} \langle a\hat{v} , \nabla
    f(\hat{x} + a\hat{v})\rangle\big),
    \\
    %%%%%%%%%%%%%%%%%%%%%%% 5
    % C_{\ET}(\hat{p};a) & = \big((1 +
    % \sqrt{\mu s})\frac{\alpha- \sqrt{ \mu}}{2 L} +
    % (\alpha 2\mu- 
    % \\
    % & \quad \frac{\sqrt{\mu}(1 +
    %   \sqrt{\mu s})\mu}{2})\frac{1}{L^2} \big)\norm{\nabla
    %   f(\hat{x})}^2  + (\alpha\frac{3}{4} - \sqrt{\mu})\norm{v}^2 
    % \\
    % & \quad + a \big(-(1 + \sqrt{\mu s})\mu \norm{\hat{v}}^2
    % \\
    % & \quad + \displaystyle\frac{1}{\mu}\sqrt{\mu} (1 + \sqrt{\mu
    %   s})L\norm{\hat{v}}\norm{\nabla f(\hat{x})}) \big),
    % C_{\ET}(\hat{p};a) & = \big((1 +
    % \sqrt{\mu s})\frac{\sqrt{\mu}/4- \sqrt{ \mu}}{2 L} +
    % (\alpha 2\mu- 
    % \\
    % & \quad \frac{\sqrt{\mu}(1 +
    %   \sqrt{\mu s})\mu}{2})\frac{1}{L^2} \big)\norm{\nabla
    %   f(\hat{x})}^2  + (\alpha\frac{3}{4} - \sqrt{\mu})\norm{v}^2 
    % \\
    % & \quad + a \big(-(1 + \sqrt{\mu s})\mu \norm{\hat{v}}^2
    % \\
    % & \quad + \displaystyle\frac{1}{\mu}\sqrt{\mu} (1 + \sqrt{\mu
    %   s})L\norm{\hat{v}}\norm{\nabla f(\hat{x})}) \big), 
    % \\
    %%%%%%%%%%%%%%%%%%%%%%% 5
    & C_{\ET}(\hat{p};a) = -\frac{13\sqrt{\mu}}{16}\norm{\hat{v}}^2
    -\frac{\mu^2\sqrt{s}}{2}\displaystyle\frac{\norm{\nabla
        f(\hat{x})}^2}{L^2}
    \\
    & \quad + \oneplusmus\big(\frac{-3\sqrt{\mu}}{8L}\norm{\nabla
      f(\hat{x})}^2
    \\
    & \quad + \sqrt{\mu}(f(\hat x) - f(\hat{x} + a\hat{v})) +
    \sqrt{\mu}\norm{\nabla f(\hat{x})}\norm{a \hat{v}}
    \\
    & \quad -\frac{\mu^{3/2}}{2}\norm{a\hat{v}}^2 -\langle \nabla
    f(\hat{x} + a\hat{v}) - \nabla f (\hat{x}) , \hat{v}\rangle
    \\
    & \quad + \sqrt{\mu}\langle \nabla f(\hat{x} + a\hat{v}), a\hat{v}
    \rangle\big) ,
    \\
    % \big(\frac{4
    % \mu^{2}+3 L \sqrt{\mu} \left(1+\sqrt{\mu s}\right)}{8 L^2}
    % \big)\norm{\nabla
    % f(\hat{x})}^2
    % \\
    % & \quad + \frac{-13\sqrt{\mu}}{16}\norm{\hat{v}}^2
    % \\
    % & \quad + a \big(-(1 + \sqrt{\mu s})\mu \norm{\hat{v}}^2
    % \\
    % & \quad + \displaystyle\frac{1}{\mu}\sqrt{\mu} (1 + \sqrt{\mu
    % s})L\norm{\hat{v}}\norm{\nabla f(\hat{x})}) \big),
    % \\
    & A_{\ST}(\hat{p};a) = 2 \mu \norm{\hat{v}}^2+ \oneplusmus
    \big(L\norm{\hat{v}}^2 + 2 \sqrt{\mu} \langle \nabla f(\hat{x} + a
    \hat{v}), \hat{v} \rangle
    \\
    &\quad + \oneplusmus \norm{\nabla f(\hat{x} + a \hat{v})}^2\big),
    \\
    %%%%%%%%%%%%%%%%%%%%%%%%%%%% 
    & B^l_{\ST}(\hat{p};a) = \frac{\sqrt{\mu}}{4}\big(
    -\sqrt{\mu}\norm{\hat{v}}^2 + \oneplusmus( \langle \nabla
    f(\hat{x})- \nabla f(\hat{x} + a\hat{v}),\hat{v} \rangle
    % -\langle \hat{v},\nabla f(\hat{x} + a\hat{v})\rangle
    \\
    & \quad - \frac{\sqrt{\mu}}{L}\norm{\nabla
      f(\hat{x} + a \hat{v})}^2
    % \\
    % &\quad
    + \sqrt{\mu}\langle a\hat{v} , \nabla
    f(\hat{x} + a\hat{v})\rangle)\big),
    \\
    %%%%%%%%%%%%%%%%%%%%%%%%%%%% 
    & B^q_{\ST}(\hat{p};a) =
    \frac{\sqrt{\mu}}{16}\norm{2\sqrt{\mu}\hat{v} +\oneplusmus\nabla
      f(\hat{x} + a \hat{v})}^2
    \\
    & \quad +
    \frac{\sqrt{\mu}\oneplusmus}{4}\big(\frac{L}{2}\norm{\hat{v}}^2
    +\frac{\oneplusmus}{4}\norm{\nabla f(\hat{x} + a \hat{v})}^2\big),
    \\
    & C_{\ST }(\hat{p};a) = C_{\ET}(\hat{p};a).
  \end{align*}
  Let $t\mapsto p(t) = \hat{p} + t \Xhba(\hat{p})$ be the trajectory
  of the zero-order hold dynamics $\dot p= \Xhba (\hat{p})$, $p(0) =
  \hat{p}$.  Then, for $t\ge 0$,
  \begin{align*}
    \frac{d}{dt} V(p(t)) +\frac{\sqrt{\mu}}{4} V({p}(t))
    % & = \langle \nabla
    % V(p(t)),\Xhba(\hat{p}) \rangle + \alpha V(p(t))
    % \\
    & \leq b^\db_{\ET}(\hat{p},t;a) \leq b^\db_{\ST}(\hat{p}, t;a) .
  \end{align*}
\end{proposition}

The importance of Proposition~\ref{prop:upper-bound-derivative} stems
from the fact that the triggering conditions defined by $b^\db_\#$,
$\# \in \{\ET,\ST\}$, can be evaluated without knowledge of the
optimizer~$x_*$.  We build on this result next to establish an upper
bound for the performance-based triggering condition.

\begin{proposition}\longthmtitle{Upper bound for performance-based
    triggering with zero-order hold}\label{prop:upper-bound-performance}
  Let $a \ge 0$ and
  \begin{align*}
    b^\pb_{\#}(\hat{p},t;a) &= \int_0^te^{\frac{\sqrt{\mu}}{4}
      \zeta}b^\db_{\#}(\hat p,\zeta;a) d\zeta ,
  \end{align*}
  for $\# \in \{\ET,\ST\}$. Let $t\mapsto p(t) = \hat{p} + t
  \Xhba(\hat{p})$ be the trajectory of the zero-order hold dynamics
  $\dot p= \Xhba (\hat{p})$, $p(0) = \hat{p}$.  Then, for $t \ge 0 $,
  \begin{align*}
    V(p(t)) \!-\! e^{-\frac{\sqrt{\mu}}{4} t} V(\hat p) \!\leq\!
    e^{-\frac{\sqrt{\mu}}{4} t} b^\pb_{\ET} (\hat{p},t;a) \!\leq\!
    e^{-\frac{\sqrt{\mu}}{4} t} b^\pb_{\ST} (\hat{p},t;a).
  \end{align*}
\end{proposition}
% 
% \marginJC{Miguel, if we define $ g^\pb_{\#}(\hat{p},t) =
% \int_0^te^{\alpha (\zeta-t)}g^\db_{\#}(\hat p,\zeta) d\zeta$, then we
% can remove the factor $e^{-\alpha t}$ in the inequality. If we do
% this, rest of the paper needs to be made consistent with the
% change.}
\begin{proof}%[Proof of Proposition~\ref{prop:upper-bound-performance}]
  We rewrite $ V(p(t)) - e^{-\frac{\sqrt{\mu}}{4} t} V(\hat p) =
  e^{-\frac{\sqrt{\mu}}{4} t} (e^{\frac{\sqrt{\mu}}{4} t} V(p(t)) -
  V(\hat p))$, and note that
  \begin{align*}
    & e^{\frac{\sqrt{\mu}}{4} t} V(p(t)) - V(\hat p)
    \\
    & \quad = \int_0^t \frac{d}{d\zeta} \big( e^{\frac{\sqrt{\mu}}{4} \zeta}V(p(\zeta))
    - V(\hat p)\big) d\zeta
    % \\
    % & \quad = \int_0^te^{\alpha \zeta}\alpha V(p(\zeta)) + e^{\alpha
    %   \zeta}\langle \nabla V(p(\zeta)),\Xhba(\hat{p})\rangle d\zeta
    \\
    & \quad = \int_0^te^{\frac{\sqrt{\mu}}{4} \zeta} \Big( \frac{d}{d\zeta} V(p(\zeta)) +
    \frac{\sqrt{\mu}}{4} V(p(\zeta)\Big) d\zeta.
  \end{align*}
  Note that the integrand corresponds to the derivative-based
  criterion bounded in
  Proposition~\ref{prop:upper-bound-derivative}. Therefore, 
  \begin{align*}
    e^{\frac{\sqrt{\mu}}{4} t} V(p(t)) - V(\hat p) & \leq \int_0^t e^{\frac{\sqrt{\mu}}{4} \zeta}
    b^\db_{\ET}(\hat p,\zeta;a) d\zeta
    \\
    & = b^\pb_{\ET}(\hat{p},t;a) \leq b^\pb_{\ST}(\hat{p},t;a)
  \end{align*}
  for $t \ge 0$,  and the result follows.
  % \begin{remark}
  %   The manipulations above are easily explained by the fact that they
  %   lead to easier estimates. Other way of looking at them is thought
  %   the optimality gap approach presented in
  %   ~\cite{JD-LO:17-arxiv}. The idea is to try to design dynamics such
  %   that the decay condition $\beta(t)(f(x(t)) - x_*) \leq c$ is
  %   satisfied for some increasing function $\beta(t)$ and some real
  %   constant $c$. In our case one could see that $\beta(t) = e^{\alpha
  %   t}$ and the objective function is replaced by the Lyapunov
  %   function $V$.
  % \end{remark}  
\end{proof}

Propositions~\ref{prop:upper-bound-derivative}
and~\ref{prop:upper-bound-performance} provide us with the tools to
determine the stepsize according to the derivative- and
performance-based triggering criteria, respectively. For convenience,
and following the notation in~\eqref{eq:step-generic}, we define the
stepsizes
\begin{subequations}
  \begin{align}
    \dstep_{\#}(\hat{p};a) & = \min \setdef{t >
      0}{b^\db_{\#}(\hat{p},t;a) = 0} ,
    \label{eq:step_derivative}
    \\
    \pstep_{\#}(\hat{p};a) & = \min \setdef{t > 0}{b^\pb_{\#}(\hat{p},t;a) = 0} ,
    \label{eq:step_performance}
  \end{align}
\end{subequations}
for $\# \in \{\ET,\ST\}$.  Observe that, as long as $\hat{p} \neq p_*
= [x_*,0]^T$ and $0\leq a \leq a^*_1$, we have $C_{\#}(\hat{p};a) < 0$
for $\#\in\{\ST,\ET\}$ and, as a consequence, $
b^\db_{\#}(\hat{p},0;a)< 0$.  The ET/ST terminology is justified by
the following observation: in the ET case, the equation defining the
stepsize is in general implicit in~$t$. Instead, in the ST case, the
equation defining the stepsize is explicit in~$t$.
% Since $g^\db_{\ST}(p,t)$ is a second-order polynomial in $t$, it is
% easy to obtain an explicit expression for the integral
% $\int_0^te^{\alpha \zeta}g^\db_{\ST}(p(0),\zeta) d\zeta$, as in general
% \[
% \begin{array}{l}
%   \int_0^te^{\alpha \zeta}(a\zeta^2 + b\zeta + c)d\zeta 
%   \\
%   \noalign{\smallskip}
%   = [e^{\alpha t}(\frac{a}{\alpha}t^2 + (\frac{b}{\alpha} -
%   \frac{2a}{\alpha^2})t + \frac{c}{\alpha} -
%   \frac{b}{\alpha^2}+\frac{2a}{\alpha^3})]^t_0  
%   \\
%   \noalign{\smallskip}
%   = e^{\alpha t}(\frac{a}{\alpha}t^2 + (\frac{b}{\alpha} -
%   \frac{2a}{\alpha^2})t + \frac{c}{\alpha} -
%   \frac{b}{\alpha^2}+\frac{2a}{\alpha^3}) - ( \frac{c}{\alpha} -
%   \frac{b}{\alpha^2}+\frac{2a}{\alpha^3}) \label{integral} 
% \end{array}
% \]
Equipped with this notation, we define the variable-stepsize algorithm
described in Algorithm~\ref{algo:DG}, which consists of following
the dynamics~\eqref{eq:forward-euler-a} until the exponential decay of
the Lyapunov function is violated as estimated by the derivative-based
($\diamond = \db$) or the performance-based ($\diamond = \pb$) triggering condition.
When this happens, the algorithm re-samples the state before continue
flowing along~\eqref{eq:forward-euler-a}.

\begin{algorithm}[h]
  \SetAlgoLined
  \textbf{Design Choices:} $\diamond \in \{\db,\pb\}$, $\# \in \{\ET,\ST\}$
  \textbf{Initialization:} Initial point ($p_0$), % convergence rate ($\alpha$),
  objective function ($f$), tolerance ($\epsilon$), $a \ge 0$, $k=0$
  \\
  \While{$\norm{\nabla f(x_k)}\geq \epsilon$}{
    Compute stepsize $\Delta_k = \step^\diamond_{\#}(p_k;a)$
    \\
    Compute next iterate $p_{k+1} = p_k +\Delta_k \Xhba(p_k)$
    \\
    Set $k = k+1$ }
  \caption{Displaced-Gradient Algorithm}\label{algo:DG}
\end{algorithm}

\subsection{Convergence Analysis of Displaced-Gradient
  Algorithm}\label{sec:analysis}

Here we characterize the convergence properties of the derivative- and
performance-based implementations of the Displaced-Gradient
Algorithm. In each case, we show that algorithm is implementable
(i.e., it admits a MIET) and inherits the convergence rate from the
continuous-time dynamics.  The following result deals with the
derivative-based implementation of Algorithm~\ref{algo:DG}.

\begin{theorem}\longthmtitle{Convergence of derivative-based
    implementation of Displaced-Gradient 
    Algorithm}\label{non-zeno-hb-db}
  Let $\hat{\beta}_1,\dots,\hat{\beta}_5>0$ be
  \begin{alignat*}{2}
    \hat{\beta}_1 & = \oneplusmus(\frac{3\sqrt{\mu}}{2} + L), &
    \hat{\beta}_2 & = \sqrt{\mu}\oneplusmus\frac{3}{2},
    \\
    \hat{\beta}_3& = \frac{13\sqrt{\mu}}{16},% \norm{v}^2
    & \hat{\beta}_4& = \frac{4 \mu^{2}\sqrt{s}+3 L \sqrt{\mu}
      \oneplusmus}{8 L^2},
    \\
    \hat{\beta}_5& = \oneplusmus\big(\frac{5\sqrt{\mu}L}{2} -
    \frac{\mu^{3/2}}{2}\big), &&
  \end{alignat*}
  and define
  \begin{align}\label{eq:a2}
    a^*_2 = \alpha \min \Big\{\frac{-\hat \beta_1 + \sqrt{\hat
        \beta_1^2 + 4 \hat \beta_5\hat \beta_3}}{2\hat
      \beta_5},\frac{\hat{\beta}_4}{\hat{\beta}_2} \Big\} ,
  \end{align}
  with $0<\alpha<1$.  Then, for $0 \leq a \leq a^*_2$, $\diamond =
  \db$, and $\# \in \{\ET,\ST\}$, the variable-stepsize strategy in
  Algorithm~\ref{algo:DG} has the following properties
  \begin{enumerate}
  \item[(i)] the stepsize is uniformly lower bounded by the positive
    constant $\miet(a)$, where
    \begin{equation}\label{g-definition}
      \miet(a)= -\nu + \sqrt{\nu^2 + \eta}, 
    \end{equation}
    $\eta= \min\{\eta_1,\eta_2\}$, $\nu = \max\{\nu_1,\nu_2\}$, and
    \begin{align*}
      \eta_1 &=\frac{8 a \oneplusmus \left(a (\mu -5 L)-\frac{2
            L}{\sqrt{\mu }}-3\right)+13}{2 \oneplusmus L \left(3 a^2
          \oneplusmus L+1\right)+8 \mu } ,
      \\
      % \eta_2 &= \scalemath{0.7}{-\frac{4 \alpha \mu +\frac{a L^3
      % \left(\sqrt{\mu s}+1\right)}{\sqrt{\mu }}+L \left(\alpha
      %   -\sqrt{\mu }\right) \left(\sqrt{\mu s}+1\right)-\mu ^{3/2}
      % \left(\sqrt{\mu s}+1\right)}{3 L^2\left(\sqrt{\mu
      %   s}+1\right)^2}},
      % \\
      \eta_2 &=-\frac{3 \oneplusmus \sqrt{\mu } L (4 a L-1)-4 \mu ^2
        \sqrt{s}}{3 \mu_s \sqrt{\mu } L^2},
      \\
      \nu_1 &=\frac{\mu \left(2 a^3 \oneplusmus L^2+a
          \oneplusmus+16\right)+8 \oneplusmus L \left(2 a^2
          \oneplusmus L+1\right)}{2 \sqrt{\mu } \left(\oneplusmus L
          \left(3 a^2 \oneplusmus L+1\right)+4 \mu \right)}
      \\
      &\quad+ \frac{\oneplusmus  (a L (8 a L+1)+4)}{
       \oneplusmus L \left(3 a^2 \oneplusmus
            L+1\right)+4 \mu},
        % \\
        % \nu_1 &= 1/2\frac{2(a^3\oneplusmus\mu +
        % 8a^2\oneplusmus^2)L^2 + 8L\oneplusmus}{(3L^2a^2\oneplusmus^2
        % + L\oneplusmus + 4\mu)\sqrt{\mu}}
        % \\
        % & + 1/2\frac{ (a\oneplusmus + 16)\mu + 2(8L^2a^2\oneplusmus
        % + La\oneplusmus +
        % 4\oneplusmus)\sqrt{\mu}}{(3L^2a^2\oneplusmus^2 +
        % L\oneplusmus + 4\mu)\sqrt{\mu}}
        \\
        \nu_2 &=\frac{a \mu +8 \oneplusmus+8 \sqrt{\mu}}{3 \oneplusmus
          \sqrt{\mu}};
    \end{align*}
  \item[(ii)] $ \frac{d}{dt} V(p_{k}+t \Xhba (p_k)) \leq
    -\frac{\sqrt{\mu}}{4} V(p_k+t \Xhba (p_k))$ for all $t \in
    [0,\Delta_k]$ and $k \in \{0\} \cup \naturals$.
  \end{enumerate}
  As a consequence, $ f(x_{k+1})-f(x_*) =
  \mathcal{O}(e^{-\frac{\sqrt{\mu}}{4}\sum_{i=0}^k \Delta_i})$.
\end{theorem}
\begin{proof}%[Proof of Theorem~\ref{non-zeno-hb-db}]
  Regarding fact (i), we prove the result for the $\ST$-case, as the
  $\ET$-case follows from $\dstep_{\ET}(\hat{p};a) \geq
  \dstep_{\ST}(\hat{p}; a)$.
  % Using $\norm{y_1 + y_2}^2\leq
  % 2\norm{y_1}^2 + 2\norm{y_2}^2$ we have
  % \begin{align*}
  %   & \norm{\nabla f(\hat x + a \hat v)}^2  = \norm{\nabla f(\hat x + a \hat v) - \nabla f(\hat x) + \nabla f(\hat x)}^2
  %   \\
  %   & \leq 2\norm{\nabla f(\hat x + a \hat v) - \nabla f(\hat x)}^2 + 2\norm{\nabla f(\hat x )}^2
  %   \\
  %   &  \leq 2L^2a^2\norm{\hat v}^2 + 2\norm{\nabla f(\hat x)}^2.
  % \end{align*}  
  % From~\eqref{eq:step_derivative}, we have
  % \begin{align}\label{eq:step_derivative_explicit}
  %   & \dstep_{\ST}(\hat{p};a) = \frac{-(A_{\ST}(\hat{p};a) +
  %     B^l_{\ST}(\hat{p};a))}{2B^q_{\ST}(\hat{p};a)}
  %   \\
  %   &\quad + \sqrt{\left(\frac{A_{\ST}(\hat{p};a) +
  %         B^l_{\ST}(\hat{p};a)}{ 2B^q_{\ST}(\hat{p};a) }\right)^2
  %     -\frac{C_{\ST}(\hat{p};a)}{ B^q_{\ST}(\hat{p};a)}}. \notag
  % \end{align}
  We start by upper bounding $C_{\ST}(\hat{p};a)$ by a negative
  quadratic function of $\norm{\hat{v}}$ and $\norm{\nabla
    f(\hat{x})}$ as follows,
  \begin{align*}
    & C_{\ST}(\hat{p};a) = -\frac{13\sqrt{\mu}}{16}\norm{\hat{v}}^2 +
    \oneplusmus\frac{-3\sqrt{\mu}}{8L}\norm{\nabla
        f(\hat{x})}^2
    \\
    & \quad -\frac{\mu^2\sqrt{s}}{2L^2}\displaystyle\norm{\nabla
        f(\hat{x})}^2
    +\oneplusmus\big(\sqrt{\mu}\underbrace{(f(\hat x) - f(\hat{x} +
      a\hat{v}))}_{\textrm{(a)}}
    \\
    & \quad + \sqrt{\mu}\underbrace{\norm{\nabla f(\hat{x})}\norm{a
        \hat{v}}}_{\textrm{(b)}} -\frac{\mu^{3/2}}{2}\norm{a\hat{v}}^2
    \\
    & \quad +\underbrace{\langle \nabla f(\hat{x}) - \nabla f (\hat{x}
      + a\hat{v} ) , \hat{v}\rangle}_{\textrm{(c)}} +
    \sqrt{\mu}\underbrace{\langle \nabla f(\hat{x} + a\hat{v}),
      a\hat{v} \rangle}_{\textrm{(d)}}\big).
  \end{align*}
  Using the $L$-Lipschitzness of the gradient and Young's inequality,
  we can easily upper bound
  \begin{align*}
    \textrm{(a)} & \le \underbrace{\langle \nabla f(\hat x + a \hat
      v), - a \hat v \rangle + \frac{L}{2}a^2\norm{\hat v}^2
    }_{\textrm{Using~\eqref{eq:aux-d}}}
    % \\
    % & = \langle \nabla f(\hat x + a \hat v) - \nabla f(\hat x) +
    % \nabla f(\hat x), - a \hat v \rangle + \frac{L}{2}a^2\norm{\hat
    %   v}^2
    \\
    & = \langle \nabla f(\hat x + a \hat v) - \nabla f(\hat x), - a
    \hat v \rangle + \frac{L}{2}a^2\norm{\hat v}^2
    \\
    & \quad + \langle \nabla f(\hat x), - a \hat v \rangle
    % \\
    % & \leq La^2\norm{\hat v}^2 + \frac{L}{2}a^2\norm{\hat v}^2 +
    % a\norm{\nabla f(\hat x)}\norm{\hat v}
    \\
    & \leq La^2\norm{\hat v}^2 + \frac{L}{2}a^2\norm{\hat v}^2 + a
    \big(\frac{\norm{\nabla f(\hat x)}^2}{2} + \frac{\norm{\hat
        v}^2}{2} \big)
    \\
    % & = (La^2 + \frac{L}{2}a^2 + \frac{a}{2})\norm{\hat v}^2 +
    % \frac{a}{2} \norm{\nabla f(\hat x)}^2
    % \\
    & = \frac{3La^2 + a}{2}\norm{\hat v}^2 + \frac{a}{2} \norm{\nabla
      f(\hat x)}^2,
    \\
    \textrm{(b)} & \le a \big (\frac{\norm{\nabla f(\hat x)}^2}{2} +
    \frac{\norm{\hat v}^2}{2} \big),
    % \\
    % \textrm{(c)} & \leq- a^2\norm{\hat v}^2,
    \\
    \textrm{(c)} & \leq L a \norm{\hat v}^2,
    \\
    \textrm{(d)} & = \langle \nabla f(\hat{x} + a\hat{v}) - \nabla
    f(\hat x) + \nabla f(\hat x), a\hat{v} \rangle
    \\
    & \leq L a^2\norm{\hat v}^2 + \langle \nabla f(\hat x), a \hat v
    \rangle
    % \\
    % & \leq L a^2\norm{\hat v}^2 + a (\frac{\norm{\nabla f(\hat
    %     x)}^2}{2} + \frac{\norm{\hat v}^2}{2})
    \\
    & = \frac{2La^2 + a}{2}\norm{\hat v}^2 + \frac{a}{2}\norm{\nabla
      f(\hat z)}^2.
  \end{align*}
  Note that, with the definition of the constants
  $\hat{\beta}_1,\dots,\hat{\beta}_5>0$ in the statement, we can write
  % \begin{alignat*}{2}
  %   \hat{\beta}_1 & = \oneplusmus(\frac{3\sqrt{\mu}}{2} + L), &
  %   \hat{\beta}_2 & = \sqrt{\mu}\oneplusmus\frac{3}{2},
  %   \\
  %   \hat{\beta}_3& = \frac{13\sqrt{\mu}}{16},% \norm{v}^2
  %   & \hat{\beta}_4& = \frac{4 \mu^{2}\sqrt{s}+3 L \sqrt{\mu}
  %     \oneplusmus}{8 L^2},
  %   \\
  %   \hat{\beta}_5& = \oneplusmus\big(\frac{5\sqrt{\mu}L}{2} -
  %   \frac{\mu^{3/2}}{2}\big), &&
  % \end{alignat*}
  % so that we can express
  \begin{align*}
    C_{\ST}(\hat p;a)& \leq a \hat{\beta}_1\norm{\hat{v}}^2 +
    a^2\hat{\beta}_5\norm{\hat v}^2 + a\hat{\beta}_2\norm{\nabla
      f(\hat x)}^2
    \\
    & \quad -\hat{\beta}_3\norm{\hat{v}}^2 -\hat{\beta}_4\norm{\nabla
      f(\hat{x})}^2 .
  \end{align*}
  % The only non-obvious inequality is $\hat\beta_5 > 0$ which follows
  % from the fact that $\mu \leq
  % L$.  Taking $0\leq a \leq 1$, we have $a^2 \leq a $ that allows us
  % to re-write the previous expression as an affine function of
  % $a$. That is
  % \begin{align*}
  %   C_{\ST}(\hat p;a)& \leq a( (\hat{\beta}_1 + \hat{\beta}_
  %   5)\norm{\hat{v}}^2
  %   +  \hat{\beta}_2\norm{\nabla f(\hat x)}^2) \\
  %   & \quad -\hat{\beta}_3\norm{\hat{v}}^2
  %   -\hat{\beta}_4\norm{\nabla f(\hat{x})}^2
  % \end{align*}
  Therefore, for $a \in [0,a^*_2]$, we have
  \begin{align*}
    a \hat{\beta}_1 + a^2\hat \beta_5 - \hat{\beta}_3 & \le a^*_2
    \hat{\beta}_1 + (a^*_2)^2\hat \beta_5 - \hat{\beta}_3 = -\gamma_1
    < 0
    \\
    a \hat{\beta}_2 - \hat{\beta}_4 & \le a^*_2 \hat{\beta}_2 -
    \hat{\beta}_4 = -\gamma_2< 0,
  \end{align*}
  and hence $C_{\ST}(\hat{p};a)\leq - \gamma_1\norm{\hat{v}}^2 -
  \gamma_2\norm{\nabla f(\hat{x})}^2$.  Similarly,  introducing
  \begin{align*}
    \gamma_3 &= 2 a^2 \mu_s L^2+2 a^2 \oneplusmus \sqrt{\mu }
    L^2+\oneplusmus \sqrt{\mu }+\oneplusmus L +2 \mu,
    \\
    \gamma_4 & = 2 \mu_s + 2 \oneplusmus \sqrt{\mu}, \; \gamma_5 \!=\!
    \frac{1}{8} a \oneplusmus \left(2 a^2 \mu L^2+\mu +2 \sqrt{\mu }
      L\right),
    \\
    \gamma_6 & = \frac{a\mu\oneplusmus}{4}, \; \gamma_7 = \frac{3}{8}
    a^2 \mu_s \sqrt{\mu } L^2+\frac{1}{8} \oneplusmus \sqrt{\mu }
    L+\frac{\mu ^{3/2}}{2},
    \\
    \gamma_8 & = \frac{3 \mu_s\sqrt{\mu}}{8},
  \end{align*}
  one can show that
  \begin{align*}
    A_{\ST}(\hat p;a) \le \hat{A}_{\ST}(\hat p;a) &= \gamma_3
    \norm{\hat v}^2 + \gamma_4 \norm{\nabla f(\hat x)}^2,
    \\
    B^l_{\ST}(\hat p;a) \le \hat{B}^l_{\ST}(\hat p;a) & = \gamma_5
    \norm{\hat v}^2 + \gamma_6 \norm{\nabla f(\hat x)}^2,
    \\
    B^q_{\ST}(\hat p;a) \le \hat{B}^q_{\ST}(\hat p;a) & =\gamma_7
    \norm{\hat v}^2 + \gamma_8 \norm{\nabla f(\hat x)}^2.
  \end{align*}
  Thus, from~\eqref{eq:step_derivative}, we have
  \begin{align}\label{eq:step_derivative_explicit}
    & \dstep_{\ST}(\hat{p};a) \geq \frac{-(\hat A_{\ST}(\hat{p};a) +
      \hat B^l_{\ST}(\hat{p};a))}{2 \hat B^q_{\ST}(\hat{p};a)}
    \\
    &\quad + \sqrt{\left(\frac{\hat A_{\ST}(\hat{p};a) +
          \hat B^l_{\ST}(\hat{p};a)}{ 2 \hat  B^q_{\ST}(\hat{p};a) }\right)^2
      -\frac{C_{\ST}(\hat{p};a)}{\hat  B^q_{\ST}(\hat{p};a)}}. \notag
  \end{align}
 % A similar reasoning leads us
 %  to obtain constants $\gamma_i>0$ such that
 %  \begin{align*}
 %    B^q_{\ST}(\hat{p};a) & \leq \gamma_3\norm{\hat{v}}^2 +
 %    \gamma_4\norm{\nabla f(\hat{x})}^2,
 %    \\
 %    2B^q_{\ST}(\hat{p};a) & \geq \gamma_5\norm{\hat{v}}^2 +
 %    \gamma_6\norm{\nabla f(\hat{x})}^2,
 %    \\
 %    A_{\ST}(\hat{p};a)+ B^l_{\ST}(\hat{p};a) & \leq \gamma_7
 %    \norm{\hat{v}}^2 + \gamma_8\norm{\nabla f(\hat{x})}^2 .
 %  \end{align*}
  % Notice that instead of $A_{\ST}(\hat{p};a),\ B^l_{\ST}(\hat{p};a)$
  % and $ B^q_{\ST}(\hat{p};a)$ one may actually use any upper bound
  % of
  % these functions, which simplifies the computations.
  % Now, observe that
  % \begin{align*}
  %   \frac{\gamma_1\norm{\hat{v}}^2 + \gamma_2\norm{\nabla
  %   f(\hat{x})}^2}{\gamma_7\norm{\hat{v}}^2 + \gamma_8\norm{\nabla
  %   f(\hat{x})}^2} & \leq
  %   \frac{-C_{\ST}(\hat{p};a)}{\hat B^q_{\ST}(\hat{p};a)} ,
  %   \\
  %   \frac{\hat A_{\ST}(\hat{p};a) +
  %   \hat B^l_{\ST}(\hat{p};a)}{2\hat B^q_{\ST}(\hat{p};a)} &=
  %   \frac{(\gamma_3 + \gamma_5)\norm{\hat{v}}^2 + (\gamma_4 +
  %   \gamma_6)\norm{\nabla
  %   f(\hat{x})}^2}{2\gamma_7\norm{\hat{v}}^2 + 2\gamma_8\norm{\nabla
  %   f(\hat{x})}^2}.
  % \end{align*}
  Using now \cite[supplementary material, Lemma~1]{MV-JC:19-nips}, we
  deduce
  \begin{align*}
    \eta \leq 
     \frac{-C_{\ST}(\hat{p};a)}{\hat B^q_{\ST}(\hat{p};a)}  ,
     % \frac{\gamma_1\norm{\hat{v}}^2 + \gamma_2\norm{\nabla
     % f(\hat{x})}^2}{\gamma_7\norm{\hat{v}}^2 + \gamma_8\norm{\nabla
     % f(\hat{x})}^2},
     \quad
     % \frac{(\gamma_3 + \gamma_5)\norm{\hat{v}}^2 + (\gamma_4 +
     % \gamma_6)\norm{\nabla
     % f(\hat{x})}^2}{2\gamma_7\norm{\hat{v}}^2 +
     % 2\gamma_8\norm{\nabla
     % f(\hat{x})}^2}
     \frac{\hat A_{\ST}(\hat{p};a) + \hat B^l_{\ST}(\hat{p};a)}{2\hat
       B^q_{\ST}(\hat{p};a)} \leq \nu,
  \end{align*}
  where
  \begin{align*}
    \eta &=
    \min\{\frac{\gamma_1}{\gamma_7},\frac{\gamma_2}{\gamma_8}\}, \quad
    \nu = \max\{\frac{\gamma_3 + \gamma_5}{2\gamma_7},\frac{\gamma_4 +
      \gamma_6}{2\gamma_8}\}.
  \end{align*}
  With these elements in place and referring
  to~\eqref{eq:step_derivative_explicit}, we have
  \begin{align*}
    \dstep_{\ST}(\hat{p};a) & \geq \frac{-(\hat A_{\ST}(\hat{p};a) +
      \hat B^l_{\ST}(\hat{p};a))}{2\hat B^q_{\ST}(\hat{p};a)}
    \\
    & \quad + \sqrt{\left(\frac{\hat A_{\ST}(\hat{p};a) +
         \hat  B^l_{\ST}(\hat{p};a)}{ 2\hat B^q_{\ST}(\hat{p};a) }\right)^2 +
      \eta}.
    \end{align*}
    We observe now that $z \mapsto g(z) = -z + \sqrt{z^2 + \eta} $ is
    monotonically decreasing and lower bounded. So, if $z$ is upper
    bounded, then $g(z)$ is lower bounded by a positive
    constant. Taking $z = \frac{(\hat A_{\ST}(\hat{p};a) +
      \hat B^l_{\ST}(\hat{p};a))}{2\hat B^q_{\ST}(\hat{p};a)} \leq \nu$ gives
    the bound of the stepsize.  Finally, the algorithm design together
    with Proposition~\ref{prop:upper-bound-derivative} ensure fact
    (ii) throughout its evolution.
\end{proof}

It is worth noticing that the derivative-based implementation of the
Displaced-Gradient Algorithm generalizes the algorithm proposed in our
previous work~\cite{MV-JC:19-nips} (in fact, the strategy proposed
there corresponds to the choice $a=0$).  The next result characterizes
the convergence properties of the performance-based implementation of
Algorithm~\ref{algo:DG}.

\begin{theorem}\longthmtitle{Convergence 
    of performance-based implementation of Displaced-Gradient
    Algorithm}\label{non-zeno-hb-pb} 
  For $0 \leq a \leq a^*_2$, $\diamond = \pb$, and $\# \in
  \{\ET,\ST\}$, the variable-stepsize strategy in
  Algorithm~\ref{algo:DG} has the following properties
  \begin{enumerate}
  \item[(i)] the stepsize is uniformly lower bounded by the positive
    constant $\miet(a)$;
  \item[(ii)] $ V(p_{k }+t \Xhba (p_k)) \leq e^{-\frac{\sqrt{\mu}}{4}
      t} V(p_k)$ for all $t \in [0,\Delta_k]$ and $k \in \{0\} \cup
    \naturals$.
  \end{enumerate}
  As a consequence, $ f(x_{k+1})-f(x_*) =
  \mathcal{O}(e^{-\frac{\sqrt{\mu}}{4}\sum_{i=0}^k \Delta_i})$.
\end{theorem}
\begin{proof}
  To show (i), notice that it is sufficient to prove that
  $\pstep_{\ST}$ is uniformly lower bounded away from zero. This is
  because of the definition of stepsize in~\eqref{eq:step_performance}
  and the fact that $b^\pb_{\ET} (\hat p,t;a) \le b^\pb_{\ST} (\hat
  p,t;a)$ for all $\hat p$ and all $t$.  For an arbitrary fixed $\hat
  p$, note that $t \mapsto b^\db_{\ST}(\hat p,t;a)$ is strictly
  negative in the interval $[0,\dstep_{\ST}(p;a))$ given the
  definition of stepsize in~\eqref{eq:step_derivative}.  Consequently,
  the function $t \mapsto b^\pb_{\ST} (\hat p,t;a)=\int_0^t
  e^{\frac{\sqrt{\mu}}{4} \zeta} b^\db_{\ST}(\hat p;\zeta,a) d\zeta$
  is strictly negative over $(0,\dstep_{\ST}(\hat p;a))$.  From the
  definition of $\pstep_{\ST}$, it then follows that
  $\pstep_{\ST}(\hat p;a)\geq \dstep_{\ST}(\hat p;a)$. The result now
  follows by noting that $ \dstep_{\ST}$ is uniformly lower bounded
  away from zero by a positive constant,
  cf. Theorem~\ref{non-zeno-hb-db}(i).

  To show (ii), we recall that $\Delta_k = \pstep_{\#}(p_k;a)$ for $\#
  \in \{\ET,\ST\}$ and use
  Proposition~\ref{prop:upper-bound-performance} for $\hat p = p_k$ to
  obtain, for all $t \in [0,\Delta_k]$,
  \begin{align*}
    V(p(t)) - e^{-\frac{\sqrt{\mu}}{4} t} V(p_k) & \leq e^{-\frac{\sqrt{\mu}}{4} t} b^\pb_{\#}
    (p_k,t;a)
    \\
    & \le e^{-\frac{\sqrt{\mu}}{4} t} b^\pb_{\#} (p_k,\Delta_k;a) = 0 ,
  \end{align*}
  as claimed.
\end{proof}

The proof of Theorem~\ref{non-zeno-hb-pb} brings up an interesting
geometric interpretation of the relationship between the stepsizes
determined according to the derivative- and performance-based
approaches.  In fact, since
\begin{align*}
  \frac{d}{dt} b^\pb_{\#}(\hat p,t;a) = e^{\frac{\sqrt{\mu}}{4} t}
  b^\db_{\#} (\hat p,t;a) ,
\end{align*}
we observe that $\dstep_{\#}(\hat{p};a)$ is precisely the (positive)
critical point of $t \mapsto b^\pb_{\#}(\hat p,t;a)$. Therefore,
$\pstep_{\ST}(\hat p;a)$ is the smallest nonzero root of $t \mapsto
b^\pb_{\#}(\hat p,t;a)$, whereas $\dstep_{\ST}(\hat p;a)$ is the time
where $t \mapsto b^\pb_{\#}(\hat p,t;a)$ achieves its smallest value,
and consequently is furthest away from zero.  This confirms the fact
that the performance-based approach obtains larger stepsizes than the
derivative-based approach.

% Given this discussion, a natural question is whether the difference
% $\pstep_{\#}- \dstep_{\#}$ can be positively lower bounded.  The next
% result shows that away from a small neighborhood of the minimizer the
% difference between the performance and derivative-based stepsizes is
% uniformly lower bounded.

% \begin{proposition}\longthmtitle{Uniform bound between performance and
%     derivative-based
%     stepsizes}\label{prop:performance-derivative-stepsize-comparison_aneq0}
%   Given a point $p_0\in \real^{2n}$ and $r > 0$, there exists
%   $\epsilon>0$, depending on $V(p_0)$ and $r$, such that
%   \[
%   \pstep_{\#}(\hat{p};a) - \dstep_{\#}(\hat{p};a) \geq \epsilon ,
%   \]
%   for $\#\in\{\ST,\ET\}$, $a \leq a^*_1$ and $\hat p\in S =
%   \{p\in\mathbb{R}^{2n}\textrm{ such that }\norm{[x_*,0] - p} \geq
%   r\textrm{ and }V( p) \leq V(p_0)\}$.
% \end{proposition}
% This result says that the performance-based approach obtains
% (uniformly) larger stepsizes than the derivative-based approach. This
% bound may depend on the initial condition, but it is valid throughout
% all the sublevel sets including the initial condition of the dynamics.

\section{Exploiting Sampled Information to Enhance Algorithm
  Performance}\label{sec:exploit-sampled-info}

Here we describe two different refinements of the implementations
proposed in Section~\ref{sec:performance-based} to further enhance
their performance. Both of them are based on further exploiting the
sampled information about the system.  The first refinement,
cf. Section~\ref{sec:adaptive}, looks at the possibility of adapting
the value of the gradient displacement as the algorithm is
executed. The second refinement, cf. Section~\ref{sec:HOH}, develops a
high-order hold that more accurately approximates the evolution of the
continuous-time heavy-ball dynamics with displaced gradient.

\subsection{Adaptive Gradient Displacement}\label{sec:adaptive}

The derivative- and performance-based triggered implementations 
% of the heavy-ball dynamics with displaced gradient described 
in Section~\ref{sec:trigger-design} both employ a constant value of
the parameter~$a$. Here, motivated by the observation made in
Remark~\ref{rem:a-over-region}, we develop triggered implementations
that adaptively adjust the value of the gradient displacement
depending on the region of the space to which the state belongs.
Rather than relying on the condition~\eqref{eq:adaptive-a}, which
would require partitioning the state space based on bounds on~$\nabla
f(x)$ and~$v$, we seek to compute on the fly a value of the
parameter~$a$ that ensures the exponential decrease of the Lyapunov
function at the current state. Formally, the strategy is stated in
Algorithm~\ref{algo:ADG}.

\begin{algorithm}[h]
  \SetAlgoLined
  \textbf{Design Choices:} $\diamond \in \{\db,\pb\}$, $\# \in \{\ET,\ST\}$
  \textbf{Initialization:} Initial point ($p_0$), % convergence rate
  % ($\alpha$),
  objective function ($f$), tolerance ($\epsilon$), increase rate
  ($r_i > 1$), decrease rate ($0 < r_d < 1$), stepsize lower bound
  ($\tau$), $a\ge 0$, $k=0$
  \\
  \While{$\norm{\nabla f(x_k)}\geq \epsilon$}{
    increase = True

    exit = False

   \While{ {\rm exit} = {\rm False}}{
    \While{$C_{\#}(p_k;a) \geq 0$  }{
        $a = a r_d$

        increase = False
      }
      \uIf{$\step^\diamond_{\#}(p_k;a) \geq \tau$}{
        exit = True
      }
     \uElse{
     $a = a r_d$

     increase = False
     }
    }
       
    Compute stepsize $\Delta_k = \step^\diamond_{\#}(p_k;a)$
    \\
      Compute next iterate $p_{k+1} = p_k +\Delta_k \Xhb^a(p_k)$
      \\
  Set $k = k+1$

  \uIf{\rm increase = True}{ $a = a r_i$ }

 }
 \caption{Adaptive Displaced-Gradient Algorithm}\label{algo:ADG}
\end{algorithm}

\begin{proposition}\longthmtitle{Convergence  of Adaptive Displaced-Gradient
    Algorithm}\label{prop:non-zeno-sampled-algorithm}
  For % $0\leq\alpha\leq \sqrt{\mu}/4$,
  $\diamond \in \{\db,\pb\}$, $\# \in \{\ET,\ST\}$, and $\tau \le
  \min_{a\in [0,a^*_2]} \miet(a)$, the variable-stepsize strategy in
  Algorithm~\ref{algo:ADG} has the following properties:
  \begin{enumerate}
  \item[(i)] it is executable (i.e., at each iteration, the parameter
    $a$ is determined in a finite number of steps);
  \item[(ii)] the stepsize is uniformly lower bounded by~$\tau$;
  \item[(iii)] it satisfies $ f(x_{k+1})\!-\!f(x_*) \!=\!
    \mathcal{O}(e^{-\frac{\sqrt{\mu}}{4}\sum_{i=0}^k \Delta_i} )$, for
    $k \in \{0\} \cup \naturals$.
  \end{enumerate}
 \end{proposition}
\begin{proof}
  Notice first that the function~$a \mapsto \miet(a)>0$ defined
  in~\eqref{g-definition} is continuous and therefore attains its
  minimum over a compact set.  At each iteration,
  Algorithm~\ref{algo:ADG} first ensures that $C_{\#}(\hat{p};a)< 0 $,
  decreasing $a$ if this is not the case. We know this process is
  guaranteed as soon as $a < a_2^*$ (cf. proof of
  Theorem~\ref{non-zeno-hb-db}) and hence only takes a finite number
  of steps. Once $C_{\#}(\hat{p};a)< 0 $, the stepsize could be
  computed to guarantee the desired decrease of the Lyapunov function
  $V$.  The algorithm next checks if the stepsize is lower bounded
  by~$\tau$. If that is not the case, then the algorithm reduces $a$
  and re-checks if $C_{\#}(\hat{p};a) < 0$.  With this process and in
  a finite number of steps, the algorithm eventually either computes a
  stepsize lower bounded by~$\tau$ with $a > a^*_2$ or $a$ decreases
  enough to make $a \leq a^*_2$, for which we know that the stepsize
  is already lower bounded by~$\tau$. These arguments establish facts
  (i) and (ii) at the same time. Finally, fact (iii) is a consequence
  of the prescribed decreased of the Lyapunov function along the
  algorithm execution.
\end{proof}

\subsection{Discretization via High-Order Hold}\label{sec:HOH}

The modified zero-order hold based on employing displaced gradients
developed in Section~\ref{sec:performance-based} is an example of the
possibilities enabled by more elaborate uses of sampled
information. In this section, we propose another such use based on the
observation that the continuous-time heavy-ball dynamics can be
decomposed as the sum of a linear term and a nonlinear
term. Specifically, we have
\begin{align*}
  \Xhb^a (p) & =
  \begin{bmatrix}
    v
    \\
    - 2\sqrt{\mu}v
  \end{bmatrix}
  + 
  \begin{bmatrix}
    0
    \\
    - \oneplusmus\nabla f(x + av)
    \end{bmatrix} .
\end{align*}
Note that the first term in this decomposition is linear, whereas the
other one contains the potentially nonlinear gradient term that
complicates finding a closed-form solution.  Keeping this in mind when
considering a discrete-time implementation, it would seem reasonable
to perform a zero-order hold only on the nonlinear term while exactly
integrating the resulting differential equation. Formally, a
zero-order hold at $\hat{p}=[\hat{x},\hat{v}]$ of the nonlinear term
above yields a system of the form
% \begin{align*}
%   \Xhb^{H,\hat{p},a}(x,v) = 
%   \begin{bmatrix}
%     v
%     \\
%     -2\sqrt{\mu}v - (1+\sqrt{\mu s})\nabla f(\hat{x} + a \hat{v}) 
%   \end{bmatrix}
% \end{align*}
\begin{align}\label{eq:linear-inhomo}
  \begin{bmatrix}
    \dot x
    \\
    \dot v
  \end{bmatrix}
  & =
  A
  \begin{bmatrix}
    x
    \\
    v
  \end{bmatrix}
  +   b ,
\end{align}
with $p(0) = \hat p$, and where
\begin{align*}
  A=
\begin{bmatrix}
  0 & 1
  \\
  0 & -2\sqrt{\mu}
\end{bmatrix},
\quad b = \begin{bmatrix}
    0
    \\
    - \oneplusmus\nabla f(\hat x + a\hat v)
  \end{bmatrix} .
\end{align*}
Equation~\eqref{eq:linear-inhomo} is an in-homogeneous linear
dynamical system, which is integrable by the method of variation of
constants~\cite{LP:00}.  Its solution is given by $ p(t) = e^{At}
\big(\int_0^t e^{-A\zeta} b d\zeta + p(0) \big) $, or equivalently,
\begin{subequations}\label{eq:hoh-flow}
  \begin{align}
    x(t) & =  \hat{x} -\frac{\oneplusmus\nabla f(\hat{x} + a \hat{v})t }{ 2
      \sqrt{\mu} } 
    \\
    & \quad + (1-e^{-2 \sqrt{\mu } t}) \frac{\oneplusmus\nabla f(\hat{x} +
      a \hat{v})+2 \sqrt{\mu } \hat{v}}{4\mu} , \notag
    % \\
    %   & \quad -\frac{e^{-2 \sqrt{\mu } t} \left((1+\sqrt{\mu s})\nabla
    %       f(\hat{x} + a \hat{v})+2 \sqrt{\mu } \hat{v}\right)}{4 \mu
    %   } \notag
    %   \\
    %   & \quad +\frac{(1+\sqrt{\mu s})\nabla f(\hat{x} + a \hat{v})+2
    %     \sqrt{\mu } \hat{v} +4 \mu \hat{x}}{4 \mu } ,
    \\
    v(t) & = e^{-2 \sqrt{\mu } t}\hat{v} + (e^{-2 \sqrt{\mu } t} -1)
    \frac{\oneplusmus\nabla f(\hat{x} + a \hat{v}) }{2 \sqrt{\mu
      }} .
  \end{align}
\end{subequations}
We refer to this trajectory as a \emph{high-order-hold integrator}.
In order to develop a discrete-time algorithm based on this type of
integrator, the next result provides a bound of the evolution of the
Lyapunov function $V$ along the high-order-hold integrator
trajectories.  The proof is presented in
Appendix~\ref{app:appendix}.

\begin{proposition}\longthmtitle{Upper bound for derivative-based
    triggering with high-order
    hold}\label{prop:upper-bound-derivative-hoh}
  Let $a\ge 0$ %, $0\leq\alpha \le \sqrt{\mu}/4$
  and define
  \begin{align*}
    \hsd_{\ET}(\hat{p}, t;a) &= \fA_{\ET}(\hat{p}, t;a) +
    \fB_{\ET}(\hat{p}, t;a)
    \\
    & \quad + \fC_{\ET}(\hat{p};a) +\fD_{\ET}(\hat{p},t;a),
    \\
    \hsd_{\ST}(\hat{p}, t;a) & = (\fA_{\ST}^q(\hat{p};a) + \fB^q_{\ST}
    (\hat{p};a) ) t^2 + (\fA_{\ST}^l(\hat{p};a)
    \\
    & \quad + \fB^l_{\ST}(\hat{p};a) + \fD_{\ST} (\hat{p};a) )t
    +\fC_{\ST}(\hat{p};a) ,
  \end{align*}
  where
  \begin{align*}
    & \fA_{\ET}(\hat p,t;a) =\oneplusmus(\langle \nabla f(x(t)) -
    \nabla f(\hat{x}),v(t) \rangle
    \\
    & \quad - \langle v(t) - \hat{v}, \nabla f(\hat{x} + a\hat{v})
    \rangle
    \\
    & \quad - \sqrt{\mu}\langle x(t) - \hat{x}, \nabla f(\hat{x} +
    a\hat{v}) \rangle)
    \\
    &\quad -\sqrt{\mu}\langle v(t) - \hat{v},v(t)\rangle,
    \\
    %%%%%%%%%%%%%%%%%%%% 
    & \fB_{\ET}(\hat p,t;a) = \frac{\sqrt{\mu}}{4}\big(\oneplusmus
    (f(x(t)) - f(\hat{x}))
    \\
    &\quad - \sqrt{\mu}\oneplusmus t \frac{\norm{\nabla f(\hat{x} +
        a\hat{v})}^2}{L}
    \\
    &\quad + \sqrt{\mu}\oneplusmus t \langle \nabla f(\hat{x} + a
    \hat{v}), a \hat{v}\rangle + \frac{1}{4} (\norm{v(t)}^2 -
    \norm{\hat{v}}^2)
    \\
    &\quad + \frac{1}{4} \norm{v(t)-\hat{v} + 2 \sqrt{\mu}
      (x(t)-\hat{x})}^2
    \\
    & \quad + \frac{1}{2}\langle v(t)-\hat{v} + 2 \sqrt{\mu}
    (x(t)-\hat{x}), \hat{v} \rangle \big),
    \\
    %%%%%%%%%%%%%%%%%%%%%%%%%%%% 
    & \fC_{\ET}(\hat p;a) = C_{\ET}(\hat{p};a) ,
    \\
    & \fD_{\ET} (\hat p,t;a) = \oneplusmus\langle \nabla f(\hat{x}),
    v(t) - \hat{v}\rangle
    \\
    & \quad -\sqrt{\mu}\langle \hat{v} , v(t) - \hat{v} \rangle ,
  \end{align*}
  and 
  \begin{align*}
    & \fA_{\ST}^l(\hat p;a) = \norm{2 \sqrt{\mu}\hat v
      +\oneplusmus\nabla f(\hat x + a \hat v) } \Big( \sqrt{\mu}
    \norm{\hat v}
    \\
    & \quad + \frac{L\oneplusmus}{2\sqrt{\mu}} \norm{\hat v} +
    \frac{3\oneplusmus}{2} \norm{\nabla f(\hat x + a \hat v)} \Big)
    \\
    & \quad + \frac{\mu_s}{2} \norm{\nabla f(\hat x + a \hat
      v)} \Big( \frac{L}{\sqrt{\mu}} \norm{\hat v} + \norm{\nabla
      f(\hat x + a \hat v)} \Big) ,
    % \fA_{\ST}^l(\hat p;a) & = \frac{L(1 + \sqrt{\mu s})}{2\sqrt{\mu}}
    % \norm{2 \sqrt{\mu}\hat v + (1 + \sqrt{\mu s})\nabla f(\hat x + a
    %   \hat v) }
    % \\
    % & \quad \cdot{} \norm{\hat v} + \frac{L(1 + \sqrt{\mu
    %     s})^2}{2\sqrt{\mu}}\norm{\nabla f(\hat x + a \hat v)}
    % \norm{\hat v}
    % \\
    % & \quad+\sqrt{\mu}\norm{2 \sqrt{\mu}\hat v + (1 + \sqrt{\mu
    %     s})\nabla f(\hat x + a \hat v)}\norm{\hat v}
    % \\
    % & \quad + \frac{3(1 + \sqrt{\mu s})}{2}\norm{2 \sqrt{\mu}\hat v +
    %   (1 + \sqrt{\mu s})\nabla f(\hat x + a \hat v)}
    % \\
    % & \quad \cdot{}\norm{\nabla f(\hat x + a \hat v)} + \frac{(1 +
    %   \sqrt{\mu s})^2}{2}\norm{\nabla f(\hat x + a \hat v)}^2 ,
    \\
    %%%%%%%%%%%%%%%%%%%%%%%%%  
    & \fA_{\ST}^q(\hat p;a) = \norm{2 \sqrt{\mu}\hat v +
      \oneplusmus\nabla f(\hat x + a \hat v) }
    \\
    & \quad \cdot{}\Big( \big( \frac{L\oneplusmus}{2\sqrt{\mu}} +
    \sqrt\mu \big)\norm{2 \sqrt{\mu}\hat v + \oneplusmus\nabla f(\hat
      x + a \hat v) }
    \\
    & \quad + \frac{L\mu_s}{2\sqrt{\mu}} \norm{\nabla f(\hat x
      + a \hat v)} \Big) ,
    % \\
    % & \big( \frac{L(1 + \sqrt{\mu s})}{2\sqrt{\mu}} + \sqrt\mu
    % \big)\norm{2 \sqrt{\mu}\hat v + (1 + \sqrt{\mu s})\nabla f(\hat x
    %   + a \hat v) }^2
    % \\
    % & \quad + \frac{L(1 + \sqrt{\mu s})^2}{2\sqrt{\mu}}\norm{2
    %   \sqrt{\mu}\hat v + (1 + \sqrt{\mu s})\nabla f(\hat x + a \hat v)
    % }
    % \\
    % & \quad \cdot{} \norm{\nabla f(\hat x + a \hat v)},
    \\
    %%%%%%%%%%%%%%%%%%%%%%%%% 
    & \fB^l_{\ST}(\hat p;a) = \frac{\sqrt{\mu} \oneplusmus}{4} \Big(
    \frac{\oneplusmus}{2\sqrt{\mu}}\norm{\nabla f(\hat x + a \hat
      v)}\norm{\nabla f(\hat x)}
    \\
    & \quad + \frac{1}{2} \norm{2 \sqrt{\mu}\hat v \!+\!
      \oneplusmus\nabla f(\hat x + a \hat v) } \big(
    \frac{\norm{\nabla f(\hat x)}}{\sqrt{\mu}} \!+\!  \frac{\norm{\hat
        v}}{\oneplusmus}\big)
    \\
    &\quad - \sqrt{\mu} \frac{\norm{\nabla f(\hat{x} + a
        \hat{v})}^2}{L} + (a\sqrt{\mu} - \frac{1}{2}) \langle \nabla
    f(\hat{x} + a \hat{v}), \hat{v} \rangle \Big),
    \\
    %%%%%%%%%%%%%%%%%%%%%%%%% 
    & \fB^q_{\ST}(\hat p;a) = % \frac{1}{256} \left((4 \mu +1) (8 \mu
      % +1)+8 L^2 \oneplusmus\right)    \\    \\
    \frac{10 \mu ^2+L^2 \oneplusmus}{32 \mu ^{3/2}}
\\
& \quad \cdot{}\norm{2 \sqrt{\mu}\hat v + \oneplusmus\nabla f(\hat x + a \hat v) }^2
    \\
    & \quad +%  \frac{1}{8} \oneplusmus^2
    % \left(\frac{L^2 \oneplusmus}{\mu
    %     ^2}+4\right)
    \frac{\mu_s \left(4 \mu ^2+L^2 \oneplusmus\right)}{32 \mu
      ^{3/2}} \norm{\nabla f(\hat x + a \hat v)}^2
    \\
    & \quad+\frac{\oneplusmus \left(4 \mu ^2+L^2
        \oneplusmus\right)}{16 \mu ^{3/2}}\norm{2 \sqrt{\mu}\hat v +
      \oneplusmus\nabla f(\hat x + a \hat v)}
    \\
    & \quad \cdot{}\norm{\nabla f(\hat x + a \hat v)}),
    \\
    %%%%%%%%%%%%%%%%%%%%%%%%% 
    & \fC_{\ST}(\hat p;a) = C_{\ST}(\hat p;a) ,
    %%%%%%%%%%%%%%%%%%%%%%%%% 
    \\
    & \fD_{\ST}(\hat p;a) = \norm{2 \sqrt{\mu} \hat v +
      \oneplusmus\nabla f(\hat{x} + a \hat{v})} \cdot
    \\
    & \quad \Big( \oneplusmus\norm{\nabla f(\hat x)} +
    \sqrt{\mu}\norm{\hat v} \Big).
  \end{align*}
  Let $t\mapsto p(t)$ be the high-order-hold integrator
  trajectory~\eqref{eq:hoh-flow} from $p(0) = \hat{p}$.  Then, for
  $t\ge 0$,
  \begin{align*}
    \frac{d}{dt} V(p(t)) +\frac{\sqrt{\mu}}{4} V({p}(t))
    % & = \langle \nabla
    % V(p(t)),\Xhba(\hat{p}) \rangle + \alpha V(p(t))
    % \\
    & \leq \hsd_{\ET}(\hat{p},t;a) \leq \hsd_{\ST}(\hat{p}, t;a) .
  \end{align*}
  % \begin{equation}\label{4-terms}
  %   \begin{array}{l}
  %     \frac{d}{dt} V(p(t)) +\alpha V({p}(t)) \\ \noalign{\smallskip} \noalign{\medskip}
  %     = \langle \nabla V(p(t)),\Xhb^{H,\hat{p},a}(p(t)) \rangle + \alpha V(p(t))
  %     \\ \noalign{\smallskip} \noalign{\medskip}
  %     =\underbrace{\langle \nabla V(p(t)) -\nabla
  %       V(\hat{p}), \Xhb^{H,\hat{p},a}(p(t)) \rangle}_{\mathrm{Term^H \ I}} 
  %     \\ \noalign{\smallskip} 
  %     + \underbrace{\langle \nabla V(p(0)), \Xhb^{H,\hat{p},a} (p(t))
  %       - \Xhb^{H,\hat{p},a} (\hat{p}) \rangle}_{\mathrm{Term^H \ II}} 
  %     \\ \noalign{\smallskip}
  %     + \underbrace{\alpha(V(p(t))-V(p(0))}_{\mathrm{Term^H \ III}}
  %     \\ \noalign{\smallskip}
  %     +\underbrace{\langle \nabla V(\hat{p}), \Xhb^{H,\hat{p},a}(\hat{p})\rangle
  %       + \alpha V(\hat{p})}_{\mathrm{Term^H \ IV}}.
  %   \end{array}
  % \end{equation}
  % Then, the following bounds hold:
  % \begin{enumerate}
  % \item $ \mathrm{Term^H \ I} \leq A^H_{\ET}(\hat{p}, t;a) \leq
  %   A^H_{\ST}(\hat{p};a)t$;
  % \item $ \mathrm{Term^H \ II} \leq E^H_{\ET}(\hat{p}, t;a) \leq
  %   E^H_{\ST}(\hat{p}, t;a)$;
  % \item $ \mathrm{Term^H \ III} \leq BC^H_{\ET}(\hat{p}, t;a) \leq
  %   BC^H_{\ST}(\hat{p}, a,t)$;
  % \item $ \mathrm{Term^H \ VI} \leq D^H_{\ET}(\hat{p}, t;a) =
  %   D^H_{\ST}(\hat{p};a)$,
  % \end{enumerate}
\end{proposition}

Analogously to what we did in Section~\ref{sec:trigger-design}, we
build on this result to establish an upper bound for the
performance-based triggering condition with the high-order-hold
integrator.

\begin{proposition}\longthmtitle{Upper bound for performance-based
    triggering with high-order
    hold}\label{prop:upper-bound-performance-hoh}
  Let $0 \leq a$ and
  \begin{align}\label{eq:bpb}
    \hsp_{\#}(\hat{p},t;a) &= \int_0^te^{\frac{\sqrt{\mu}}{4} \zeta}\hsd_{\#}(\hat
    p,\zeta;a) d\zeta ,
  \end{align}
  for $\# \in \{\ET,\ST\}$.  Let $t\mapsto p(t)$ be the
  high-order-hold integrator trajectory~\eqref{eq:hoh-flow} from $p(0)
  = \hat{p}$.  Then, for $t\ge 0$,
  \begin{align*}
    V(p(t)) \!-\! e^{-\frac{\sqrt{\mu}}{4} t} V(\hat p) \!\leq\!
    e^{-\frac{\sqrt{\mu}}{4} t} \hsp_{\ET} (\hat{p},t;a) \!\leq\!
    e^{-\frac{\sqrt{\mu}}{4} t} \hsp_{\ST} (\hat{p},t;a).
  \end{align*}
\end{proposition}

Using Proposition~\ref{prop:upper-bound-derivative-hoh}, the proof of
this result is analogous to that of
Proposition~\ref{prop:upper-bound-performance}, and we omit it for
space reasons.  Propositions~\ref{prop:upper-bound-derivative-hoh}
and~\ref{prop:upper-bound-performance-hoh} are all we need to fully
specify the variable-stepsize algorithm based on high-order-hold
integrators. Formally, we set
\begin{align}
  \hstep^\diamond_{\#}(\hat{p};a) = \min \setdef{t >
    0}{\mathfrak{b}^\diamond_{\#}(\hat{p},t;a) = 0} ,
\end{align}
for $\diamond \in \{\db,\pb \}$ and $\# \in \{\ET,\ST\}$. With this in place,
we design Algorithm~\ref{algo:HOH}, which is a higher-order
counterpart to  Algorithm~\ref{algo:ADG}, and whose convergence
properties are characterized in the following result.

\begin{algorithm}[h]
  \SetAlgoLined
  %
% \textbf{Design Choices:} $* \in \{\db,\pb\}$, $\# \in \{\ET,\ST\}$
%   \textbf{Initialization:} Initial point ($p_0$), convergence rate ($\alpha$),
%   objective function ($f$), initial $a$ value ($a$), tolerance ($\epsilon$), increase rate
%   ($r_i > 1$), decrease rate ($0 <r_d <1$), stepsize lower bound
%   ($\tau$), $k=0$
%   \\
%   %
%   \While{$\norm{\nabla f(x_k)}\geq \epsilon$ }{
%     %
%     Compute stepsize $\Delta_k= \hstep^*_{\#}(p_k;a)$
%     \\
%     % 
%     \uIf{$\fC_{\#}(p_k;a) \leq 0$}{
%  \While{$\hstep^*_{\#}(p_k;a) \leq \tau$}{
%      % 
%      $a = a r_d$
%      %
%       }
%  % 
%     Compute stepsize $\Delta_k = \hstep^*_{\#}(p_k;a)$
%     \\
%       %
%       Compute next iterate $p_{k+1}$ using
%       \eqref{eq:hoh-flow}
%       \\
%     %
%       $a = a r_i$
%     }
%     % 
%     \uElse{ \While{$\fC_{\#}(p_k;a) > 0$}{
%         % 
%         $a = a r_d$
%       }
%       % 
%       %Compute stepsize $\delta_k= \hstep_{\ET}(p_k;a)$\;
%       % 
%       %Compute next iterate $p_{k+1}$ using
%       % \eqref{hoh-flow}\; 
%     }
%     % 
%     Set $k = k+1$}
%   \caption{Adaptive High-Order-Hold Algorithm}\label{algo:HOH}
% \end{algorithm}

\textbf{Design Choices:} $\diamond \in \{\db,\pb\}$, $\# \in \{\ET,\ST\}$
  \textbf{Initialization:} Initial point ($p_0$), % convergence rate ($\alpha$),
  objective function ($f$), tolerance ($\epsilon$), increase rate
  ($r_i > 1$), decrease rate ($0 <r_d <1$), stepsize lower bound
  ($\tau$), $a\ge 0$, $k=0$
  \\
  \While{$\norm{\nabla f(x_k)}\geq \epsilon$}{
    increase = True

    exit = False

   \While{ {\rm exit} = {\rm False}}{
    \While{$\fC_{\#}(p_k;a) \geq 0$  }{
        $a = a r_d$

        increase = False
      }
      \uIf{$\hstep^\diamond_{\#}(p_k;a) \geq \tau$}{
        exit = True
      }
     \uElse{
     $a = a r_d$

     increase = False
     }

    }
       
    Compute stepsize $\Delta_k = \hstep^\diamond_{\#}(p_k;a)$
    \\
       Compute next iterate $p_{k+1}$ using \eqref{eq:hoh-flow}
      \\
  Set $k = k+1$

  \uIf{\rm increase = True}{ $a = a r_i$ }

 }
  %%%%%%%%%%%%%%%%%%%%%%%%%%%%%%%%
%%
% \While{$\norm{\nabla f(x_k)}\geq \epsilon$ }{
% %
%     increase = True
% %

%     \While{$\fC_{\#}(p_k;a) \geq 0$}{
% % 
%          $a = a r_d$
% %

%          increase = False
%     }
% % 
%     \While{$\hstep^\diamond_{\#}(p_k;a) \leq \tau$}{
% %

%           $a = a r_d$
% %

%           increase = False
% %
%     }

%     Compute stepsize $\Delta_k = \hstep^\diamond_{\#}(p_k;a)$
% \\
% %      %
%     Compute next iterate $p_{k+1}$ using \eqref{eq:hoh-flow}
% %

   % \uIf{\rm increase = True}{ $a = a r_i$ } Set $k = k+1$}
  \caption{Adaptive High-Order-Hold Algorithm}\label{algo:HOH}
\end{algorithm}

\begin{proposition}\longthmtitle{Convergence
    of Adaptive High-Order-Hold Algorithm}\label{non-zeno:foh}
  For % $0\leq\alpha\leq \sqrt{\mu}/4$,
  $\diamond \in \{\db,\pb\}$, and $\# \in \{\ET,\ST\}$, there exists
  $\miet^\diamond$ such that for $\tau \leq \miet^\diamond$, the
  variable-stepsize strategy in Algorithm~\ref{algo:HOH} has the
  following properties:
  \begin{enumerate}
  \item[(i)] it is executable (i.e., at each iteration, the parameter
    $a$ is determined in a finite number of steps);
  \item[(ii)] the stepsize is uniformly lower bounded by~$\tau$;
  \item[(iii)] it satisfies $ f(x_{k+1})\!-\!f(x_*) \!=\!
    \mathcal{O}(e^{-\frac{\sqrt{\mu}}{4}\sum_{i=0}^k \Delta_i} )$, for
    $k \in \{0\} \cup \naturals$.
  \end{enumerate}
\end{proposition}
% \begin{proof}
  % \begin{align*}
  %   x(t) - x(0) &= \displaystyle\frac{(1+\sqrt{\mu s}) \nabla
  %   f(\hat{x})(1 - e^{-2t\sqrt{\mu}})}{4 \mu}
  %   \\
  %   & \quad + \displaystyle\frac{2\hat{v}\sqrt{\mu}(1 - e^{-2t
  %   \sqrt{\mu}}) - 2(1+\sqrt{\mu s})\nabla f(\hat{x})
  %   t\sqrt{\mu}}{4 \mu}
  %   \\
  %   v(t) - v(0) & = 1/2(e^{-2t\sqrt{\mu}}-1)(2v(0) + (1+\sqrt{\mu
  %   s})\nabla f(x) /\sqrt{\mu})
  % \end{align*}
  % we may bound $\step^H_{\ET}$ by an expression of the form 
  % \[
  % \step^H_{\ET}\leq (\beta_1\norm{\hat{v}}^2 +
  % \beta_2\norm{\nabla f(\hat{x})}^2)t + D_{ST}(\hat{p})
  % \]
  % where $\beta_i$ are strictly positive. To ensure
  % \[
  % \textrm{steps}^H_{\ET} \leq 0
  % \]
  % we just need
  % \begin{align*}
  %   (\beta_1\norm{\hat{v}}^2 + \beta_2\norm{\nabla f(\hat{x})}^2)t +
  %   D_{ST}(\hat{p}) \leq 0\\ \Leftrightarrow t \leq
  %   \displaystyle\frac{-D_{\ST}}{ (\beta_1\norm{\hat{v}}^2 +
  %   \beta_2\norm{\nabla f(\hat{x})}^2}.
  % \end{align*}
  % Finally, using Lemma $1$ in the supplementary material of
  % ~\cite{MV-JC:19-nips} we can proof the result.
%\end{proof}

We omit the proof of this result, which is analogous to that of
Proposition~\ref{prop:non-zeno-sampled-algorithm}, with lengthier
computations.

\section{Simulations}\label{sec:simulation}
Here we illustrate the performance of the methods resulting from the
proposed resource-aware discretization approach to accelerated
optimization flows. Specifically, we simulate in two examples the
performance-based implementation of the Displaced Gradient algorithm
(denoted DG$^\pb$) and the derivative- and performance-based
implementations of the High-Order-Hold (HOH$^\db$ and HOH$^\pb$
respectively) algorithms. We compare these algorithms against the
Nesterov's accelerated gradient and the heavy-ball methods, as they
still exhibit similar or superior performance to the discretization
approaches proposed in the literature, cf. Section~\ref{sec:intro}.

% The first example is a $2$-dimensional, ill-conditioned quadratic
% objective function. This example allows us to investigate the
% evolution of the dynamics in the state space, showing the behavior of
% the algorithm. The second example is a $4$-dimensional logarithmic
% regression. 
% Finally, we also show how the less computationally
% demanding combination {\it performance-based} $+$ {\it
%   gradient-displacement} (from now on DG$^\pb$) behaves.  We observe a
% very fast convergence toward the minimizer, competing with
% state-of-the-art methods (Nesterov).

%\marginJC{Values of $s$ in two subsections?}

\subsection*{Optimization of Ill-Conditioned Quadratic Objective
  Function}
Consider the optimization of the objective function
$\map{f}{\real^2}{\real}$ defined by $f(x) = 10^{-2}x_1^2 + 10^2
x_2^2$. Note that $\mu = 2\cdot 10^{-2}$ and $L = 2\cdot 10^{2}$.  We
use $s = \mu/(36L^2)$ and initialize the velocity according
to~\eqref{eq:initial-state}. For DG$^\pb$, HOH$^\db$, and HOH$^\pb$,
we set $a=0.1$ and implement the event-triggered approach (at each
iteration, we employ a numerical zero-finding routine to explicitly
determine the stepsizes $ \pstep_{\ET}$, $\hstep^\db_{\ET}$, and
$\hstep^\pb_{\ET}$, respectively).

% We take implement an algorithm based on the sampled-hoh designed
% above, but we only decrease the stepsize if necessary, that is
% $r_i=0$ in Algorithm~\ref{algo:HOH}.

Figure~\ref{fig:stepsize:quadratic}(a) illustrates how the stepsize of
HOH$^\pb$ changes during the first $1000$ iterations. After the tuning
of the stepsize during the first iterations, it becomes quite steady
(likely due to the simplicity of quadratic functions) until the
trajectory approaches the minimizer. % We observe only a noticeable
% difference between the computed stepsize and the optimal one during
% the first iterations.
After $5$ iterations, the algorithm stepsize becomes almost equal to
the optimal stepsize.

\begin{figure}[htb]
  \centering
  \subfigure[]{\includegraphics[width=0.83\linewidth]{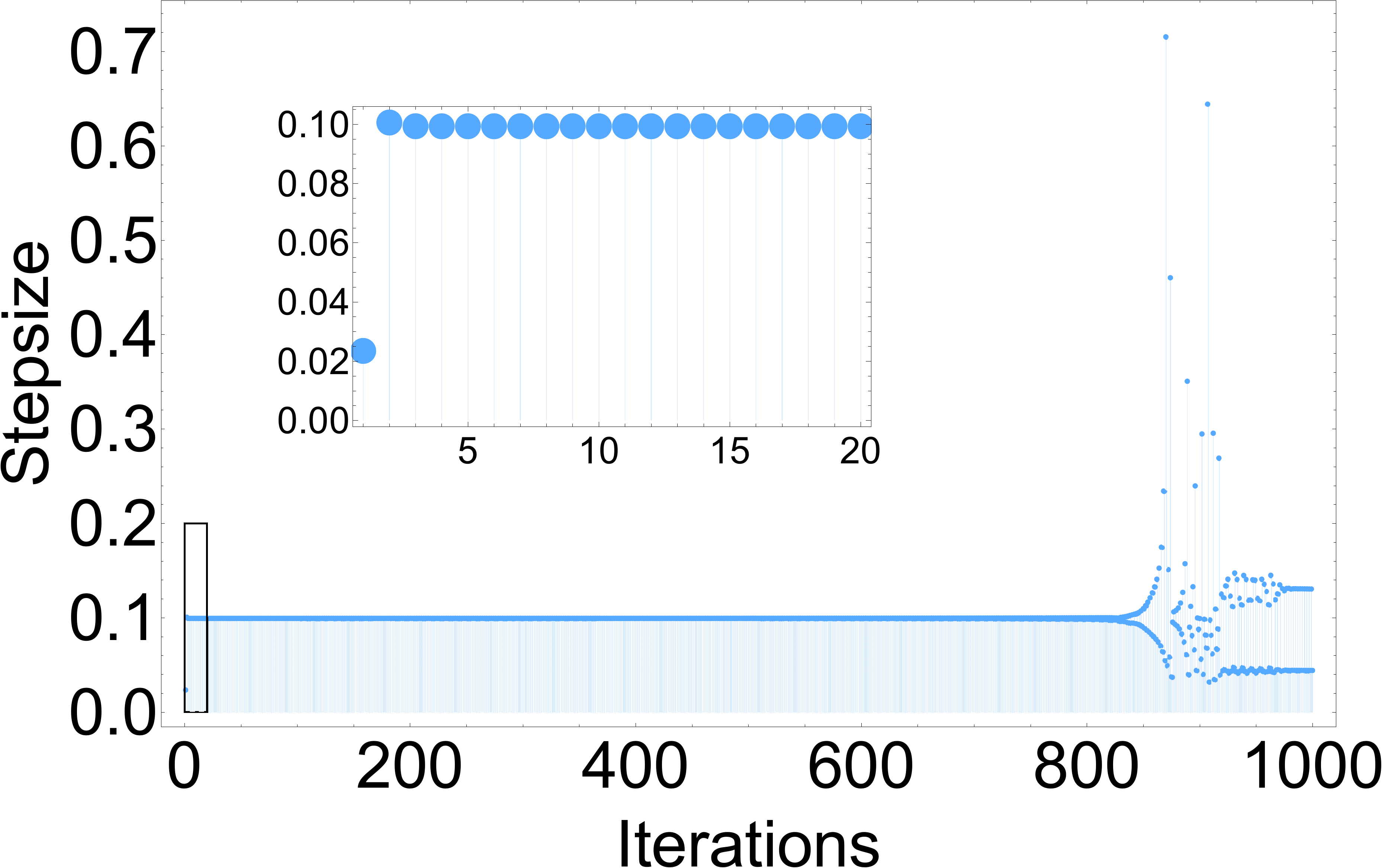}}
  \\
  \subfigure[]{\includegraphics[width=0.8\linewidth]{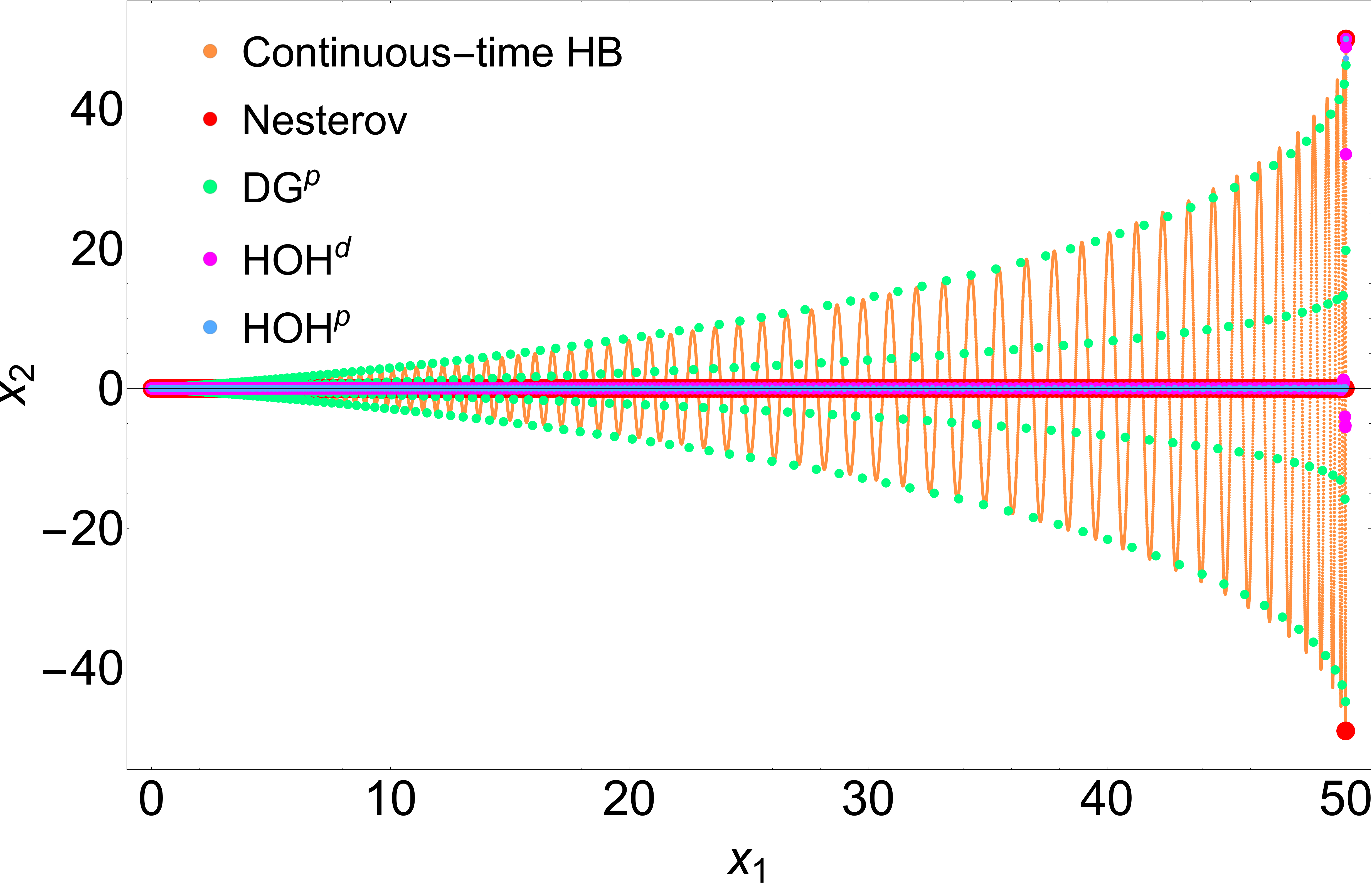}}
  \caption{Ill-conditioned quadratic objective function example. (a)
    Evolution of the stepsize along the execution of HOH$^\pb$ during
    the first $1000$ iterations. (b) State evolution along DG$^\pb$,
    HOH$^\db$, HOH$^\pb$, continuous heavy-ball dynamics, and
    Nesterov's method starting from $x = (50,50)$ and $v
      = (-0.0023, -4.7139)$.}\label{fig:stepsize:quadratic}
  \vspace*{-1ex}
\end{figure}

Figure~\ref{fig:stepsize:quadratic}(b) compares the performance of
DG$^\pb$, HOH$^\db$, and HOH$^\pb$ against the continuous heavy-ball
method and the discrete Nesterov method for strongly convex
functions. The DG$^\pb$ algorithm takes large stepsizes following the
evolution
% \footnote{video available at
% \url{https://www.dropbox.com/s/11r5if6bhejtgve/animate.avi?dl=0}.}
of the continuous heavy-ball along the straight lines $p(t) = p_k +
t\Xhb^a(p_k)$. Meanwhile, the higher-order nature of the hold employed
by HOH$^\db$ and HOH$^\pb$ makes them able to leap over the
oscillations, yielding a state evolution similar to Nesterov's method.
% \footnote{video available at
% \url{https://www.dropbox.com/s/5pf2rdjau68cph1/animate1.avi?dl=0}.}.
Figure~\ref{fig:stepsize:quadratic:bottom} shows the evolution of the
objective and Lyapunov functions. We observe that after some initial
iterations, HOH$^\pb$ outperforms Nesterov's method. Eventually, also
DG$^\pb$ catches up to Nesterov's method.
%We compare only with Nesterov's method for
% strongly-convex functions and heavy-ball, as these methods are
% state-of-the art first-order optimization algorithms.
\begin{figure}[htb]
  \centering
  \subfigure[]{\includegraphics[width=0.83\linewidth]{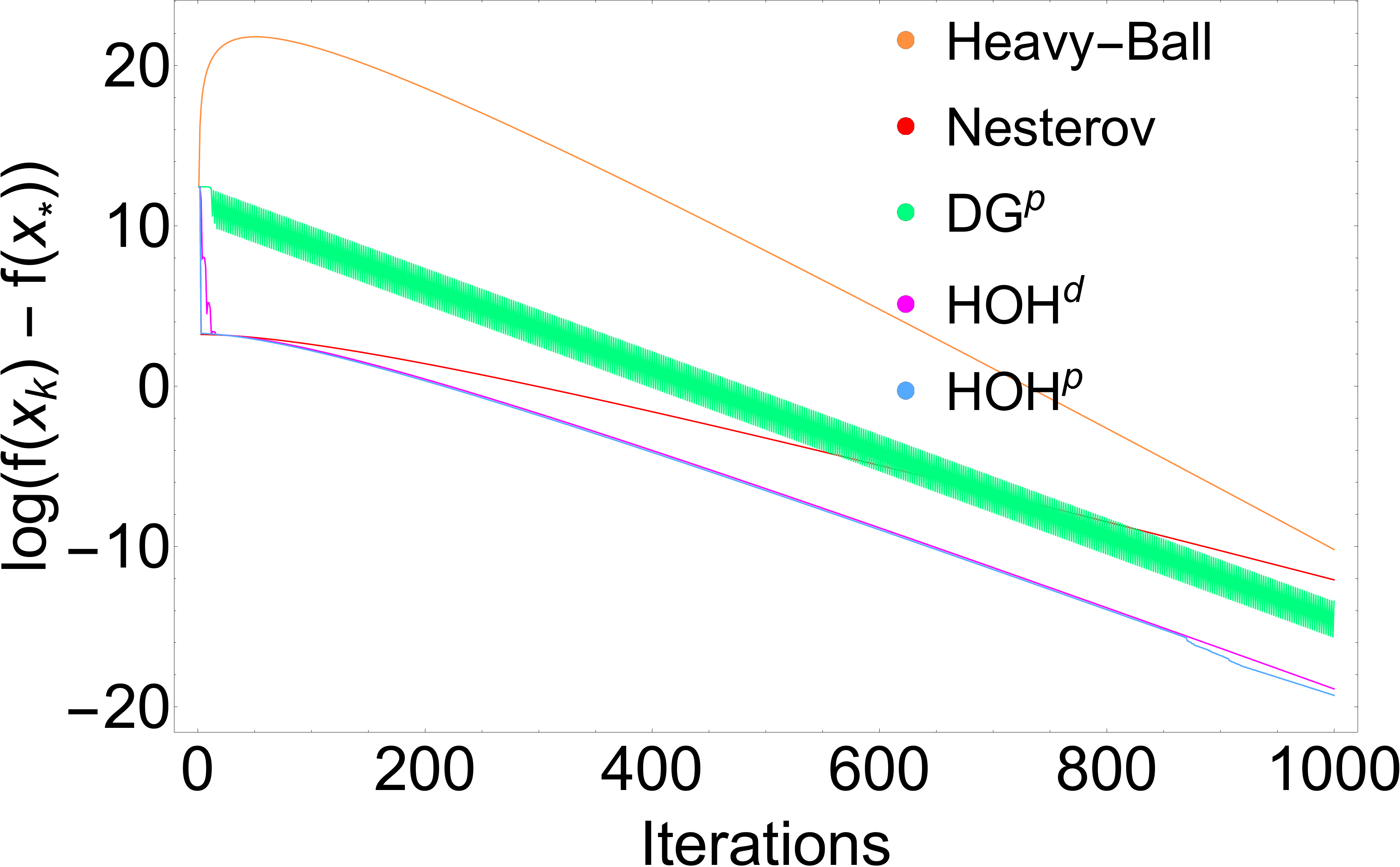}}
  \\
  \subfigure[]{
    \includegraphics[width=0.83\linewidth]{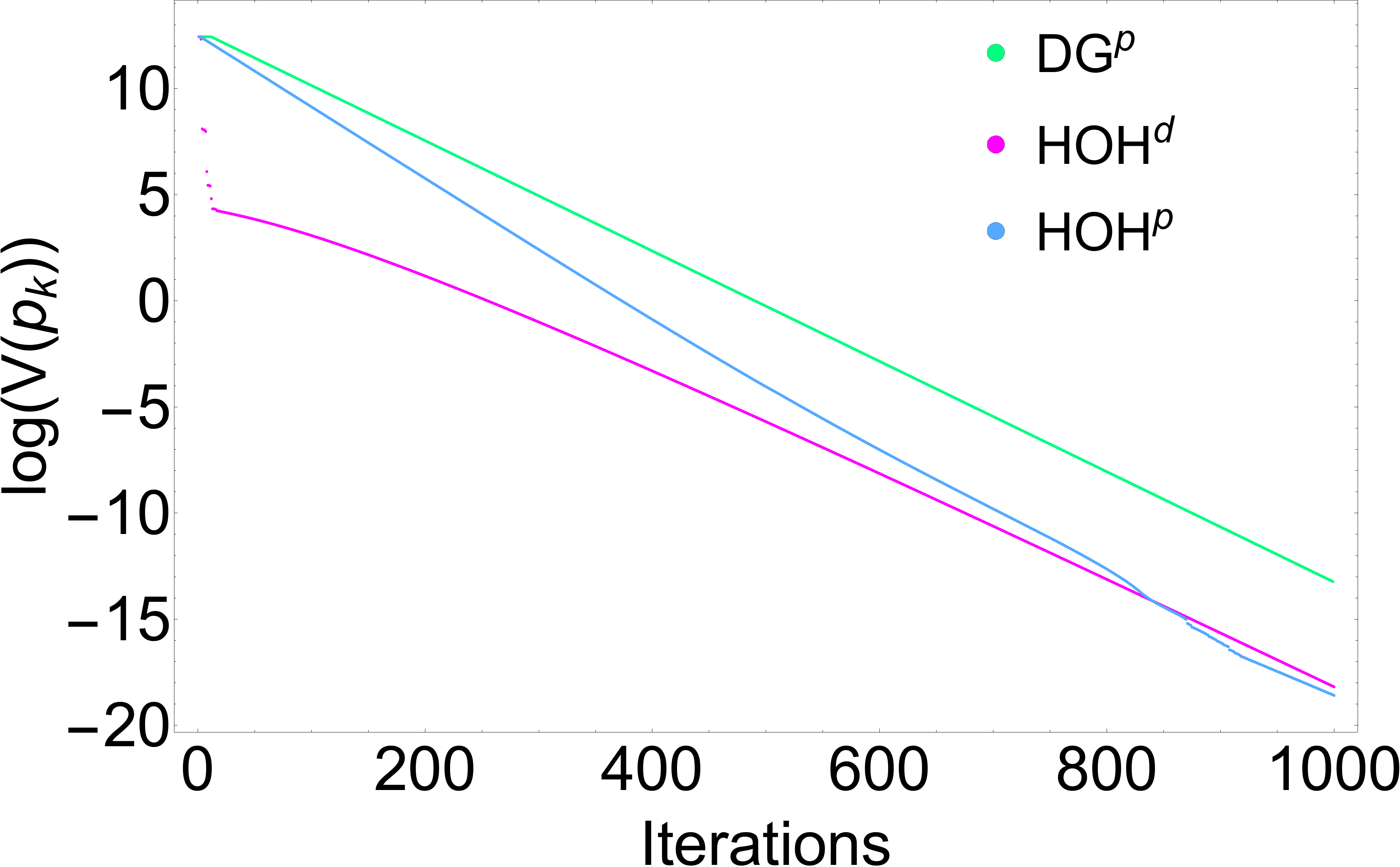}}
  \caption{Ill-conditioned quadratic objective function example. (a)
    Evolution of the logarithm of the objective function under
    DG$^\pb$, HOH$^\db$, HOH$^\pb$, the heavy-ball method, and
    Nesterov's method starting from $x = (50,50)$ and $v = (-0.0023,
    -4.7139)$. (b) Corresponding evolution of the logarithm of the
    Lyapunov function along DG$^\pb$, HOH$^\db$,
    and~HOH$^\pb$.}\label{fig:stepsize:quadratic:bottom}
  \vspace*{-1ex}
\end{figure}

\subsection*{Logarithmic Regression}
Consider the optimization of the regularized logistic regression cost
function $\map{f}{\real^4}{\real}$ defined by $f(x) =
\sum_{i=1}^{10}\log (1 + e^{-y_i\langle z_i, x\rangle}) +
\frac{1}{2}\norm{x}^2$, where the points $\{z_i\}_{i=1}^{10} \subset
\mathbb{R}^4$ are generated randomly using a uniform distribution in
the interval $[-5,5]$, and the points $\{y_i\}_{i=1}^{10} \subset
\{-1,1\}$ are generated similarly with quantized values.  This
objective function is $1$-strongly convex and one can also compute the
value $L=177.49$. We use $a = 0.025$ and $s = \mu/(36L^2)$, and
initialize the velocity according
to~\eqref{eq:initial-state}. Figure~\ref{fig:stepsize:logarithmic}(a)
and (b) show the evolution of the stepsize along HOH$^\pb$, which
changes as a function of the state looking to satisfy the desired
decay of the Lyapunov function.
\begin{figure}[htb]
  \centering
  \subfigure[]{\includegraphics[width=0.83\linewidth]{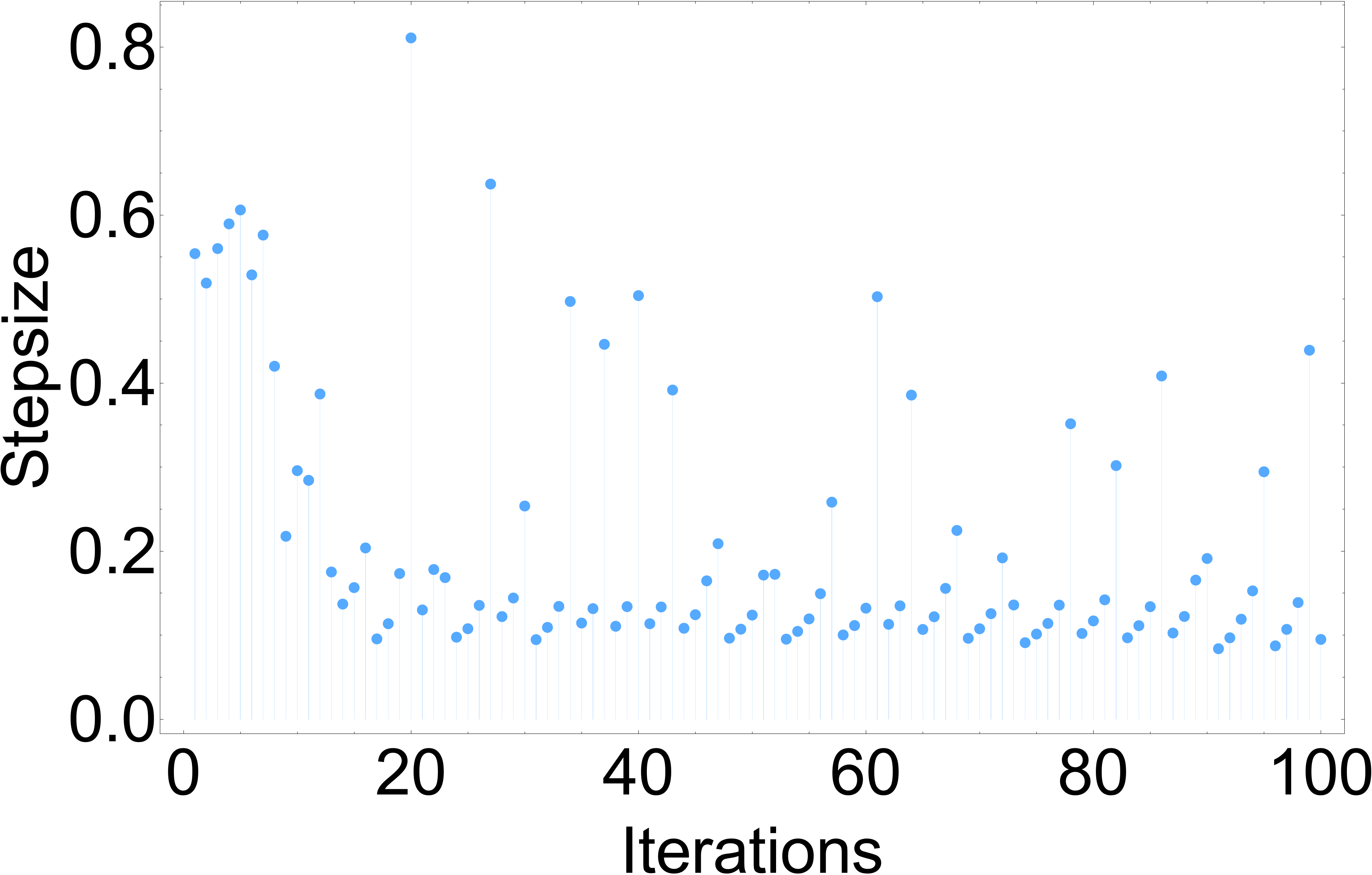}}
  \\
  \subfigure[]{
    \includegraphics[width=0.83\linewidth]{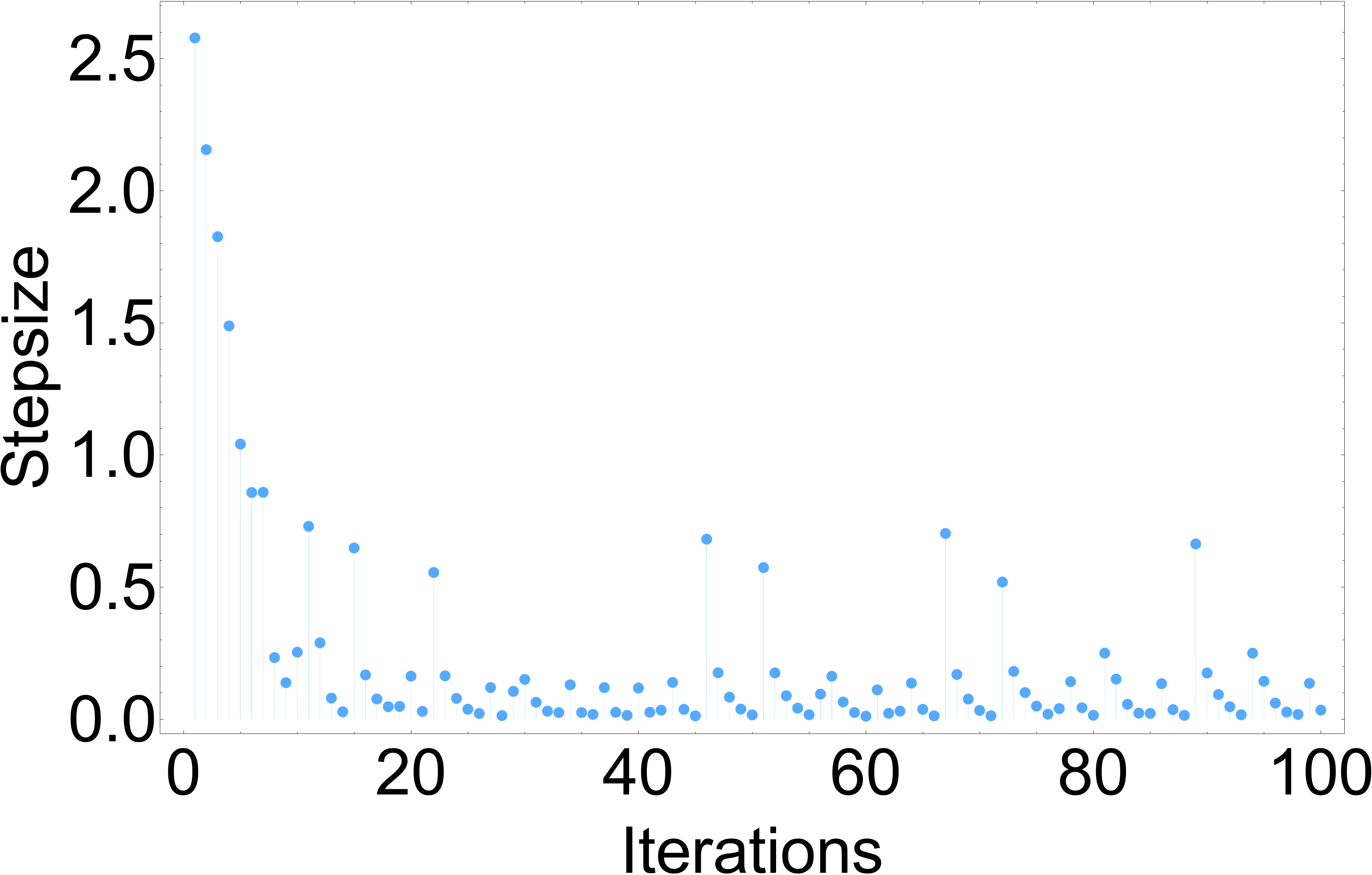}}
  \caption{Logarithmic regression example. (a) Evolution of the
    stepsize along the execution of HOH$^\pb$ starting from $x=
    (50,50, 50, 50)$ and $v = (-0.1026, -0.09265, -0.1078,
    -0.0899)$. Notice the complex pattern, with significant increases
    and oscillations along the trajectory. (b) Difference between the
    optimal stepsize (computed using the exact Lyapunov function,
    which assumes knowledge of the minimizer) and the stepsize of
    HOH$^\pb$.  The largest difference is achieved at the beginning:
    after a few iterations, the difference descreases significantly,
    periodically becoming almost
    zero.}\label{fig:stepsize:logarithmic}
\vspace*{-1ex}
\end{figure}
Figure~\ref{figure:zoh} shows the evolution of the objective and
Lyapunov functions. We observe how HOH$^\db$ and HOH$^\pb$ outperform
Nesterov's method, although eventually the heavy-ball algorithm
performs the best.  The Lyapunov function decreases at a much faster
rate along HOH$^\db$ and HOH$^\pb$ than along DG$^\pb$.

\begin{figure}[htb]
  \centering
  \subfigure[]{\includegraphics[width=0.83\linewidth]{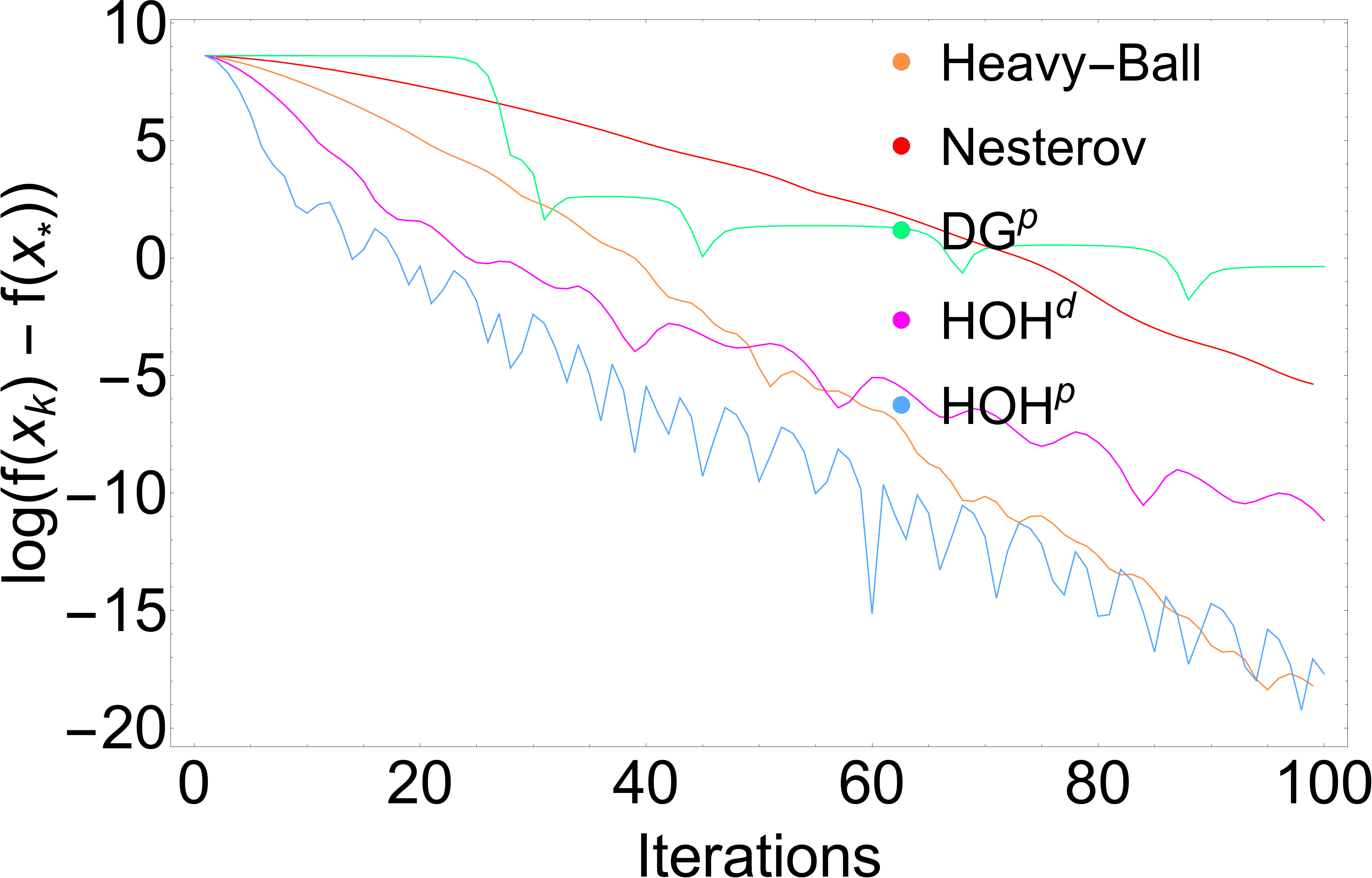}}
  \\
  \subfigure[]{
    \includegraphics[width=0.83\linewidth]{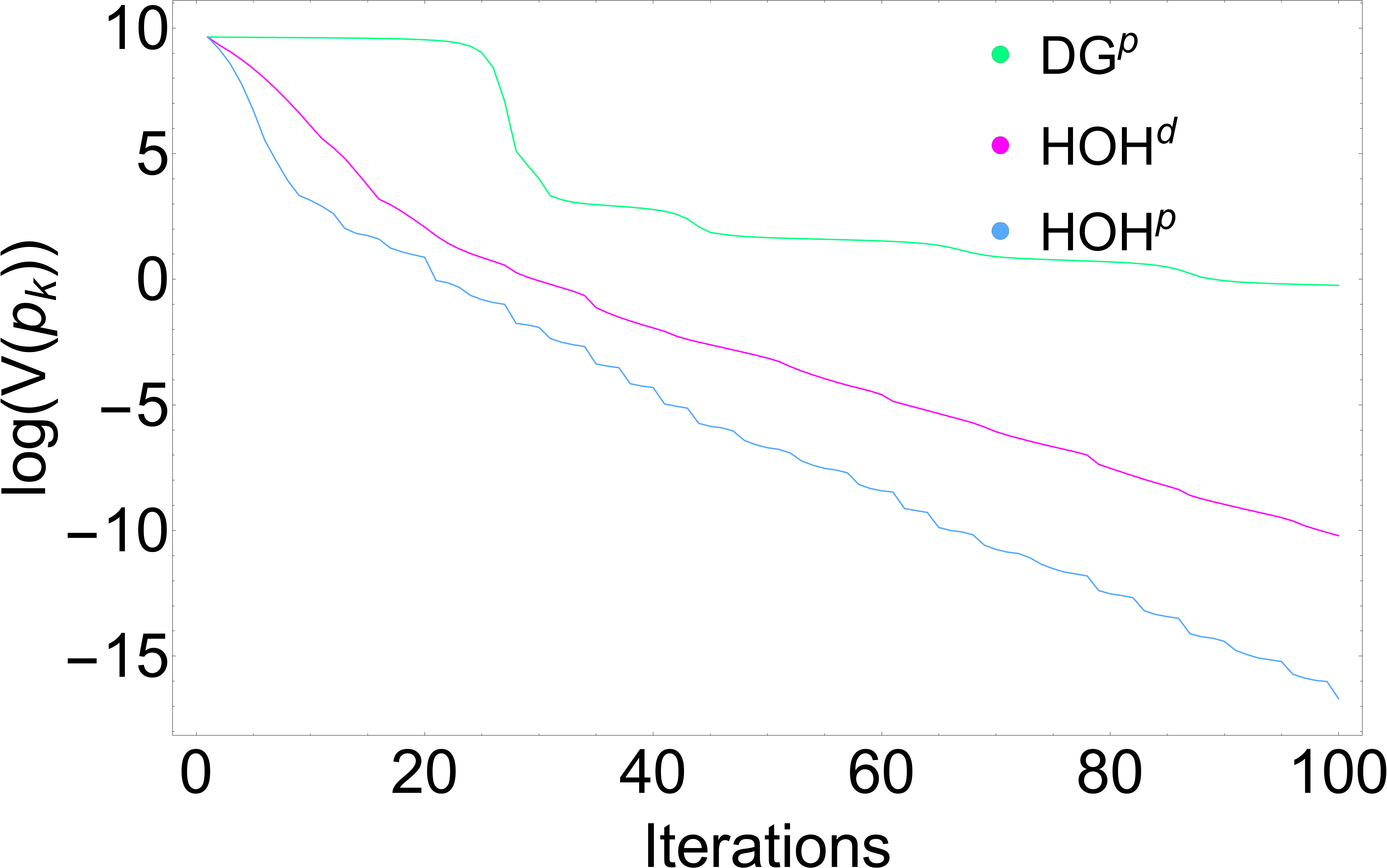}}
  \caption{Logarithmic regression example. (a) Evolution of the
    logarithm of the objective function under DG$^\pb$, HOH$^\db$,
    HOH$^\pb$, the heavy-ball method, and Nesterov's method starting
    from $x= (50,50, 50,50)$ and $v = (-0.1026, -0.09265, -0.1078,
    -0.0899)$.  (b) Corresponding evolution of the logarithm of the
  Lyapunov function along DG$^\pb$, HOH$^\db$, and
  HOH$^\pb$.}\label{figure:zoh} \vspace*{-1ex}
\end{figure} 

\section{Conclusions}

We have introduced a resource-aware control framework to the
discretization of accelerated optimization flows that specifically
takes advantage of their dynamical properties.  We have exploited
fundamental concepts from opportunistic state-triggering related to
the various ways of encoding the notion of valid Lyapunov
certificates, the use of sampled-data information, and the
construction of state estimators and holders to synthesize
variable-stepsize optimization algorithms that retain by design the
convergence properties of their continuous-time counterparts.
% (i) larger stepsize while conserving Lyapunov decay using other
% triggered conditions, like performance-based ones; (ii) better
% approximations of the gradient term along the dynamics, using
% efficient sampling designs; and (iii) high-performance integrators
% using more accurate surrogates of the original dynamics.
We believe these results open the way to a number of exciting research
directions.  Among them, we highlight the design of adaptive learning
schemes to refine the use of sampled data and optimize the algorithm
performance with regards to the objective function, the
characterization of accelerated convergence rates, employing tools and
insights from hybrid systems for analysis and design, enriching the
proposed designs by incorporating re-start schemes as triggering
conditions to avoid overshooting and oscillations, and developing
distributed implementations for network optimization problems.
% Finally, we would like to stress that our constructions are by no
% means limited to acceleration and can find applications in the
% discretization of any globally asymptotically stable continuous
% dynamical system endowed with a Lyapunov certificate.

% \marginJC{Why is the appearence of [21] and [25] so different if they
%   appear in the same conference? Also, missing conference location in
%   [11] and [20].  Check for other inconsistencies}

\bibliography{alias,Main,Main-add,JC}
\bibliographystyle{IEEEtran}

\appendices
\setcounter{equation}{0}
\renewcommand{\theequation}{\thesection.\arabic{equation}}

% \section{Basic facts and auxiliary result}\label{app:appendix}
\section{}\label{app:appendix}

Throughout the appendix, we make use of a number of basic facts that
we gather here for convenience,
\begin{subequations}\label{eq:aux-facts}
  \begin{align}
    f(x_*) - f(x) & \leq - \frac{\norm{\nabla
        f(x)}^2}{2L} \label{eq:aux-a}
    \\
    \frac{ \norm{\nabla f(x)}}{L} \leq \norm{x - x_*} & \leq
    \frac{\norm{\nabla f(x)}}{\mu} \label{eq:aux-c}
    \\
    f(y) - f(x) - \langle \nabla f(x),y-x\rangle & \leq 
    \frac{L}{2}\norm{y-x}^2 \label{eq:aux-d}
    \\
    \frac{1}{L}\norm{\nabla f(x) - \nabla f(y)}^2 & \le \langle \nabla
    f(x) - \nabla f(y), x - y \rangle \label{eq:aux-e}
    \\
    f(y) - f(x) - \langle \nabla f(x),y-x\rangle & \leq
    \frac{1}{2\mu}\norm{\nabla f(y) - \nabla f(x)}^2 \label{eq:aux-f}
  \end{align}
\end{subequations}
We also resort at various points to the expression of the gradient of
$V$,
\begin{align}\label{eq:gradient-V}
  \nabla V(p) & =
    \begin{bmatrix}
     \oneplusmus \nabla f(x) +\sqrt{\mu}v+2\mu(x-x_*)
      \\
      v+\sqrt{\mu}(x - x_*)
    \end{bmatrix} .
\end{align}

The following result is used in the proof of
Theorem~\ref{th:continuous-sampled}.

\begin{lemma}\label{lemma:bound}
  For $\beta_1,\dots,\beta_4 >0$, the function
  \begin{align}
    \label{eq:g}
    g(z) = \frac{\beta_3 + \beta_4z^2}{-\beta_1 + \beta_2 z}
  \end{align}
  is positively lower bounded on $(\beta_1/\beta_2,\infty)$.
\end{lemma}
\begin{proof}
  The derivative of $g$ is
  \[
  g'(z) = \frac{-\beta_2 \beta_3 + \beta_4 z (-2 \beta_1 +
    \beta_2z)}{(\beta_1 - \beta_2 z)^2} .
  \]
  The solutions to $g'(z) = 0$ are then given by
  \begin{align}\label{eq:zroot}
    z_{\root}^\pm = \frac{\beta_1\beta_4 \pm
      \sqrt{\beta_2^2\beta_3\beta_4 +
        \beta_1^2\beta_4^2}}{\beta_2\beta_4}.
  \end{align}
  Note that $ z^-_{\root} < 0 < \beta_1/\beta_2 <z^+_{\root}$, $g'$ is
  negative on $(z^-_{\root}, z^+_{\root})$, and positive on
  $(z^+_{\root},\infty)$. Therefore the minimum value over $
  (\beta_1/\beta_2,\infty)$ is achieved at $z^+_{\root}$, and
  corresponds to $g(z^+_{\root})>0$.
\end{proof}

% \section{Proofs of results in
%   Section~\ref{sec:performance-based}}\label{app:performance-based}

\begin{proof}[Proof of Proposition~\ref{prop:upper-bound-derivative}]
  We break out $ \frac{d}{dt} V(p(t)) +\frac{\sqrt{\mu}}{4} V({p}(t))$
  as follows
  \begin{align*}
    &\frac{d}{dt}V(\hat{p} + t \Xhb^a(\hat{p})) + \frac{\sqrt{\mu}}{4}
    V(\hat{p}+t\Xhb^a(\hat{p})) =
    \\
    &= \underbrace{\langle \nabla V(\hat{p}),\Xhb^a(\hat{p})\rangle +
      \frac{\sqrt{\mu}}{4} V(\hat{p})}_{\textrm{Term I + II + III}}
    \\
    & \quad + \underbrace{\langle \nabla V(\hat{p} + t\Xhb^a(\hat{p})) -\nabla
      V(\hat p),\Xhb^a(\hat{p}) \rangle }_{\textrm{Term IV + V}}
    \\
    & \quad + \frac{\sqrt{\mu}}{4}(\underbrace{V(\hat{p} + t\Xhb^a(\hat{p})) -
      V(\hat{p})}_{\textrm{Term VI}}),
  \end{align*}
  and bound each term separately.
  
  $\mathbf{ Term \ I + II + III }$. % Note that we have bounded this term
  % by~\eqref{eq:a_condition} in the proof of
  % Theorem~\ref{th:continuous-sampled}.  Here, we provide a different
  % bound that is more convenient for the algorithmic implementation of
  % the variable-stepsize integrator.
  % We would like to mention that is possible to design different
  % discretization procedures along the lines of the ones introduced here
  % based on different bounds the expression $\dot{V} +
  % \frac{\sqrt{\mu}}{4}V$. We present here an alternative bound of Term I
  % + II + III that can be used to build algorithms similar to
  % Algorithm~\ref{algo:ADG} and the obtainment of the corresponding
  % convergence results. We used this bound in the
  % Section~\ref{sec:simulation} and observe nice convergence in
  % simulations. A theoretical study comparing the algorithms
  % corresponding to different bounds is out of the scope of this paper.
  % 
  From the definition~\eqref{eq:continuous-hb-lyapunov} of $V$ and the
  fact that $\norm{y_1 + y_2}^2 \leq 2\norm{y_1}^2 + 2\norm{y_2}^2$,
  we have
  \begin{align*}
    V(\hat{p}) & =\oneplusmus(f(\hat{x}) - f(x_*))
    +\frac{1}{4}\norm{\hat{v}}^2
    \\
    & \quad + \frac{1}{4}\norm{\hat v + 2\sqrt{\mu}(\hat{x} - x_*)}^2
    \\
    & \leq\oneplusmus(f(\hat{x}) - f(x_*))
    \\
    & \quad +\frac{1}{4}\norm{\hat{v}}^2 + \frac{2}{4}\norm{\hat v}^2 +
    \frac{2}{4}\norm{2\sqrt{\mu}(\hat{x} - x_*)}^2
    \\
    & = \oneplusmus (f(\hat{x}) - f(x_*))
    +\frac{3}{4}\norm{\hat{v}}^2 + 2\mu\norm{\hat{x} - x_*}^2.
  \end{align*}
  Using this bound, we obtain
  \begin{align*}
    & \langle \nabla V(\hat{p}), X^a_{\operatorname{hb}}(\hat{p})
    \rangle + \frac{\sqrt{\mu}}{4} V(\hat{p})
    \\
    & \leq -\sqrt{\mu}\norm{\hat{v}}^2 + \frac{\sqrt{\mu}}{4}
    \oneplusmus(f(\hat{x}) - f(x_*)) +
    \frac{3\sqrt{\mu}}{16}\norm{\hat{v}}^2
    \\
    & \quad + \frac{\mu\sqrt{\mu}}{2}\norm{\hat{x} - x_*}^2
    +\oneplusmus\langle \nabla f(\hat{x}) - \nabla f(\hat{x} +
    a\hat{v}),\hat{v} \rangle
    \\
    & \quad - \sqrt{\mu}\oneplusmus \langle \nabla f(\hat{x} +
    a\hat{v}),\hat{x} - x_* \rangle.
  \end{align*}
  Writing $0$ as $0= a \hat{v}-a \hat{v}$ and using strong convexity,
  we can upper bound $\langle \nabla f(\hat{x} + a\hat{v}),
  x_*  - \hat{x}\rangle$ in the last summand by the expression
  \begin{align*}
    f(x_*) - f(\hat{x} + a\hat{v}) - \frac{\mu}{2} \norm{\hat{x} +
      a\hat{v} - x_*}^2 + \langle \nabla f(\hat{x} +a\hat{v}),a\hat{v}
    \rangle.
  \end{align*}
  Substituting this bound above and re-grouping terms,
  \begin{align*}
    & \langle \nabla V(\hat{p}), X^a_{\operatorname{hb}}(\hat{p})
    \rangle + \frac{\sqrt{\mu}}{4} V(\hat{p}) \leq
    -\sqrt{\mu}\norm{\hat{v}}^2
    \\
    & \quad + \underbrace{\sqrt{\mu}\oneplusmus \Big(
      \frac{1}{4} (f(\hat{x}) - f(x_*)) + f(x_*) - f(\hat{x} +
      a\hat{v})\Big)}_{\textrm{(a)}}
    \\
    & \quad +\frac{3\sqrt{\mu}}{16}\norm{\hat{v}}^2 +\oneplusmus\langle \nabla f(\hat{x}) - \nabla f(\hat{x} +
    a\hat{v}),\hat{v} \rangle
    \\
    & \quad \underbrace{+ \frac{\mu\sqrt{\mu}}{2}\norm{\hat{x} -
        x_*}^2+ \sqrt{\mu}\oneplusmus(-\frac{\mu}{2}\norm{\hat{x} + a\hat{v} -
        x_*}^2)}_{\textrm{(b)}}
    \\
    & \quad +\sqrt{\mu}\oneplusmus\langle \nabla f(\hat{x}
    +a\hat{v}),a\hat{v} \rangle.
  \end{align*}
  Observe that
  \begin{align*}
    \textrm{(a)} & = \sqrt{\mu} \oneplusmus\big(- \frac{3}{4}
    (f(\hat{x}) - f(x_*)) + f(\hat{x}) - f(\hat{x} +a \hat{v})\big) ,
    \\
    \textrm{(b)} & \leq - \frac{\mu^2\sqrt{s}}{2}\norm{\hat{x} -
      x_*}^2 +\oneplusmus \mu^{3/2}\norm{\hat{x} - x_*}\norm{a
      \hat {v}}
    \\
    & \quad - \oneplusmus\mu^{3/2}/2\norm{a\hat{v}}^2 ,
  \end{align*}
  where, in the expression of~(a), we have expressed $0$ as $0 =
  3/4(f(\hat{x})- f(\hat{x})) $ and, in the expression of~(b), we have
  expanded the square and used the Cauchy-Schwartz
  inequality~\cite{SL:93}.  Finally, resorting
  to~\eqref{eq:aux-facts}, we obtain
  \begin{align*}
    & \langle \nabla V(\hat{p}), X^a_{\operatorname{hb}}(\hat{p})
    \rangle + \frac{\sqrt{\mu}}{4} V(\hat{p}) \leq C_{\ET}(\hat p;a) =
    C_{\ST} (\hat p;a) .
    % -\frac{13\sqrt{\mu}}{16}\norm{\hat{v}}^2 + (1+\sqrt{\mu
    % s})\frac{-3\sqrt{\mu}}{4}\frac{\norm{\nabla f(\hat{x})}^2}{2L}
    % \\
    % & \quad +(1 + \sqrt{\mu s})\sqrt{\mu}(f(x) - f(\hat{x} + a\hat{v}))
    % \\
    % & \quad + (\frac{\sqrt{\mu}\mu}{2} - \sqrt{\mu} \mu /2(1+\sqrt{\mu
    % s}))\displaystyle\frac{\norm{\nabla f(\hat{x})}^2}{L^2}
    % \\
    % & \quad + \sqrt{\mu}(1+\sqrt{\mu s})\norm{\nabla f(\hat{x})}\norm{a
    % \hat{v}}
    % \\
    % & \quad -\sqrt{\mu}(1+\sqrt{\mu s})\mu/2\norm{a\hat{v}}^2
    % \\
    % & \quad -(1+\sqrt{\mu s})\langle \nabla f(\hat{x} + a\hat{v}) -
    % \nabla f (\hat{x}) , \hat{v}\rangle
    % \\
    % & \quad +\sqrt{\mu}(1+\sqrt{\mu s})\langle \nabla f(\hat{x} +
    % a\hat{v}), a\hat{v} \rangle .
  \end{align*}
  
  %%%%%%%%%%%%%%%%%%%%%%%%%%%%%%%%%%%%%%%%% 
  $\bullet$ $\mathbf{ Term \ IV + V }$. Using~\eqref{eq:gradient-V} we have
  \begin{align*}
    & \nabla V(\hat{p}+t\Xhb^a(\hat{p})) =
    \\
    &
    \begin{bmatrix}
      \oneplusmus\nabla f(\hat{x} +t\hat{v})
      +\sqrt{\mu}\hat{v}- 2\mu t\hat  v  \\
      -t\sqrt{\mu}\oneplusmus\nabla f(\hat{x} + a\hat{v})
      +2\mu(\hat{x} + t\hat v -x_*)
      \\
      \noalign{\medskip} \hat{v} - 2t\sqrt{\mu} \hat{v} - t\oneplusmus\nabla f(\hat{x} +a\hat{v}) +\sqrt{\mu}(\hat{x} + t\hat v-
      x_*)
    \end{bmatrix}.  
  \end{align*}
  Therefore, $\nabla V(\hat{p} + t\Xhb^a(\hat{p})) -\nabla V(\hat{p})$
  reads
  \begin{align*}
    \begin{bmatrix}
     \oneplusmus(\nabla f(\hat x \!+\! t \hat v) \!-\! \nabla f(\hat x))
      \!-\!t\sqrt{\mu}\oneplusmus\nabla f(\hat{x} \!+\! a\hat{v})
      \\
      \noalign{\medskip} -\sqrt{\mu}t \hat v - t\oneplusmus\nabla
      f(\hat x + a \hat v)
    \end{bmatrix}
  \end{align*}
  and hence
  \begin{align*}
    & \langle \nabla V(\hat{p} + t\Xhb^a(\hat{p})) -\nabla
    V(\hat{p}),\Xhb^a(\hat{p}) \rangle
    \\
    & = \oneplusmus\langle \nabla f(\hat{x}+t\hat{v}) -\nabla
    f(\hat{x}), \hat v \rangle
    \\
    & \quad + 2t \sqrt{\mu}\oneplusmus\langle \nabla f(\hat{x} +a
    \hat{v}),\hat{v}\rangle + 2t\mu\norm{\hat{v}}^2
    \\
    & \quad + t\mu_s\norm{\nabla f(\hat{x} +a\hat{v})}^2.
  \end{align*}
  The RHS of the last expression is precisely $A_{\ET}(\hat p,t;a)$.
  Using the $L$-Lipschitzness of $\nabla f$, one can see that
  $A_{\ET}(\hat p,t;a) \le A_{\ST}(p;a) t$.
  
  %%%%%%%%%%%%%%%%%%%%%%%%%%%%%%%%%%%%%%%%%%%%%% SECOND TERM
  $\bullet$ $\mathbf{ Term \ VI
  }$. From~\eqref{eq:continuous-hb-lyapunov},
  \begin{align*}
    & V(\hat{p}+t\Xhb^a(\hat{p})) - V(\hat{p}) =\oneplusmus(f(\hat{x} + t\hat{v}) - f(x_*))
    \\
    & \quad + \frac{1}{4}\norm{\hat{v}- 2t\sqrt{\mu}\hat{v}
      -t\oneplusmus\nabla f(\hat{x} + a\hat{v})}^2
    \\
    & \quad + \frac{1}{4}\left\lVert \hat{v} \right. -
    2t\sqrt{\mu}\hat{v} -t\oneplusmus\nabla f(\hat{x} +a\hat{v})
    \\
    & \quad + \left. 2\sqrt{\mu}(\hat{x} + t\hat{v} - x_*)\right\rVert^2
    - \oneplusmus(f(\hat{x}) - f(x_*))
    \\
    & \quad - \frac{1}{4}\norm{\hat{v}}^2 -\frac{1}{4}\norm{\hat{v}
      +2\sqrt{\mu}(\hat{x} - x_*)}^2.
  \end{align*}
  Expanding the squares in the second and third summands, and
  simplifying, we obtain
  \begin{align*}
    & V(\hat{p}+t\Xhb^a(\hat{p})) - V(\hat{p}) = \oneplusmus(f(\hat{x} + t\hat{v}) - f(\hat{x}))
    \\
    & \quad + \frac{1}{4}\norm{-2t\sqrt{\mu}\hat{v} -t\oneplusmus\nabla f(\hat{x} + a\hat{v})}^2
    \\
    & \quad + \frac{1}{2}\langle \hat{v}, -2t\sqrt{\mu}\hat{v}
    -t\oneplusmus\nabla f(\hat{x} +a\hat{v})\rangle
    \\
    & \quad + \frac{1}{4}\norm{-t\oneplusmus\nabla f(\hat{x}
      +a\hat{v})}^2
    \\
    & \quad +\frac{1}{2}\langle \hat{v} +2\sqrt{\mu}(\hat{x} - x_*),
    -t(\oneplusmus\nabla f(\hat{x} +a\hat{v})\rangle
    \\
    & = \oneplusmus(f(\hat{x} + t\hat{v}) - f(\hat{x}))
    \\
    & \quad + \frac{1}{4}\norm{-2t\sqrt{\mu}\hat{v} -t\oneplusmus\nabla f(\hat{x} + a\hat{v})}^2
    \\
    & \quad -t\sqrt{\mu}\norm{\hat{v}}^2 - t \oneplusmus\langle
    \hat{v} ,\nabla f(\hat{x} + a\hat{v}) \rangle
    \\
    & \quad + \frac{1}{4}\norm{-t\oneplusmus\nabla f(\hat{x}
      +a\hat{v})}^2
    \\
    & \quad +\langle \sqrt{\mu}(\hat{x} - x_*), -t\oneplusmus\nabla f(\hat{x} +a\hat{v})\rangle.
  \end{align*}
  Note that
  \begin{align*}
    & \langle x_* - \hat{x},\nabla f(\hat{x} + a\hat{v}) \rangle
    \\
    & = \langle x_* - \hat{x} - av,\nabla f(\hat{x} + a\hat{v}) \rangle
    + \langle a\hat{v},\nabla f(\hat{x} + a\hat{v}) \rangle
    \\
    & \leq - \frac{\norm{\nabla f(\hat{x} +a\hat{v} )}^2}{L} + \langle
    a\hat{v},\nabla f(\hat{x} + a\hat{v}) \rangle ,
  \end{align*}
  where in the inequality we have used~\eqref{eq:aux-e} with $x = \hat
  x + a \hat v$ and $y = x_*$. Using this in the equation above, one
  identifies the expression of $B_{\ET} (p,t;a)$.  Finally,
  applying~\eqref{eq:aux-d}, one can show that $B_{\ET}(p,t;a) \leq
  B_{\ST}^l(p;a)t + B^q_{\ST}(p;a)t^2$, concluding the proof.
\end{proof}

\begin{proof}[Proof of
  Proposition~\ref{prop:upper-bound-derivative-hoh}]
  For convenience, let
  \begin{align*}
    \Xhb^{a,\hat{p}} (p) & =
    \begin{bmatrix}
      v
      \\
      - 2\sqrt{\mu}v - \oneplusmus\nabla f(\hat{x} + a\hat{v})
    \end{bmatrix},
  \end{align*}
  where $\hat p = [\hat x, \hat v]$. We next provide a bound for the
  expression
  \begin{align*}
    &\frac{d}{dt}V(p(t))) + \frac{\sqrt{\mu}}{4} V(p(t)) =
    \underbrace{\langle \nabla
      V(\hat{p}),\Xhb^{a,\hat{p}}(\hat{p})\rangle +
      \frac{\sqrt{\mu}}{4} V(\hat{p})}_{\textrm{Term I + II + III}}
    \\
    & \quad + \underbrace{\langle \nabla V(p(t)) -\nabla
      V(\hat{p}),\Xhb^{a,\hat{p}}(p(t)) \rangle }_{\textrm{Term IV}}
    \\
    & \quad + \underbrace{\langle \nabla V(\hat{p}),
      \Xhb^{a,\hat{p}}(p(t))- \Xhb^{a,\hat{p}}(\hat{p})
      \rangle}_{\textrm{Term V}} +
    \frac{\sqrt{\mu}}{4}\underbrace{(V(p(t)) -
      V(\hat{p}))}_{\textrm{Term VI}}.
  \end{align*}
  Next, we bound each term separately.
  
  %%%%%%%%%%% TERM I + II + III
  $\bullet$ $\mathbf{ Term \ I + II + III }$. Since $ \Xhb^{a,\hat{p}}
  (\hat{p})= \Xhb^a(\hat{p})$, this term is exactly the same as Term~I
  + II + III in the proof of
  Proposition~\ref{prop:upper-bound-derivative}, and hence the bound
  obtained there is valid.

  %%%%%%%%%%% TERM IV
  $\bullet$ $\mathbf{ Term \ IV }$. 
  Using~\eqref{eq:gradient-V}, we have
  \begin{align*}
    & \langle \nabla V(p(t)) - \nabla V(\hat{p}),
    \Xhb^{a,\hat{p}}(p(t)) \rangle
    \\
    & = \oneplusmus\langle \nabla f(x(t)) - \nabla f(\hat{x}),v(t)
    \rangle
    \\
    & \quad + \sqrt{\mu}\langle v(t) - \hat{v} ,v(t) \rangle + 2\mu
    \langle x(t) - \hat{x},v(t) \rangle
    \\
    & \quad - 2\sqrt{\mu}\langle v(t) - \hat{v},v(t)\rangle
    - \oneplusmus \langle v(t) - \hat{v},\nabla f(\hat{x} +
    a\hat{v}) \rangle
    \\
    & \quad -2\mu \langle x(t) - \hat{x},v(t) \rangle
 - \oneplusmus \sqrt{\mu}\langle x(t) - \hat{x}, \nabla
    f(\hat{x} + a\hat{v}) \rangle
    \\
    & = \oneplusmus\langle \nabla f(x(t)) - \nabla f(\hat{x}),v(t)
    \rangle
 -\sqrt{\mu}\langle v(t) - \hat{v} ,v(t) \rangle
    \\
    & \quad -\oneplusmus \langle v(t) - \hat{v},\nabla f(\hat{x} +
    a\hat{v}) \rangle
    \\
    & \quad - \oneplusmus \sqrt{\mu}\langle x(t) - \hat{x}, \nabla
    f(\hat{x} + a\hat{v}) \rangle,
  \end{align*}
  from where we obtain $\mathrm{ Term \ IV } \le \fA_{\ET} (\hat
  p,t;a)$. Now, using $0= \hat v - \hat v$, the $L$-Lipschitzness of
  $\nabla f$, and the Cauchy-Schwartz inequality, we have
  \begin{align*}
    | \fA_{\ET} (\hat p,t;a) | & \le \oneplusmus
  % \norm{\nabla f(x(t)) - \nabla f(\hat{x})}
    L \norm{x(t) -\hat{x}} ( \norm{v(t) - \hat{v}} + \norm{\hat{v}})
    \\
    & \quad + \sqrt{\mu}\norm{v(t) - \hat{v}}^2 + \sqrt{\mu}\norm{v(t)
      - \hat{v}}\norm{\hat{v}}
    \\
    & \quad + \oneplusmus\norm{v(t) - \hat{v}}\norm{\nabla f(\hat{x} +
      a \hat{v})}
    \\
    & \quad + \oneplusmus\sqrt{\mu}\norm{x(t) - \hat{x}}\norm{\nabla
      f(\hat{x} + a \hat{v})}.
  \end{align*}
  % \begin{subequations}
  %   \begin{align*}
  %     x(t) - \hat{x} & =(1 - e^{-2\sqrt{\mu} t})\frac{ (1 + \sqrt{\mu
  %     s}) \nabla f(\hat{x} + a \hat{v}) + 2\sqrt{\mu}\hat{v}}{4\mu}
  %     \nonumber
  %     \\
  %     & \quad -\frac{t(1 + \sqrt{\mu s})\nabla f(\hat{x} + a
  %     \hat{v})}{2\sqrt{\mu}}
  %     \\
  %     v(t) - \hat{v} & = \frac{e^{-2\sqrt{\mu}t} - 1}{2} \Big(2\hat{v} + (1 +
  %     \sqrt{\mu s})\frac{\nabla f(\hat{x} + a \hat{v})}{\sqrt{\mu}}\Big) .
  %   \end{align*}
  % \end{subequations}
  Using~\eqref{eq:hoh-flow}, the triangle inequality, and $1 -
  e^{-2\sqrt{\mu}t}\leq 2\sqrt{\mu}t$, we can write
  \begin{subequations}\label{differencexv}
    \begin{align}
      \norm{x(t) - \hat x}& \leq \frac{t}{2\sqrt{\mu}}\norm{2
        \sqrt{\mu}\hat v + \oneplusmus\nabla f(\hat x + a \hat v) }
      \nonumber
      \\
      & \quad + \frac{\oneplusmus t}{2\sqrt{\mu}}\norm{\nabla f(\hat x
        + a \hat v)}, \label{eq:aux-diffx}
      \\
      \norm{v(t) - \hat v} &\leq t\norm{2 \sqrt{\mu}\hat v
        +\oneplusmus\nabla f(\hat x + a \hat v)}. \label{eq:aux-diffv}
    \end{align}
  \end{subequations}
  Substituting into the bound for $| \fA_{\ET} (\hat p,t;a) |$ above,
  we obtain
  \begin{align*}
    | \fA_{\ET} (\hat p,t;a) | \le \fA^q_{ST} (\hat p;a) t^2 +
    \fA_{\ST}^l (\hat p;a)t
  \end{align*}
  as claimed.

  $\bullet$ $\mathbf{ Term \ V }$.  Using~\eqref{eq:gradient-V}, we have
  \begin{align*}
    & \Big \langle \nabla V(\hat{p}), \Xhb^{a,\hat{p}}(p(t))
    -\Xhb^{a,\hat{p}}(\hat{p}) \Big \rangle
    \\
    & = \langle
    \begin{bmatrix}
      \oneplusmus\nabla f(\hat{x}) +\sqrt{\mu}\hat v+2\mu(\hat{x}-x_*)
      \\
      \hat v+\sqrt{\mu}(\hat{x} - x_*)
    \end{bmatrix}
    ,
    \\
    & \quad
    \begin{bmatrix}
      v(t) - \hat{v}
      \\
      -2\sqrt{\mu}(v(t) - \hat{v})
    \end{bmatrix} \rangle
    \\
    & = \oneplusmus\langle \nabla f(\hat{x}), v(t) -
    \hat{v}\rangle +\sqrt{\mu}\langle \hat{v} , v(t) - \hat{v} \rangle
    \\
    & \quad +2\mu\langle \hat{x} - x_*,v(t) - \hat{v}\rangle
    -2\sqrt{\mu} \langle \hat{v},v(t) - \hat{v} \rangle
    \\
    & \quad -2\mu\langle \hat{x} - x_*,v(t) - \hat{v} \rangle =
    \fD_{\ET} (\hat p,t;a).
  \end{align*}
  Taking the absolute value and using the Cauchy-Schwartz inequality
  in conjunction with \eqref{differencexv}, we obtain the expression
  corresponding to~$\fD_{\ST}$.
  % % 
  % \begin{align*}
  %   (1 + \sqrt{\mu s})t\norm{\nabla f(\hat x)}\norm{2 \sqrt{\mu}\hat v
  %   + (1 + \sqrt{\mu s})\nabla f(\hat x + a \hat v)}
  %   \\
  %   \sqrt{\mu}t\norm{2 \sqrt{\mu}\hat v + (1 + \sqrt{\mu s})\nabla
  %   f(\hat x + a \hat v)}\norm{v}
  % \end{align*}
  
  %%%%%%%%%%% TERM VI
  $\bullet$ $\mathbf{ Term\ VI }$.
  From~\eqref{eq:continuous-hb-lyapunov},
  \begin{align*}
    & V(p(t)) - V(\hat{p}) = \oneplusmus(f(x(t)) - f(x_*)) +
    \frac{1}{4}\norm{v(t)}^2
    \\
    & \quad + \frac{1}{4}\norm{v(t) + 2\sqrt{\mu}(x(t) - x_*)}^2
    \\
    & \quad -\oneplusmus (f(\hat x) - f(x_*)) - \frac{1}{4}\norm{\hat
      v}^2
    \\
    & \quad - \frac{1}{4}\norm{\hat v + 2\sqrt{\mu}(\hat x - x_*)}^2.
  \end{align*}
  Expanding the third summand (using $x(t) = \hat x + (x(t) - \hat x)$
  and $v(t) = \hat v + (v(t) - \hat v)$) as $ \norm{\hat v
    +2\sqrt{\mu}(\hat x - x_*) }^2 + 2\langle \hat v+ 2\sqrt{\mu}(\hat
  x - x_*), v(t) - \hat v + 2\sqrt{\mu}(x(t) - \hat x) \rangle
  +\norm{v(t) - \hat v + 2\sqrt{\mu}(x(t) - \hat x)}^2$, we obtain
  after simplification
  \begin{align}
    & V(p(t)) - V(\hat{p}) = \oneplusmus (f(x(t))-f(\hat{x}))
    \label{sum_of_terms}
    \\
    &\quad + \frac{1}{4} (\norm{v(t)}^2 - \norm{\hat{v}}^2) +
    \frac{1}{4} \norm{v(t) - \hat v + 2\sqrt{\mu} (x(t) - \hat x)}^2
    \nonumber
    \\
    &\quad + \frac{1}{2} \langle v(t) - \hat v + 2 \sqrt{\mu} (x(t) -
    \hat x),\hat v + 2\sqrt{\mu}(\hat x-x_{\ast}) \rangle.  \notag
  \end{align}
  Using~\eqref{eq:hoh-flow}, we have
  \begin{align*}
    & \langle v(t) - \hat v + 2\sqrt{\mu} (x(t) - \hat x), 2
    \sqrt{\mu} (\hat{x}-x_{\ast}) \rangle \nonumber
    \\
    & = -2 \sqrt{\mu}\oneplusmus t \langle \nabla f(\hat{x} + a
    \hat{v}), \hat{x} - x_{\ast} \rangle \nonumber
    \\
    & = -2 \sqrt{\mu} \oneplusmus t \langle \nabla f(\hat{x} + a
    \hat{v}), \hat{x} + a \hat{v}- x_{\ast} \rangle \nonumber
    \\
    & \quad -2 \sqrt{\mu} \oneplusmus t \langle \nabla f(\hat{x} + a
    \hat{v}), - a \hat{v} \rangle \nonumber
    \\
    & \leq -2 \sqrt{\mu} \oneplusmus t \frac{\norm{\nabla f(\hat{x} +
        a \hat{v})}^2}{L}
    \\
    & \quad +2 \sqrt{\mu} \oneplusmus t \langle \nabla f(\hat{x} + a
    \hat{v}), a \hat{v} \rangle ,
  \end{align*}
  where we have used~\eqref{eq:aux-e} to derive the
  inequality. Substituting this bound into~\eqref{sum_of_terms}, we
  obtain $\frac{\sqrt{\mu}}{4}(V(p(t)) - V(\hat{p})) \le \fB_{\ET}
  (\hat p,t;a)$.  To obtain the $\ST$-expressions, we bound each
  remaining term separately as follows. Note that
  \begin{align*}
    & f(x(t)) - f(\hat{x}) \underbrace{\leq}_{\eqref{eq:aux-f}}
    \langle \nabla f(\hat{x}), x(t) - \hat{x}\rangle +
    \frac{L^2}{2\mu}\norm{x(t) - \hat{x}}^2
    \\
    &\leq \norm{x(t) - \hat x} \norm{\nabla f(\hat x)} +
    \frac{L^2}{2\mu}\norm{x(t) - \hat{x}}^2
  % \\
  % &\leq \frac{t}{2\sqrt{\mu}}\norm{2 \sqrt{\mu}\hat v + (1 +
  % \sqrt{\mu
  % s})\nabla f(\hat x + a \hat v) }\norm{\nabla f(\hat x)}
  % \\
  % & \quad + \frac{(1 + \sqrt{\mu s})t}{2\sqrt{\mu}}\norm{\nabla
  % f(\hat
  % x + a \hat v)}\norm{\nabla f(\hat x)} + \frac{L^2}{2\mu}\norm{x(t)
  % - \hat{x}}^2
  % \\
  % &\leq \norm{x(t) - \hat x} \norm{\nabla f(\hat x)} +
  % \frac{L^2}{2\mu}\norm{x(t) - \hat{x}}^2
    \\
    &\leq \frac{t}{2\sqrt{\mu}}\norm{2 \sqrt{\mu}\hat v +
      \oneplusmus\nabla f(\hat x + a \hat v) }\norm{\nabla f(\hat x)}
    \\
    & \quad + \frac{\oneplusmus t}{2\sqrt{\mu}}\norm{\nabla f(\hat x +
      a \hat v)}\norm{\nabla f(\hat x)}
    \\
    &\quad + \frac{L^2}{2\mu}(\frac{t^2}{4\mu}\norm{2 \sqrt{\mu}\hat v
      + \oneplusmus\nabla f(\hat x + a \hat v) }^2
    \\
    &\quad + \frac{\mu_st^2}{4\mu}\norm{\nabla f(\hat x + a
      \hat v)}^2
    \\
    & \quad+ \frac{\oneplusmus t^2}{2\mu}\norm{2 \sqrt{\mu}\hat
      v + \oneplusmus\nabla f(\hat x + a \hat v)} \norm{\nabla f(\hat
      x + a \hat v)}),
  \end{align*}
  where we have used~\eqref{eq:aux-diffx} to obtain the last inequality.
  Next,
  \begin{align*}
    &\norm{v(t)}^2 - \norm{\hat{v}}^2
    % = \norm{v(t) - \hat{v} + \hat{v}}^2 - \norm{\hat{v}}^2
    % \\
    % &= \norm{v(t) - \hat{v}}^2 + 2\langle v(t) - \hat{v},\hat{v} \rangle
    % + \norm{\hat{v}}^2 - \norm{\hat{v}}^2
    = \norm{v(t) - \hat{v}}^2 + 2\langle v(t) - \hat{v},\hat{v}
    \rangle
    \\
    &\leq t^2\norm{2 \sqrt{\mu}\hat v + \oneplusmus\nabla f(\hat x + a
      \hat v)}^2
    \\
    & \quad + 2t\norm{2 \sqrt{\mu}\hat v + \oneplusmus\nabla f(\hat x
      + a \hat v)}\norm{\hat v},
  \end{align*}
  where we have used~\eqref{eq:aux-diffv} to obtain the last
  inequality. Using $\norm{y_1 + y_2}^2\leq 2\norm{y_1}^2 +
  2\norm{y_2}^2 $, we bound
  \begin{align*}
    & \norm{v(t) \!-\! \hat v \!+\! 2\sqrt{\mu} (x(t) \!-\! \hat x)}^2 \leq
    2\norm{v(t) \!-\! \hat v}^2 \!+\! 8\mu\norm{x(t) \!-\! \hat x}^2
    \\
    &\leq 2t^2\norm{2 \sqrt{\mu}\hat v +\oneplusmus \nabla f(\hat x +
      a \hat v)}^2 + 4 \sqrt{\mu}t \cdot
    \\
    &\quad \cdot \big(\norm{2 \sqrt{\mu}\hat v + \oneplusmus\nabla
      f(\hat x + a \hat v) } + \oneplusmus \norm{\nabla f(\hat x + a
      \hat v)} \big)^2,
  \end{align*}
  where we have used~\eqref{differencexv}. Finally, 
  \begin{align*}
    \langle v(t) - \hat v + 2\sqrt{\mu}(x(t) - \hat x), \hat v\rangle
    \leq -\oneplusmus t \langle \nabla f(\hat{x} + a \hat{v}), \hat v
    \rangle .
  \end{align*}
  Employing these bounds in the expression of $\fB_{\ET}$, we obtain $
  | \fB_{\ET} (\hat p,t;a) | \le \fB^q_{ST} (\hat p;a) t^2 +
  \fB_{\ST}^l (\hat p;a)t$, as claimed.
\end{proof}

\begin{IEEEbiography}[{\includegraphics[width=1in,height=1.25in,clip,keepaspectratio]{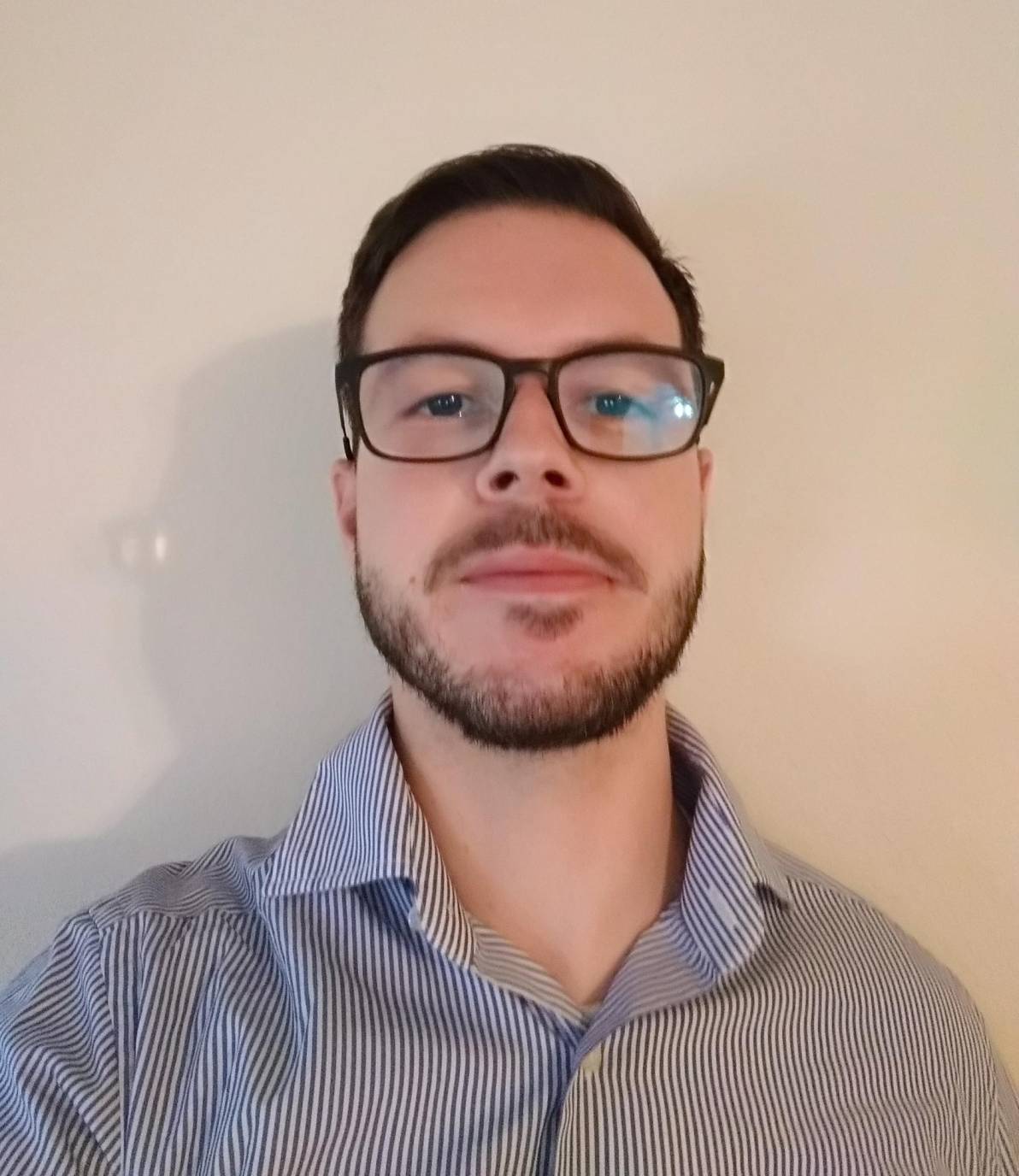}}]{Miguel
    Vaquero}
  was born in Galicia, Spain. He received his Licenciatura and
  Master's degree in mathematics from the Universidad de Santiago de
  Compostela, Spain and the Ph.D. degree in mathematics from Instituto
  de Ciencias Matem\'aticas (ICMAT), Spain in $2015$. He was then a
  postdoctoral scholar working on the ERC project ``Invariant
  Submanifolds in Dynamical Systems and PDE'' also at ICMAT. Since
  October 2017, he has been a postdoctoral scholar at the Department
  of Mechanical and Aerospace Engineering of UC San Diego. He will
  start in $2021$ as an Assistant Professor at the School of Human
  Sciences and Technology of IE University, Madrid, Spain.  His
  interests include optimization, dynamical systems, control theory,
  machine learning, and geometric mechanics.
\end{IEEEbiography}

\vspace*{-1ex}

\begin{IEEEbiography}[{\includegraphics[width=1in,height=1.25in,clip,keepaspectratio]{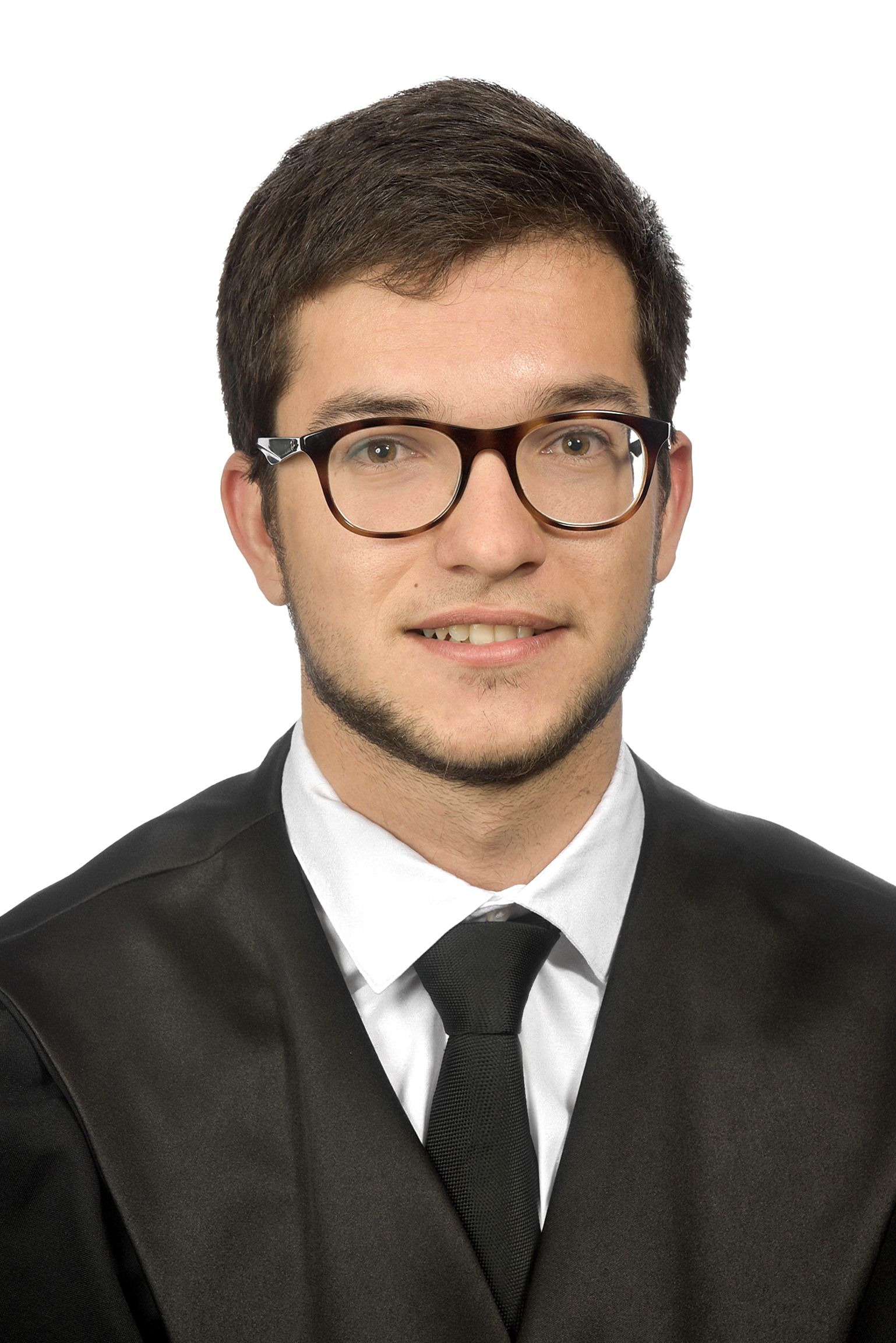}}]{Pol Mestres}
  was born in Catalonia, Spain in 1997. He is currently a student of
  the Bachelor's Degree in Mathematics and the Bachelor's Degree in
  Engineering Physics at Universitat Polit\`{e}cnica de Catalunya,
  Barcelona, Spain. He was a visiting scholar at University of
  California, San Diego, CA, USA, from September 2019 to March 2020.
\end{IEEEbiography}

\vspace*{-1ex}

\begin{IEEEbiography}[{\includegraphics[width=1in,height=1.25in,clip,keepaspectratio]{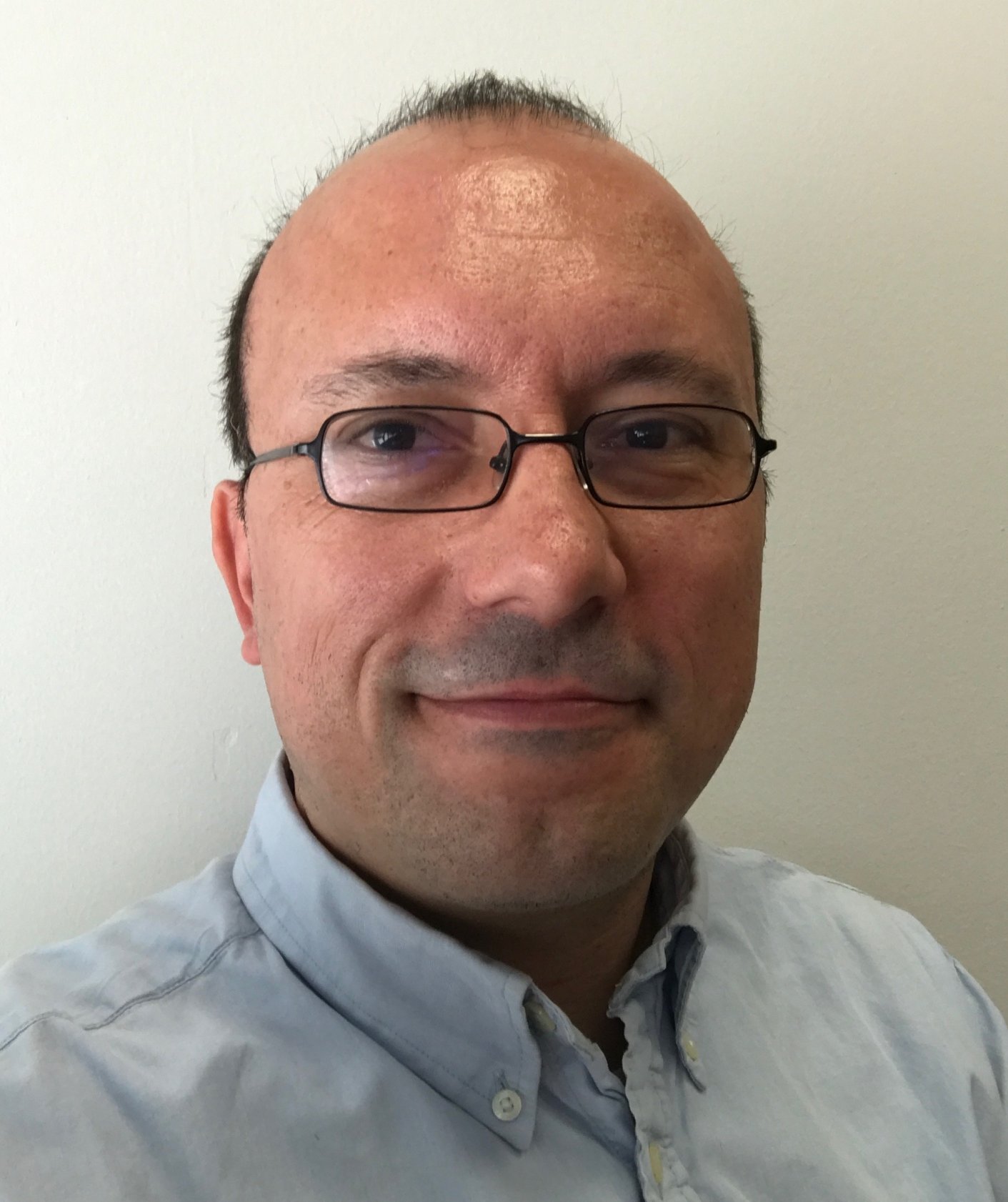}}]{Jorge
    Cort\'{e}s}
  (M'02, SM'06, F'14) received the Licenciatura degree in mathematics
  from Universidad de Zaragoza, Zaragoza, Spain, in 1997, and the
  Ph.D. degree in engineering mathematics from Universidad Carlos III
  de Madrid, Madrid, Spain, in 2001. He held postdoctoral positions
  with the University of Twente, Twente, The Netherlands, and the
  University of Illinois at Urbana-Champaign, Urbana, IL, USA. He was
  an Assistant Professor with the Department of Applied Mathematics
  and Statistics, University of California, Santa Cruz, CA, USA, from
  2004 to 2007. He is currently a Professor in the Department of
  Mechanical and Aerospace Engineering, University of California, San
  Diego, CA, USA. He is the author of Geometric, Control and Numerical
  Aspects of Nonholonomic Systems (Springer-Verlag, 2002) and
  co-author (together with F. Bullo and S.  Mart{\'\i}nez) of
  Distributed Control of Robotic Networks (Princeton University Press,
  2009).  He is a Fellow of IEEE and SIAM. At the IEEE Control Systems
  Society, he has been a Distinguished Lecturer (2010-2014), and is
  currently its Director of Operations and an elected member
  (2018-2020) of its Board of Governors.  His current research
  interests include distributed control and optimization, network
  science, resource-aware control, nonsmooth analysis, reasoning and
  decision making under uncertainty, network neuroscience, and
  multi-agent coordination in robotic, power, and transportation
  networks.
\end{IEEEbiography}

\end{document}